\documentclass{wsbart/wsbart}

\newmuskip\tinymuskip
\usepackage[utf8]{inputenc}
\usepackage[T1]{fontenc}
\usepackage{lmodern}

\usepackage[maxnames=99,backend=biber,isbn=false,uniquename=init]
{biblatex}

\DeclareRedundantLanguages{English}{english}

\usepackage{titling}
\predate{}
\postdate{}

\usepackage{amsmath,amssymb}
\usepackage{amsthm}
\usepackage{bm,dsfont,stmaryrd}
\SetSymbolFont{stmry}{bold}{U}{stmry}{m}{n}
\renewcommand{\mathbb}{\mathds}
\usepackage{mathtools}
\usepackage{accents}
\usepackage{IEEEtrantools}

\newcommand{\R}{\mathbb R}
\newcommand{\T}{\mathbb T}
\newcommand{\X}{\mathcal X}
\newcommand{\x}{\bm{x}}
\newcommand{\vect}[1]{\bm{#1}}

\newcommand{\dd}{\mathop{}\!\mathrm{d}}
\DeclareMathOperator{\Law}{\mathcal{L}}

\DeclareMathOperator{\Expect}{\mathbb{E}}
\DeclareMathOperator{\Var}{Var}
\DeclareMathOperator{\Cov}{Cov}
\DeclareMathOperator{\Osc}{Osc}

\DeclareMathOperator{\Image}{Im}
\newcommand{\1}{\mathbb{1}}

\newcommand{\adjustbin}{\negmedspace{}}

\newcommand{\Jc}{J_\textnormal{c}}
\newcommand{\hc}{h_\textnormal{c}}

\newcommand{\qd}{\textnormal{q}}
\newcommand{\bd}{\textnormal{b}}
\newcommand{\permsum}{\mathop{\overline{\sum\limits_{\sigma}}}}
\newcommand{\ind}{\textnormal{ind}}

\newcommand{\id}{\mathbf{1}}
\newcommand{\TV}{\textnormal{TV}}
\newcommand{\neu}{\textnormal{n}}
\newcommand{\tran}{\textnormal{T}}
\newcommand{\lambdaPN}{\lambda^{\textnormal{P}}_{N}}

\makeatletter
\patchcmd{\@IEEEeqnarray}{\relax}{\relax\intertext@}{}{}
\makeatother

\newtheorem{thm}{Theorem}
\newtheorem*{thm*}{Theorem}
\newtheorem{lem}[thm]{Lemma}
\newtheorem{cor}[thm]{Corollary}
\newtheorem{prop}[thm]{Proposition}
\newtheorem*{prop*}{Proposition}
\theoremstyle{definition}

\newtheorem*{defn*}{Definition}

\newtheorem*{assu*}{Assumption}
\newtheorem{innertextassu}{Assumption}
\theoremstyle{remark}
\newtheorem{rem}{Remark}

\newenvironment{textassu}[1]
  {\renewcommand\theinnertextassu{#1}\innertextassu}
  {\endinnertextassu}

\title{Large-scale concentration and relaxation\\
for mean-field Langevin particle systems}
\author{Songbo Wang}
\date{}
\datefootnote{\today}

\begin{document}

\maketitle

\begin{abstract}
We study the Langevin dynamics of diffusive particles
with regular pairwise interactions under mean-field scaling.
By approximating empirical distributions with conditional distributions,
we establish coercive and contractive properties
for the modulated free energy functional.
These properties yield near-optimal large-scale concentration
and relaxation rates for the particle system throughout the subcritical regime.
Furthermore, we derive generation of chaos estimates with the optimal order
of particle approximation.
As a simpler instance, we demonstrate long-time convergence
of the independent projection of Langevin dynamics.
\end{abstract}

\tableofcontents

\section{Introduction}
\label{sec:intro}

This paper investigates the quantitative long-time convergence
of the following mean-field interacting particle system:
\begin{equation}
\label{eq:ps-sde}
\dd X^i_t = - \nabla V(X^i_t) \dd t
- \frac 1N \sum_{j \in [N]} \nabla_1 W(X^i_t, X^j_t) \dd t
+ \sqrt 2 \dd B^i_t,\qquad i \in [N].
\end{equation}
Here, each particle $X^i_t$ takes values in $\X$, with $\X = \T^d$
or $\R^d$; the potentials $V \colon \X \to \R$ and $W \colon \X \times \X \to
\R$ are regular, with $W$ additionally symmetric;
and $B^i$ are independent standard Brownian motions in $d$ dimensions.
The Langevin process~\eqref{eq:ps-sde} is a classical model for
relaxation dynamics of physical systems with mean-field interactions.
More recently, it has also been used to
model opinion dynamics \cite{GPYConsensus}
and gradient descent algorithms in machine learning \cite{RVETrainability}.
Allowing singularities in $W$, though not considered in this paper,
further relates the system
to random matrix theory \cite{SerfatyCoulombGinzburgLandau}
and biological models of chemotaxis \cite{KellerSegelChemotaxis}.
In all these models, the number of particles $N$ may be so large
that classical ergodic theorems fail to yield satisfactory
long-time convergence rate.
Addressing this high-dimensional regime is precisely the aim of this work.

As $N$ tends to infinity,
the mean-field scaling factor $1/N$ in \eqref{eq:ps-sde} vanishes,
suggesting that particles become independent in this limit.
This leads to the following single-particle self-consistent dynamics,
also known as the \emph{mean-field limit}:
\begin{equation}
\label{eq:mf-sde}
\dd X_t = \nabla V(X_t) \dd t
- \langle \nabla_1 W(X_t,\cdot), \Law(X_t) \rangle \dd t
+ \sqrt 2 \dd B_t.
\end{equation}
Here, $\langle \cdot, \cdot\rangle$ denotes the integration between
a function and a measure,
and $\Law(X_t)$ denotes the law of the random variable $X_t$.
The convergence of \eqref{eq:ps-sde} towards \eqref{eq:mf-sde}
is also termed \emph{propagation of chaos},
a topic of significant interest since its formulation
by \textcite{KacFoundations}.
For precise notions of this convergence and methods to establish it,
see the review articles \cite{ChaintronDiezPOC1,ChaintronDiezPOC2}.

Studying \eqref{eq:mf-sde} instead of \eqref{eq:ps-sde} circumvents
high-dimensional difficulties, but the accuracy of approximating
the $N$-particle system by its mean-field limit relies on
quantitative estimates of the distance between the two stochastic processes.
The propagation of chaos estimate established by \textcite{SznitmanPOC}
indicates that this distance typically grows exponentially in time.
Consequently, reducing the long-time convergence analysis of
\eqref{eq:ps-sde} to that of its mean-field limit \eqref{eq:mf-sde}
requires understanding additional structure in these dynamics.

In our context, the additional structure is geometric.
Denote $m^N_t = \Law(\vect X_t) = \Law(X^1_t,\ldots,X^N_t)$
and $m_t = \Law(X_t)$.
The marginal distributions $m^N_t$ and $m_t$ solve
the Fokker--Planck equations
\begin{align}
\label{eq:ps-fp}
\partial_t m^N_t
&= \sum_{i \in [N]} \Delta_{x^i} m^N_t +
\nabla_{x^i} \cdot \biggl(
\Bigl( \nabla V(x^i) + \frac 1N \sum_{j \in [N]} \nabla_1 W(x^i, x^j) \Bigr)
m^N_t
\biggr), \\
\label{eq:mf-fp}
\partial_t m_t
&= \Delta m_t + \nabla \cdot \biggl(
\Bigl(\nabla V + \int_{\X} \nabla_1 W(\cdot, y)\, m_t(\dd y)\Bigr)
m_t \biggr).
\end{align}
\textcite{OttoGeometryPorous} observed that the two evolutions
$t \mapsto m^N_t$ and $t \mapsto m_t$
are in fact gradient flows in the Wasserstein geometry.
More specifically, under a formal Riemannian metric tensor induced
by optimal transport, the generators of these dynamics are
the negative gradients of the \emph{free energy} functionals
\begin{align*}
\mathcal F^N(\nu^N)
&= H(\nu^N) + N \Expect_{\vect X \sim \nu^N}
\biggl[\langle V, \mu_{\vect X}\rangle
+ \frac 12 \langle W, \mu_{\vect X}^{\otimes 2}\rangle\biggr], \\
\mathcal F(\nu)
&= H(\nu) + \langle V, \nu\rangle + \frac 12 \langle W, \nu^{\otimes 2}\rangle,
\end{align*}
defined for probabilities $\nu^N$ and $\nu$ on $\X^N$ and $\X$, respectively.
Here, $H(\nu^N)$ and $H(\nu)$ represent the entropy of the respective
particle distributions:
\[
H(\nu^N) = \int_{\X^N} \nu^N(\vect x) \log \nu^N(\vect x) \dd \vect x,
\quad
H(\nu) = \int_{\X} \nu(x) \log \nu(x) \dd x,
\]
where we identified the probability measures $\nu^N$ and $\nu$
with their density functions;
the terms $\langle V,\cdot\rangle + \frac 12\langle W,\cdot^{\otimes 2}\rangle$
encode the respective potential energies;
and $\mu_{\vect X}$ denotes the empirical law
for the particle configuration $\vect X = (X^1,\ldots,X^N)$:
\[
\mu_{\vect X} = \frac 1N \sum_{i \in [N]} \delta_{X^i}.
\]

This geometric perspective naturally identifies the free energies
as Lyapunov functionals for the gradient flows they generate.
Differentiating the free energy
functionals along the dynamics yields the \emph{Fisher informations}:
\begin{IEEEeqnarray*}{rCl}
\frac{\dd \mathcal F^N(m^N_t)}{\dd t}
&=& - \sum_{i \in [N]} \int_{\X^N}
{ \lvert \nabla_i \log m^N_t(\vect x)
+ \nabla V(x^i) + \langle \nabla_1 W(x^i, \cdot), \mu_{\vect x}\rangle
\rvert^2\, m^N_t(\dd\vect x)} \\
&\eqqcolon& - \sum_{i \in [N]} \int_{\X^N}
{ \biggl\lvert \nabla_i \log \frac{m^N_t}{m^N_*} (\vect x) \biggr\rvert^2\,
m^N_t(\dd\vect x)}
= - I(m^N_t | m^N_*), \\
\frac{\dd \mathcal F(m_t)}{\dd t}
&=& - \int_{\X} { \lvert \nabla \log m_t(x) + \nabla V(x)
+ \langle \nabla_1 W(x,\cdot), m_t \rangle \rvert^2\,m_t(\dd x) } \\
&\eqqcolon& - \int_{\X} { \biggl\lvert \nabla \log \frac{m_t}{\Pi[m_t]} (x)
\biggr\rvert^2\, m_t(\dd x) }
= - I(m_t | \Pi[m_t]),
\end{IEEEeqnarray*}
where $m^N_*$ and $\Pi[\nu]$ are probabilities defined on $\X^N$ and $\X$,
respectively, by
\begin{align}
m^N_*(\dd \vect x) &\propto
\exp \biggl( - N \langle V, \mu_{\vect x} \rangle - \frac{N}{2}
\langle W, \mu_{\vect x}^{\otimes 2}\rangle \biggr) \dd \vect x, \notag \\
\Pi[\nu](\dd x) &\propto
\exp \bigl( - V(x) - \langle W(x,\cdot), \nu\rangle \bigr) \dd x,
\label{eq:def-Pi}
\end{align}
with proportionality constants ensuring normalization.
We term the images of $\Pi$ \emph{local equilibrium} measures:
analogous to kinetic theory, the free energy dissipation indicates that
$m_t$ should converge rapidly towards $\Pi[m_t]$, while the convergence
of the latter towards global equilibrium remains more challenging.
Nevertheless, it follows, at least formally, that
$m^N_*$ is the unique invariant measure for~\eqref{eq:ps-fp}
and that any invariant measure $m_*$ for~\eqref{eq:mf-fp}, if it exists,
must satisfy the fixed-point relation
\begin{equation}
\label{eq:Pi-self-consistency}
m_* = \Pi[m_*].
\end{equation}
This relation corresponds to the \emph{self-consistency} equation
for spin models.
Analyzing the relationship between free energies and Fisher informations
provides therefore a way of understanding the long-time convergence
of these diffusive evolutions.

One research direction aims to establish a linear dominance of the Fisher
information over the free energy. Specifically, the goal is to prove that,
for all probability measures $\nu^N$ on $\X^N$,
\begin{equation}
\label{eq:ps-lsi}
I(\nu^N | m^N_*) \geqslant 2 \lambda_N H(\nu^N | m^N_*),
\end{equation}
where $\lambda_N > 0$ and $H$ is the \emph{relative entropy}
functional defined by
\[
H(\nu^N | m^N_*) = \int_{\X^N}
\log \frac{\nu^N(\vect x)}{m^N_*(\vect x)} \nu^N(\dd \vect x).
\]
This is referred to as a $\lambda_N$-logarithmic Sobolev (log-Sobolev)
inequality for $m^N_*$.
Since the free energy differs from the relative entropy only
by an additive constant:
\[
\mathcal F^N(\nu^N) - \mathcal F^N(m^N_*) = H(\nu^N | m^N_*),
\]
the log-Sobolev inequality yields directly the exponential entropy contraction
for the $N$-particle dynamics \eqref{eq:ps-fp}:
\[
\mathcal F^N(m^N_t) - \mathcal F^N(m^N_*)
= H(m^N_t | m^N_*) \leqslant e^{-2\lambda_N t} H(m^N_0 | m^N_*).
\]
The main challenge is to prove the log-Sobolev inequality \eqref{eq:ps-lsi}
with a constant $\lambda_N$ that is uniform in $N$,
as standard approaches typically suffer from the \emph{curse of dimensionality}
and yield $\lambda_N$ that decays exponentially with $N$.

Convexity is a prototypical scenario in which the curse of dimensionality can be
circumvented, as is well known in optimization theory. For convex potentials $V$
and $W$ satisfying the lower bound $\nabla^2 V \succcurlyeq \lambda > 0$,
\textcite{MalrieuLSI} and \textcite{CMVKinetic} demonstrated that $m^N_*$
satisfies a $\lambda$-log-Sobolev inequality, and that analogous entropic
contraction holds for the mean-field limit \eqref{eq:mf-fp}. This scenario is
commonly termed the \emph{displacement-convex} case, since
the convexity of potentials ensures that along optimal
transport trajectories, the free energy functionals $\mathcal F^N$
and $\mathcal F$ evolve convexly.
This perspective naturally generalizes finite-dimensional gradient flows
with convex objective functions to the Wasserstein setting.

Despite its elegance, the displacement convexity framework does not encompass
physical spin models, where the confinement potentials typically exhibit strong
concavity. \textcite{BauerschmidtBodineauSimple} address this limitation by
considering the following scale decomposition in $m^N_*$:
\[
m^N_* = \Expect_{h} [ (T_h m_*)^{\otimes N} ],
\]
where $h$ is a random variable dual to the collective spin orientation,
and $T_h m_*$ describes microscopic fluctuations of a single spin
conditioned on $h$.
The tensorized structure of $(T_h m_*)^{\otimes N}$ reduces the analysis
of the log-Sobolev inequality for $m^N_*$ to that
of the \emph{renormalized potential} for the macroscopic variable $h$,
determined by integrating out the microscopic interactions.
This approach yields a concise proof of the log-Sobolev inequality
for $\mathrm{O}(n)$ models, which, in the mean-field case,
remains valid up to the critical point.
For interactions exhibiting geometric structure,
gradually decomposing scales in accordance with this structure
leads to a continuous flow of renormalized potentials,
corresponding to the \emph{exact renormalization group} of Lagrangians
introduced by \textcite{PolchinskiRenormalization}.
Leveraging this decomposition, Bodineau, Bauerschmidt and Dagallier
developed a multi-scale Bakry--Émery criterion
that enables proofs of log-Sobolev inequalities
for various near-critical Euclidean quantum field theories
\cite{BauerschmidtBodineauSineGordon,BauerschmidtDagallierPhi4}
and spin models \cite{BauerschmidtDagallierIsing};
see their review \cite{BBDPolchinski}
for a systematic exposition of this methodology.
A recent work \cite{BBDCriterionFreeEnergy}
revisits the single-scale mean-field setting \eqref{eq:ps-fp}
and establishes a Polyak--\L ojasiewicz criterion on the renormalized potential,
in contrast to the earlier Bakry--Émery approach,
for the uniform log-Sobolev inequality for $m^N_*$.
\textcite{ChenEldanLocalizationSchemes} focus on the evolution of
microscopic fluctuations rather than renormalized potentials, and view
this scale decomposition as a \emph{localizing} stochastic process.
This approach has been instrumental in addressing the
Kannan--Lovász--Simonovits conjecture \cite{EldanThinShell,ChenAlmostKLS}.

While the preceding approaches act directly on the $N$-particle level,
the log-Sobolev inequality \eqref{eq:ps-lsi} can also
be approached through the following inequality on the mean-field limit:
\begin{equation}
\label{eq:mf-pl}
I(\nu | \Pi[\nu])
\geqslant 2\lambda \mathcal F(\nu | m_*)
\coloneqq 2\lambda \bigl( \mathcal F(\nu) - \mathcal F(m_*) \bigr).
\end{equation}
This inequality serves as the counterpart to \eqref{eq:ps-lsi},
as it yields the free energy contraction along the limit flow \eqref{eq:mf-fp}:
\[
\mathcal F(m_t | m_*) \leqslant e^{-2\lambda t}\mathcal F(m_0 | m_*).
\]
Indeed, \textcite{DGPSPhase} showed that under general assumptions,
the $N$-particle contraction rate $\lambda_N$ remains asymptotically
upper bounded by the optimal limit contraction rate $\lambda$.
This readily provides a criterion for the absence of uniform log-Sobolev
inequality for $m^N_*$ in cases where \eqref{eq:mf-pl} fails for all $\lambda$,
as typically occurs in the supercritical regime of models with phase transition.
The converse direction, establishing a lower bound on the $N$-particle
contraction rate $\lambda_N$ using $\lambda$ from \eqref{eq:mf-pl},
proves considerably more challenging.
Nevertheless, the compactness arguments in \cite{DGPSPhase} showed that
the \emph{large-scale} entropic convergence rate
of the $N$-particle system converges towards
$\lambda$ in \eqref{eq:mf-pl} as $N$ tends to infinity.
Such large-scale convergence is weaker than
that given by \eqref{eq:ps-lsi}, as it only addresses the
dissipation of $H(m^N_t | m^N_*)$ when this relative entropy is of order $N$,
while losing all control over the dynamics when it becomes $o(N)$
on longer time scales.
These two findings led them to conjecture that
the optimal log-Sobolev constants $\lambda_N$ for \eqref{eq:ps-lsi}
converge to the optimal constant $\lambda$ for \eqref{eq:mf-pl}
as $N \to \infty$.

An earlier collaboration with F.~Chen and Ren \cite{ulpoc}
\emph{quantitatively} pursued the investigation of large-scale relaxation.
The analysis focused on \emph{flat-convex} potentials $W$,
a class where the limit contraction \eqref{eq:mf-pl}
was independently proven by \textcite{NWSConvexMFL}
and \textcite{ChizatMFL},
and established the functional inequality
\begin{equation}
\label{eq:ps-lsi-large-scale}
I(\nu^N | m^N_*) \geqslant
2\lambda_N \mathcal F^N(\nu^N | m_*) - \Delta_I
\coloneqq 2\lambda_N \bigl( \mathcal F^N(\nu^N) - N\mathcal F(m_*) \bigr)
- \Delta_I.
\end{equation}
Here, $\mathcal F^N(\cdot|\cdot)$ denotes the \emph{modulated free energy}
of \textcite{BJWPKSCompteRendu,BJWMFE,BJWAttractive},
which quantifies the free-energetic deviation of $\nu^N$ from $m_*^{\otimes N}$;
the constants $\lambda_N$ and $\Delta_I$ are explicit with
$\liminf_{N\to\infty}\lambda_N > 0$ and $\Delta_I = O(1)$.
Our result~\eqref{eq:ps-lsi-large-scale} provides non-asymptotic
constants for the \emph{regularized} log-Sobolev inequality in~\cite{DGPSPhase}
and additionally yields entropy contraction up to $O(1)$:
\begin{equation}
\label{eq:ps-entropy-coarse-grained}
H(m^N_t | m_*^{\otimes N})
\leqslant \mathcal F^N(m^N_t | m_*)
\leqslant e^{-2\lambda_N t} \mathcal F^N(m^N_0 | m_*)
+ \frac{\Delta_I}{2\lambda_N}.
\end{equation}
A subsequent work by the author \cite{nulsi} further identified
\eqref{eq:ps-lsi-large-scale} as a \emph{defective}
log-Sobolev inequality for $m^N_*$.
When combined with a Poincaré inequality, this yields
the log-Sobolev inequality \eqref{eq:ps-lsi} with a different constant.

The present study establishes the large-scale contractivity
\eqref{eq:ps-lsi-large-scale} using only the limit
contractivity~\eqref{eq:mf-pl} rather than the flat convexity of $W$.
Directly applying the proof strategy of~\cite{ulpoc} fails,
as the key step in that argument invokes the flat convexity condition
for empirical measures, specifically,
\[
\langle W, (\mu_{\vect X} - m_*)^{\otimes 2}\rangle \geqslant 0.
\]
The analogous approach under the condition~\eqref{eq:mf-pl}
would require evaluating it at $\nu = \mu_{\vect X}$,
which is meaningless,
since both entropy and Fisher information diverge for atomic measures.
The difficulty appears fundamental: the Dean--Kawasaki dynamics governing
the evolution of $\mu_{\vect X}$ along \eqref{eq:ps-sde} admits
only atomic solutions, as shown by \textcite{KLvRDeanKawasaki}.

We overcome this difficulty through the \emph{conditional approximation}
\[
\bar \nu = \frac 1N \sum_{k \in [N]} \nu_k,
\]
where $\nu_k$ is the conditional law of $X^k$ given
$\vect X^{[k-1]} = (X^1,\ldots,X^{k-1})$, that is,
$\nu_k = \Law(X^k | \vect X^{[k-1]})$.
By construction, $\bar\nu$ closely approximates $\mu_{\vect X}$.
Indeed, their difference decomposes into a sum of $N$ martingale increments:
\[
\mu_{\vect X} - \bar\nu
= \frac 1N \sum_{k \in [N]}
\delta_{X^k} - \nu_k
= \frac 1N \sum_{k \in [N]}
\delta_{X^k} - \Expect[ \delta_{X^k} | \vect X^{[k-1]} ],
\]
where each increment is $O(1/N)$.
By orthogonality of increments,
$\mu_{\vect X} - \bar\nu$ is typically $O(1/\sqrt N)$.
Moreover, entropy's chain rule ensures that each $\nu_k$
has finite entropy almost surely.
Therefore \eqref{eq:mf-pl} remains meaningful
when evaluated at $\nu = \nu_k$.

Our proof strategy thus consists of alternating between $\mu_{\vect X}$ and
$\bar\nu$ and applying the limit contractivity~\eqref{eq:mf-pl} to
conditional laws, through which we ultimately establish the large-scale
functional inequality~\eqref{eq:ps-lsi-large-scale} for symmetric
$\nu^N$, with $\lambda_N$ arbitrarily close to $\lambda$
in~\eqref{eq:mf-pl} and $\Delta_I = O(1)$.
A non-symmetric version with different constants is also established
under a free energy condition stronger than \eqref{eq:mf-pl}.
These large-scale log-Sobolev inequalities constitute
a main contribution of this paper and are referred to
as the \emph{contractivity} of the modulated free energy
along Langevin dynamics.

Beyond contractivity, the conditional approximation technique
also enables proof of the entropic \emph{coercivity}
of the modulated free energy
\begin{align*}
\mathcal F^N(\nu^N | m_*) &\geqslant \delta_N H(\nu^N | m_*^{\otimes N})
- \Delta_{\mathcal F}
\intertext{by leveraging its counterpart in the mean-field limit}
\mathcal F(\nu | m_*) &\geqslant \delta H(\nu | m_*).
\end{align*}
This constitutes another principal result of this work.
Combined with contractivity, coercivity provides an entropy upper bound
analogous to~\eqref{eq:ps-entropy-coarse-grained} for
the dynamics~\eqref{eq:ps-fp}.
Coercivity alone also yields Gaussian concentration for $m^N_*$
and a quantitative large deviation upper bound of the Jabin--Z.~Wang type
\cite{JabinWang}.
In addition, coercivity permits studying near-critical behaviors
of a first-order phase transition, while contractivity typically fails,
as we will demonstrate for the supercritical Curie--Weiss model
with a small external field.

As an application of the contractivity result,
we analyze the dissipation of the dynamically modulated free energy
for the $N$-particle flow \eqref{eq:ps-fp},
which leads to generation of chaos estimates
with the optimal rate as $N\to\infty$.

Finally, we examine the independent projection of Langevin dynamics introduced
by \textcite{LackerIndependentProjections}. As this dynamics enforces particle
independence, analyzing its long-time behavior requires no conditional
approximation and is therefore especially straightforward.

\pagebreak[0]

\subsubsection*{Organization of paper}

The rest of the paper is organized as follows.
Section~\ref{sec:mr} presents our main results:
\begin{enumerate}[(i)]
\item coercivity of the modulated free energy;
\item contractivity of the modulated free energy along Langevin dynamics;
\item generation of chaos estimate;
\item long-time convergence of the independent projection of Langevin dynamics.
\end{enumerate}
Section~\ref{sec:exm} presents
the mean-field XY and the double-well Curie--Weiss models
to illustrate both the scope and limitations of the main results.
Section~\ref{sec:approx} introduces the conditional approximation
technique and establishes auxiliary properties.
Sections~\ref{sec:coercivity}--\ref{sec:ip} prove results~(i)--(iv),
respectively.
Appendix~\ref{app:recover-mf} collects lemmas
used to demonstrate the optimality of results~(i) and (ii).
Appendix~\ref{app:mode-decomposition} provides criteria
for verifying our main assumptions in models with mode decomposition.
Appendices~\ref{app:xy} and \ref{app:cw} establish additional properties
for the models discussed in Section~\ref{sec:exm}.

\section{Main results}
\label{sec:mr}

This section presents the main results of this paper.
Section~\ref{sec:mr-reg-fi} sets out the regularity
and the functional inequality framework for the analysis.
Sections~\ref{sec:mr-coercivity}--\ref{sec:mr-ip} details the results
corresponding to~(i)--(iv) outlined at the end of introduction.
Finally, Section~\ref{sec:mr-future-directions} discusses future directions.

\subsection{Regularity and functional inequalities}
\label{sec:mr-reg-fi}

We begin by introducing the necessary definitions
and assumptions that will serve as the foundation for our analysis.
These assumptions address the regularity properties
of the kernel $W$, as well as functional inequalities
for the equilibrium measure $m_*$ and the images of $\Pi$.

Throughout this paper, we work with kernel functions
of quadratic growth and probability measures with finite second moments.
These growth and moment conditions become void when $\X = \T^d$.
Let $U \colon \X \times \X \to \R$ be a kernel of quadratic growth
and let $\mathcal P_2(\X)$ denote the space of probability measures
on $\X$ with finite second moments.

\begin{defn*}
We say $U$ is \emph{positive}
if for all $\nu$, $\nu' \in \mathcal P_2(\X)$,
\[
\langle U, (\nu - \nu')^{\otimes 2}\rangle \geqslant 0.
\]
We say $U$ is \emph{negative} if $-U$ is positive.
\end{defn*}

\begin{defn*}
We define the \emph{double oscillation} of $U$ as
\[
\sup_{x, y, z, w \in \X}
\langle U, (\delta_x - \delta_y) \otimes (\delta_z - \delta_w) \rangle
\]
and denote it by $\Osc_2 U$.
\end{defn*}

The following two assumptions address the regularity properties
of the interaction kernel $W$ and its gradient $\nabla_1W$.

\begin{textassu}{W}
\label{assu:W}
The interaction potential admits
\[
W = W^+ - W^-,
\]
where both $W^+$ and $W^-$ are positive symmetric kernels of quadratic growth.
\begin{itemize}[itemsep=\topsepamount]
\item In the case $\X = \T^d$, the kernels $W^+$ and $W^-$ satisfy
\[
\Osc_2 W^+ + \Osc_2 W^- \leqslant 4M_W
\]
for some $M_W \geqslant 0$; by convention, $L_W = 0$.
\item In the case $\X = \R^d$, each kernel further decomposes as
\[
W^s = W_{\bd}^s + W_{\qd}^s, \quad \text{$s = +$, $-$,}
\]
where $W^s_{\bd}$, $W^s_{\qd}$ are symmetric and satisfy
\begin{align*}
\Osc_2 W_{\bd}^+
+ \Osc_2 W_{\bd}^- &\leqslant 4M_W, \\
\lVert \nabla_{1,2}^2 W_{\qd}^+ \rVert_{L^\infty}
+ \lVert \nabla_{1,2}^2 W_{\qd}^- \rVert_{L^\infty} &\leqslant L_W
\end{align*}
for some $M_W$, $L_W \geqslant 0$.
\end{itemize}
\end{textassu}

\begin{textassu}{R}
\label{assu:R}
There exists $R \colon \X \times \X \to \R$ of quadratic growth such that
for all $y \in \X$ and all measures $\nu$, $\nu' \in \mathcal P_2(\X)$,
\[
\lvert \langle \nabla_1 W(y,\cdot), \nu - \nu'\rangle \rvert^2
\leqslant \langle R, (\nu-\nu')^{\otimes 2}\rangle.
\]
\begin{itemize}[itemsep=\topsepamount]
\item In the case $\X = \T^d$, the kernel $R$ satisfies
\[
\Osc_2 R \leqslant 4M_R
\]
for some $M_R \geqslant 0$; by convention, $L_R = 0$.
\item In the case $\X = \R^d$, the kernel $R$ further decomposes as
\[
R = R_{\bd} + R_{\qd},
\]
where $R_{\bd}$, $R_{\qd}$ are symmetric and satisfy
\[
\Osc_2 R_{\bd} \leqslant 4M_R,
\quad
\lVert \nabla_{1,2}^2 R_{\qd} \rVert_{L^\infty} \leqslant L_R,
\]
for some $M_R$, $L_R \geqslant 0$.
\end{itemize}
\end{textassu}

The kernel $R$ in Assumption~\ref{assu:R} is necessarily positive.

In the following, we refer to the ``b'' kernels
in Assumptions~\ref{assu:W} and \ref{assu:R}
as bounded components,
and the ``q'' kernels as quadratic components.

\begin{rem}
\label{rem:bbd-spectral}
\textcite{BBDCriterionFreeEnergy} consider interaction potentials
admitting the following \emph{mode decomposition}:
\[
W(x,y) = \sum_{a} w_a r_a(x) r_a(y),
\]
where $w_a \in \R$ and $r_a \colon \X \to \R$.
Setting
\[
W^s(x,y) = \sum_{a} \1_{s w_a > 0} sw_a r_a(x) r_a(y),
\qquad s = +,\, -,
\]
yields a decomposition $W = W^+ - W^-$ into positive symmetric kernels.
The following conditions then suffice for
Assumptions~\ref{assu:W} and \ref{assu:R}.
\begin{itemize}[itemsep=\topsepamount]
\item In the case $\X = \T^d$, each $r_a$ has bounded oscillation
and gradient, with
$\sum_a |w_a| (\Osc r_a)^2 < \infty$
and $\sum_a |w_a| \lVert\nabla r_a\rVert_{L^\infty}^2 < \infty$.

\item In the case $\X = \R^d$, each $r_a$ decomposes as
$r_a = r_{a,\bd} + r_{a,\qd}$, where $r_{a,\bd}$ has bounded oscillation
and gradient and $r_{a,\qd}$ has bounded gradient,
with analogous summability over the modes.
\end{itemize}
This framework covers a broad class of
interactions via Fourier modes on $\T^d$ and quadratic interactions on $\R^d$,
of which the mean-field XY and double-well Curie--Weiss models
of Section~\ref{sec:exm} are single-mode instances.
\end{rem}

In many instances it is convenient to introduce the following notion,
which will be used repeatedly in the sequel.

\begin{defn*}
Let $U \colon \X \times \X \to \R$ be of quadratic growth
and let $m \in \mathcal P_2(\X)$.
We define a new kernel $U_m$ as
\[
U_m(x,y) = U(x,y) - \langle U(x,\cdot), m\rangle
- \langle U(\cdot,y), m\rangle
+ \langle U, m^{\otimes 2}\rangle.
\]
and call it the \emph{reduced} kernel associated to $m$.
\end{defn*}

The reduction operation simply recenters the associated bilinear form
$(\nu,\nu')\mapsto\langle U,\nu\otimes\nu'\rangle$ around $m$:
\begin{equation}
\label{eq:dual-cumulant}
\forall \nu, \nu' \in \mathcal P_2(\X),\qquad
\langle U, (\nu - m) \otimes (\nu' - m)\rangle
= \langle U_m, \nu \otimes \nu' \rangle.
\end{equation}
If $m$ is the mean-field equilibrium measure $m_*$,
we write simply $U_*$ for $U_{m_*}$.
Specifically, for the kernels $W_{c}^s$, $R_{c}^s$,
with $c = \bd$, $\qd$ and $s = +$, $-$,
we denote their reductions associated to $m_*$ by $W_{c,*}^s$, $R_{c,*}$,
respectively.
The reduction operation preserves both the double oscillation
and the mixed derivative:
\[
\Osc_2 U_m = \Osc_2 U,\quad
\nabla_{1,2}^2 U_m = \nabla_{1,2} U.
\]
Using this reduction and the invariance
condition~\eqref{eq:Pi-self-consistency}, we rewrite the mapping $\Pi$
and the functionals $\mathcal{F}(\cdot|m_*)$, $\mathcal{F}^N(\cdot|m_*)$
from~\eqref{eq:def-Pi}, \eqref{eq:mf-pl},
and~\eqref{eq:ps-lsi-large-scale}:
\begin{IEEEeqnarray}{rCl}
\Pi[\nu](\dd x) &\propto& \exp \bigl( - \langle W_*(x,\cdot),\nu \rangle\bigr)
\,m_*(\dd x), \label{eq:def-Pi-centered} \\
\mathcal F(\nu | m_*)
&=& H(\nu | m_*) + \frac 12 \langle W_*, \nu^{\otimes 2}\rangle,
\label{eq:def-relative-F-centered} \\
\mathcal F^N(\nu^N | m_*)
&=& H(\nu^N | m_*^{\otimes N}) + \frac N2
\Expect_{\vect X \sim \nu^N} [ \langle W_*, \mu_{\vect X}^{\otimes 2}\rangle].
\label{eq:def-modulated-F-centered}
\end{IEEEeqnarray}
These functionals can be understood as the Bregman divergence
of $\mathcal F$ and $\mathcal F^N$ with respect to $m_*$ and $m_*^{\otimes N}$
in an appropriate sense, which justifies their two-argument notation.
In the following, we work exclusively with these functionals
to simplify the analysis.

We introduce the key functional inequalities in our analysis.
Let $m \in \mathcal P_2(\X)$ and $\kappa>0$.
We say that $m$ satisfies a \emph{$\kappa$-Talagrand inequality}
if for all $\nu \in \mathcal P_2(\X)$,
\[
H(\nu | m) \geqslant \frac \kappa2 W_2^2(\nu, m);
\]
and that $m$ satisfies a \emph{$\kappa$-log-Sobolev inequality}
if for all $\nu \in \mathcal P_2(\X)$,
\[
I(\nu | m) \geqslant 2\kappa H(\nu | m).
\]
The following assumptions state the required functional inequalities
for $m_*$ and the images of $\Pi$.
Let $\rho \geqslant \rho_{\Pi} > 0$.

\begin{textassu}{T}
\label{assu:T}
The measure $m_*$ satisfies
a $\rho$-Talagrand inequality.
\end{textassu}

\begin{textassu}{LS}
\label{assu:LS}
The measure $m_*$ satisfies
a $\rho$-log-Sobolev inequality.
\end{textassu}

\begin{textassu}{LS$_\Pi$}
\label{assu:LS-Pi}
The images of $\Pi$ satisfies a $\rho_{\Pi}$-log-Sobolev inequality uniformly.
\end{textassu}

It is well known that Assumption~\ref{assu:LS} implies \ref{assu:T}.
Moreover, by \eqref{eq:Pi-self-consistency},
Assumption~\ref{assu:LS-Pi} implies \ref{assu:LS} with $\rho = \rho_{\Pi}$.
However, perturbation methods are typically used to establish
Assumption~\ref{assu:LS-Pi} and yield much weaker constants than those
in Assumption~\ref{assu:LS}; see \cite[Section~3.3]{ulpoc} and
\cite[Proposition~5]{socgibbs}.
We therefore maintain distinct notation for the constants in these assumptions.

The constant $\rho$ in Assumptions~\ref{assu:T} and \ref{assu:LS}
provides a natural spatial scale.
For convenience we introduce the dimensionless variables
\[
\ell_W = \frac{L_W}{\rho},\quad
\mu_R = \frac{M_R}{\rho},\quad
\ell_R = \frac{L_R}{\rho^2}.
\]

\begin{rem}
The assumptions in this section do not guarantee uniform long-time convergence
for the Langevin dynamics \eqref{eq:ps-fp}. However, if we further
assume $M_W + \ell_W + \mu_R + \ell_R \ll 1$, then the $N$-particle dynamics
can be viewed as a perturbation of a non-interacting dynamics,
which leads to uniform convergence.
Specifically, using the criterion of \textcite{ZegarlinskiDobrushinUniqueness},
\textcite{GLWZUPLSI} established a uniform log-Sobolev inequality
\eqref{eq:ps-lsi} for $m^N_*$ in a similar perturbative regime.
\end{rem}

\subsection{Coercivity}
\label{sec:mr-coercivity}

We establish the coercivity of the modulated free energy
as the first main result.
Our analysis relies on the following coercivity condition
for the free energy $\mathcal F(\cdot | m_*)$ of the mean-field limit:
\begin{equation}
\label{eq:mf-coercivity}
\forall \nu \in \mathcal P_2(\X),\qquad
\mathcal F(\nu | m_*)
\geqslant \delta H(\nu | m_*).
\tag{Coer}
\end{equation}
This condition can be interpreted as the positivity
of the free energy functional at the lower temperature $1-\delta$.
Our result regarding the coercivity of $\mathcal F^N(\cdot|m_*)$
is stated as follows.

\begin{thm}
\label{thm:coercivity}
Let Assumptions~\ref{assu:W} and \ref{assu:T} hold.
Let \eqref{eq:mf-coercivity} hold with $\delta \in (0,1]$.
Then for
all $\varepsilon \in (0,1]$ and all $\nu^N \in \mathcal P_2(\X)$,
\[
\mathcal F^N(\nu^N | m_*)
\geqslant \delta_N H(\nu^N | m_*^{\otimes N}) - \Delta_{\mathcal F},
\]
where $\delta_N$ and $\Delta_{\mathcal F}$ are given by
\begin{align*}
\delta_N &= (1-\varepsilon) \delta
- \biggl(1+\frac{M_W+\ell_W}{\varepsilon\delta}\biggr)\frac{2\ell_W}{N}, \\
\Delta_{\mathcal F} &=
2\biggl(1+\frac{M_W+\ell_W}{\varepsilon\delta}\biggr)
(M_W + \ell_W d).
\end{align*}
\end{thm}

Theorem~\ref{thm:coercivity} follows
by direct application of the conditional approximation technique
and is proved in Section~\ref{sec:coercivity}.\footnote{%
After the posting of the preprint, it was pointed out to the author that
this proof resembles the Dupuis--Ellis proof of the large deviation
upper bound \cite[Section~2.5]{DupuisEllisWeakConvergence}.}

The constant $\delta_N$ can be made arbitrarily close to $\delta$
and is therefore almost optimal.
Indeed, fix $\nu \in \mathcal{P}_2(\X)$ with $\mathcal{F}(\nu|m_*) < \infty$
and set $\nu^N = \nu^{\otimes N}$.
By Theorem~\ref{thm:coercivity},
\[
\frac{\mathcal{F}^N(\nu^{\otimes N} | m_*)}{N}
\geqslant \delta_N H(\nu | m_*) - \frac{\Delta_{\mathcal{F}}}{N}.
\]
Meanwhile, Lemma~\ref{lem:recover-free-energy}
in Appendix~\ref{app:recover-mf} yields
\[
\lim_{N \to \infty} \frac{\mathcal{F}^N(\nu^{\otimes N} | m_*)}{N}
= \mathcal{F}(\nu | m_*).
\]
Letting $N \to \infty$ and using the arbitrariness of $\varepsilon$
recovers \eqref{eq:mf-coercivity}.
In other words, if $\delta$ is the optimal constant
for \eqref{eq:mf-coercivity},
the coercivity constant $\delta_N$ in Theorem~\ref{thm:coercivity} must satisfy
$\limsup_{N \to \infty} \delta_N \leqslant \delta$.

\begin{rem}
The coercivity constant $\delta_N$ in Theorem~\ref{thm:coercivity}
undergoes a loss of order $O(1/N)$ only when $\ell_W \neq 0$.
The contractivity constants
in Theorems~\ref{thm:contractivity-1} and \ref{thm:contractivity-2} below
exhibit similar dependency.
\end{rem}

\begin{rem}
This result can be compared to that of a previous work
\cite[Theorem~3]{socgibbs}.
The proof in the present approach is considerably simpler
thanks to the conditional approximation technique.
In addition, the condition \eqref{eq:mf-coercivity} naturally incorporates
the positive component of the interaction energy
$\langle W_*^+, \cdot^{\otimes 2}\rangle$.
However, Theorem~\ref{thm:coercivity} inevitably suffers
a multiplicative loss of concentration,
since letting $\varepsilon \to 0$ causes $\Delta$ to diverge.
In contrast, \cite[Theorem~3]{socgibbs}, when reformulated on the torus,
admits the case corresponding to $\varepsilon=\delta=0$.
\end{rem}

As observed in a previous work~\cite{socgibbs}, Theorem~\ref{thm:coercivity}
implies readily a lower bound for the entropy functional $H(\cdot | m^N_*)$,
which in turn implies a Gaussian concentration estimate for $m^N_*$.
We state this result precisely below; its proof
is provided in Section~\ref{sec:coercivity}.

\begin{cor}
\label{cor:entropy-coercivity}
Under the assumptions and notations of Theorem~\ref{thm:coercivity},
for all $\nu^N \in \mathcal P_2(\X)$,
\begin{align*}
H(\nu^N | m^N_*) &\geqslant \delta_N H(\nu^N | m_*^{\otimes N})
- \Delta_{\mathcal F} - 2M_W - 2\ell_W d.
\intertext{Consequently, for all $\nu^N \in \mathcal P_1(\X)$,}
H(\nu^N | m^N_*) &\geqslant \frac{\delta_N\rho}
{64(1+\Delta_{\mathcal F}+2M_W+2\ell_Wd)^2} W_1^2(\nu^N, m^N_*).
\end{align*}
\end{cor}

By expanding the definition of $\mathcal F^N(\nu^N | m_*)$,
the conclusion of Theorem~\ref{thm:coercivity} can be rewritten as
\[
- \frac{N}{2} \Expect_{\vect X \sim \nu^N}
[ \langle W_*, \mu_{\vect X}^{\otimes 2}\rangle ]
\leqslant (1-\delta_N) H(\nu^N | m_*^{\otimes N}) + \Delta_{\mathcal F}.
\]
This inequality mirrors the large deviation estimate employed to establish
propagation of chaos in the Jabin--Z.~Wang method
\cite[Theorems~1 and 2]{JabinWang}, though their assumptions
differ from \eqref{eq:mf-coercivity}.
This perspective has also been investigated in an earlier work
\cite[Corollary~4]{socgibbs}.
For further use, we state the following result in the bounded case.

\begin{cor}
\label{cor:jw}
Let $U \colon \X \times \X \to \R$
be positive with $\Osc_2 U \leqslant 4$.
Then for all $\delta \in (0,1)$, $m \in \mathcal P(\X)$
and $\nu^N \in \mathcal P(\X^N)$,
\[
N\Expect_{\vect X \sim \nu^N}
[\langle U, (\mu_{\vect X} - m)^{\otimes 2}\rangle]
\leqslant \frac{2}{1-\delta} H( \nu^N | m^{\otimes N})
+ 4\biggl(1+\frac{1-\delta}{2\delta}\biggr).
\]
\end{cor}

Corollary~\ref{cor:jw} follows by setting $W = - (1-\delta) U$, $m_* = m$
and $\varepsilon=1$ in Theorem~\ref{thm:coercivity}.

\begin{rem}
\label{rem:jw}
Corollary~\ref{cor:jw} yields simpler constants
than those in the original work of \textcite[Theorem~4]{JabinWang}.
However, its applicability is more limited, as it requires
the kernel to be positive, which is not imposed in the cited work.
Nonetheless, in certain cases, this restriction can be bypassed.
For instance, if $U(x,y) = \varphi(x) \psi(y)$ for some
$\varphi$, $\psi \colon \X \to \R$, the polarization identity gives
\[
\varphi\otimes\psi + \psi\otimes\varphi
= \frac 12 \bigl( (\varphi+\psi)^{\otimes 2}
- (\varphi-\psi)^{\otimes 2}\bigr).
\]
Thus it suffices to apply Corollary~\ref{cor:jw} to the positive part
$(\varphi+\psi)^{\otimes 2}$.
For better constants in the Jabin--Z.~Wang large deviation estimate, see
also \textcite[Lemma~4.3]{LLNQuantitative}.
\end{rem}

\subsection{Contractivity}
\label{sec:mr-contractivity}

We next investigate the contractive properties of the modulated free energy
along Langevin dynamics. To this end, we impose two distinct
assumptions on the mean-field limit and formulate two corresponding theorems.

The first assumption operates on the level of entropy and energies:
\begin{alignat}{2}
\label{eq:mf-contractivity-1}
&\forall \nu \in \mathcal P_2(\X),&\qquad
H(\nu | \Pi[\nu])
&\geqslant \delta \mathcal F(\nu | m_*).
\tag{FE} \\
\intertext{Accordingly, we refer to this as a \emph{free energy condition}.
The second assumption is the contractivity \eqref{eq:mf-pl}
of the mean-field limit, reproduced here:}
\label{eq:mf-contractivity-2}
&\forall \nu\in \mathcal P_2(\X),&\qquad
I(\nu | \Pi[\nu])
&\geqslant 2\lambda \mathcal F(\nu | m_*).
\tag{P\L}
\end{alignat}
We refer to this as \eqref{eq:mf-contractivity-2}
because it is a \emph{Polyak--\L ojasiewicz inequality} for
the free energy $\mathcal F$ in Wasserstein space.

The two theorems are formulated as follows.

\begin{thm}
\label{thm:contractivity-1}
Let Assumptions~\ref{assu:W}, \ref{assu:R}, \ref{assu:LS}
and \ref{assu:LS-Pi} hold.
Let \eqref{eq:mf-coercivity} and \eqref{eq:mf-contractivity-1} hold
with the same constant $\delta \in (0,1]$.
Then for all $\varepsilon \in (0,1]$ and all $\nu^N \in \mathcal P_2(\X^N)$,
\[
I(\nu^N|m^N_*)
\geqslant 2\lambda_N
\mathcal F^N(\nu^N | m_*)
- \Delta_I,
\]
where $\lambda_N$ and $\Delta_I$ are given by
\begin{IEEEeqnarray*}{rCl}
\frac{\lambda_N}{\delta\rho_\Pi}
&=& 1 - \varepsilon
- \frac{8\rho\ell_R}{\rho_\Pi\varepsilon\delta^2N}
- \frac{48\ell_W}{\delta^2N}
\bigl(
1+(M_W+\ell_W)(1+\mu_R+\ell_R)\delta^{-2}\varepsilon^{-1}
\bigr), \\
\frac{\Delta_I}{2\rho_\Pi}
&=&
\frac{8\rho(\mu_R + \ell_Rd)}{\rho_\Pi\varepsilon}
+ 48(M_W+\ell_Wd)\bigl( 1
+ (M_W+\ell_W)(1+\mu_R+\ell_R)\delta^{-2}\varepsilon^{-1} \bigr).
\end{IEEEeqnarray*}
\end{thm}

\begin{thm}
\label{thm:contractivity-2}
Let Assumptions~\ref{assu:W}, \ref{assu:R} and \ref{assu:LS} hold.
Let \eqref{eq:mf-coercivity} and \eqref{eq:mf-contractivity-2}
hold with $\lambda > 0$ and $\delta \in (0,1]$
and denote $\eta = \lambda/\rho$.
Then for all $\varepsilon \in (0,1]$
and all symmetric $\nu^N \in \mathcal P_2(\X^N)$,
\[
I(\nu^N | m^N_*) \geqslant 2\lambda_N
\mathcal F^N(\nu^N | m_*^{\otimes N}) - \Delta_I,
\]
where $\lambda_N$ and $\Delta_I$ are given by
\begin{IEEEeqnarray*}{rCl}
\frac{\lambda_N}{\lambda}
&=& 1 - \varepsilon - \frac{12\ell_R}{\varepsilon\delta\eta N}
(1+\mu_R+\ell_R) \bigl( 1+(\mu_R + \ell_R)(\delta\eta)^{-1}\bigr) \\
&&\adjustbin
- \frac{48\ell_W}{\delta N}
\Bigl( 1 + (M_W+\ell_W)\bigl(1+(\mu_R+\ell_R)\delta^{-1}\bigr) \varepsilon^{-1}
\Bigr),
\\
\frac{\Delta_I}{2\rho} &=&
\frac{12(\mu_R+\ell_Rd)}{\varepsilon}
(1+\mu_R+\ell_R) \bigl( 1+(\mu_R + \ell_R)(\delta\eta)^{-1}\bigr) \\
&&\adjustbin
- 48\eta(M_W+\ell_Wd)
\Bigl( 1 + (M_W+\ell_W)\bigl(1+(\mu_R+\ell_R)\delta^{-1}\bigr) \varepsilon^{-1}
\Bigr).
\end{IEEEeqnarray*}
\end{thm}

The proofs of Theorems~\ref{thm:contractivity-1} and \ref{thm:contractivity-2}
rely crucially on the conditional approximation technique
and are given in Section~\ref{sec:contractivity}.

Observe that the assumptions of Theorem~\ref{thm:contractivity-1}
imply those of Theorem~\ref{thm:contractivity-2}
with $\lambda = \delta\rho_\Pi$.
Indeed, by successively applying Assumption~\ref{assu:LS-Pi} and
\eqref{eq:mf-contractivity-1}, we deduce
\[
I(\nu | \Pi[\nu]) \geqslant 2\rho_\Pi H(\nu | \Pi[\nu])
\geqslant 2\delta\rho_\Pi \mathcal F(\nu | m_*).
\]
Conversely, Theorem~\ref{thm:contractivity-1} applies
to any probability measure in $\mathcal P_2$,
whereas Theorem~\ref{thm:contractivity-2} is limited to symmetric ones.

The contractivity constant in Theorem~\ref{thm:contractivity-2}
is almost optimal
and the justification parallels that for Theorem~\ref{thm:coercivity}.
Indeed, fix $\nu \in \mathcal P_2(\X)$ with $I(\nu | m_*) < \infty$
and set $\nu^N = \nu^{\otimes N}$.
Theorem~\ref{thm:contractivity-2} yields
\[
\frac{I(\nu^{\otimes N} | m^N_*)}{N} \geqslant 2\lambda_N
\frac{\mathcal F^N(\nu^{\otimes N} | m_*)}{N} - \frac{\Delta_I}{N}.
\]
Meanwhile, Lemmas~\ref{lem:recover-free-energy} and \ref{lem:recover-Fisher}
in Appendix~\ref{app:recover-mf} give, respectively,
\[
\lim_{N \to \infty} \frac{I(\nu^{\otimes N} | m^N_*)}{N}
= I(\nu | \Pi[\nu]),\quad
\lim_{N \to \infty} \frac{\mathcal F^N(\nu^{\otimes N} | m^N_*)}{N}
= \mathcal F(\nu | m_*).
\]
Letting $N \to \infty$ and using the arbitrariness of $\varepsilon$
recovers \eqref{eq:mf-contractivity-2}.

\begin{rem}
The assumptions \eqref{eq:mf-contractivity-1} and \eqref{eq:mf-contractivity-2}
are imposed directly on the functionals defined
on the space of probability measures.
In contrast, \textcite{BBDCriterionFreeEnergy}
require a mode decomposition as in Remark~\ref{rem:bbd-spectral},
project the free energy functional $\mathcal F$ according to this decomposition,
and impose a Polyak--\L ojasiewicz inequality
on the projected functional.
The author is not aware whether this inequality
implies \eqref{eq:mf-contractivity-1} or \eqref{eq:mf-contractivity-2},
or conversely.
\end{rem}

\begin{rem}
In both theorems,
the error term $\Delta_I$ remains of order $O(1)$ as $N \to \infty$,
and this asymptotic behavior is realized in Gaussian examples;
see~\cite[Remark~2.7]{ulpoc}.
However, when approaching the phase transition, namely
when $\delta$, $\eta \to 0$, the error term
does not appear to exhibit the optimal critical behavior.
\end{rem}

Combining the coercivity and contractivity of the modulated free energy yields
a \emph{defective} log-Sobolev inequality for the Gibbs measure $m^N_*$.
This, in turn, yields exponential entropy convergence for the Langevin
dynamics up to $O(1)$ terms.
The precise statement is given below and the proof is
provided in Section~\ref{sec:contractivity}.

\begin{cor}
\label{cor:defective-lsi}
Under the assumptions and notations of Theorems~\ref{thm:coercivity}
and \ref{thm:contractivity-1},
if $\delta_N$, $\lambda_N > 0$,
then for all $\nu^N \in \mathcal P_2(\X^N)$,
\[
I(\nu^N | m^N_*) \geqslant 2\lambda_N H(\nu^N | m^N_*)
- \Delta_I - 2\lambda_N \Delta_{\mathcal F}.
\]
Consequently, the Langevin dynamics \eqref{eq:ps-fp} satisfies
\[
H(m^N_t | m^N_*) \leqslant e^{-2\lambda_N t} H(m^N_0 | m^N_*)
- \frac{\Delta_I}{2\lambda_N} - \Delta_{\mathcal F}.
\]

Similarly, under the assumptions and notations of Theorems~\ref{thm:coercivity}
and \ref{thm:contractivity-2},
if $\delta_N$, $\lambda_N > 0$,
then the above assertions hold for symmetric $\nu^N \in \mathcal P_2(\X^N)$
and symmetric Langevin dynamics, respectively.
\end{cor}

Corollary~\ref{cor:defective-lsi} partially confirms the recent conjecture of
\textcite{MonmarcheULSIBeyond}. Under the
Polyak--\L ojasiewicz inequality~\eqref{eq:mf-contractivity-2}, the
defective log-Sobolev inequality holds only for symmetric measures,
whereas the general case requires the stronger hypotheses of
Theorem~\ref{thm:contractivity-1}.
The original conjecture of \textcite{DGPSPhase}
remains open and appears challenging to the author.

\begin{rem}
If a $\lambdaPN$-Poincaré inequality for $m^N_*$ can be established,
namely, for all sufficiently regular $f \colon \X^N \to \R$,
\[
\int_{\X^N} { \lvert \nabla f\rvert^2 \dd m^N_*}
\geqslant \lambdaPN \biggl( \int_{\X^N} f^2 \dd m^N_*
- \Bigl( \int_{\X} f \dd m^N_* \Bigr)^{\!2}\biggr),
\]
then, by the quantitative tightening in \cite[Proposition~5]{nulsi},
the defective inequality in Corollary~\ref{cor:defective-lsi}
would yield a log-Sobolev inequality for $m^N_*$ with constant
\[
\biggl( 1 + \frac{\Delta_I + 2\lambda_N\Delta_{\mathcal F}}{4\lambdaPN}
\biggr)^{\!-1} \lambda_N.
\]
If $\lambdaPN$ admits a uniform lower bound in $N$,
then so does this log-Sobolev constant.
However, the author is not aware of a proof
of the uniform Poincaré inequality in this context.
Furthermore, even with $\lambdaPN = \lambda$,
this does not yield $\lambda$ as the log-Sobolev constant,
contrary to the conjecture in \cite{DGPSPhase}.
\end{rem}

\subsection{Generation of chaos}
\label{sec:mr-goc}

We now examine the phenomenon of \emph{generation of chaos}.
This term describes situations where the initial distribution of particles is
far from tensorized, yet dynamical dissipation drives the system so that
the particles become nearly independent after sufficient time.
This concept is even stronger than the uniform-in-time propagation of chaos
that has been studied intensively in recent years.

We employ the modulated free energy method to analyze this phenomenon.
This approach studies the evolution of the dynamically modulated free energy
\[
\mathcal F^N(m^N_t | m_t)
= H(m^N_t | m_t^{\otimes N})
+ \frac N2 \Expect_{\vect X \sim m^N_t}
[ \langle W, (\mu_{\vect X} - m_t)^{\otimes 2}\rangle ],
\]
which corresponds to replacing the static $m_*$ with the dynamical $m_t$
in the centered rewriting \eqref{eq:def-modulated-F-centered}
of modulated free energy.
\textcite{BJWPKSCompteRendu,BJWMFE}
originally introduced the modulated free energy
in its dynamical form to study propagation of chaos
for singular diffusive dynamics.
Here, the focus is not on accommodating singularities, but on establishing
long-time estimates with the optimal mean-field convergence rate.

To formulate the result, we introduce the following notations:
\begin{align*}
\Pi_{m} [\nu] (\dd x) &\propto
\exp (- \langle W(x,\cdot), \nu - m\rangle)\,m(\dd x), \\
\mathcal F(\nu | m) &= H(\nu | m)
+ \frac 12 \langle W, (\nu-m)^{\otimes 2}\rangle.
\end{align*}
They correspond to replacing $m_*$ with $m$
in the centered rewritings \eqref{eq:def-Pi-centered}
and \eqref{eq:def-relative-F-centered} of $\Pi$ and $\mathcal F(\cdot|m_*)$,
respectively.
We also introduce the coercivity and contractivity conditions
associated to the new free energy functional $\mathcal F(\cdot|m)$:
\begin{alignat}{2}
&\forall \nu \in \mathcal P_2(\X),&\qquad
\mathcal F(\nu | m) &\geqslant \delta \mathcal H(\nu | m),
\label{eq:Coer-m} \tag{Coer$_m$}\\
&\forall \nu \in \mathcal P_2(\X),&\qquad
H(\nu | \Pi_m[\nu]) &\geqslant \delta \mathcal F(\nu | m),
\label{eq:FE-m} \tag{FE$_m$}
\end{alignat}
Furthermore, define the vector field
\[
v_t(x) \coloneqq - \nabla \log \frac{m_t}{\Pi[m_t]}(x)
= - \nabla \log m_t(x) - \langle \nabla_1 W(x,\cdot), m_t\rangle.
\]
Observe that the flow \eqref{eq:mf-fp} is transported by $v_t$:
\[
\partial_t m_t + \nabla\cdot(m_tv_t) = 0.
\]

\begin{thm}
\label{thm:goc}
Let Assumptions~\ref{assu:W} and \ref{assu:R} hold.
Let $\delta$, $\rho$, $\rho_\Pi \colon [0,\infty) \to (0,\infty)$
and $\gamma$, $M \colon [0, \infty) \to [0,\infty)$.
For all $t \geqslant 0$,
suppose that \eqref{eq:Coer-m} and \eqref{eq:FE-m} hold
for $m = m_t$ and $\delta = \delta(t)$;
that $m_t$ satisfies a $\rho(t)$-log-Sobolev inequality
and the images of $\Pi_{m_t}$ satisfies a $\rho_{\Pi}(t)$-log-Sobolev equality
uniformly;
and that the dynamics \eqref{eq:mf-fp} and \eqref{eq:ps-fp} satisfy
\begin{multline}
\label{eq:goc-error}
- N \int_{\X^N} \int_{\X^2}
\nabla_1 W(y,z) \cdot v_t(z)\,(\mu_{\vect x} - m_t)^{\otimes 2}
(\dd y \dd z)\, m^N_t(\dd\vect x) \\
\leqslant 2\gamma(t) \mathcal F^N(m^N_t | m_t) + M(t).
\end{multline}
Then for all $\varepsilon \in (0,1)$ and for all $t \geqslant 0$,
\[
\mathcal F^N(m^N_t | m_t)
\leqslant \exp\bigl(-2\Gamma(t,0)\bigr)
\mathcal F^N(m^N_0 | m_0)
+ \int_0^t \exp\bigl(-2\Gamma(t,s)\bigr)
\bigl( \Delta_I(s) + M(s) \bigr) \dd s,
\]
where
$\Gamma (t,s) \coloneqq \int^t_s \bigl(\lambda_N(u) - \gamma(u)\bigr) \dd u$,
and $\lambda_N(t)$ and $\Delta_I(t)$ are obtained
from the expressions for $\lambda$ and $\Delta_I$
in Theorem~\ref{thm:contractivity-1} with the following substitutions:
\[
\delta\to\delta(t),
\quad\rho\to\rho(t),
\quad\rho_\Pi \to \rho_\Pi(t),
\quad\ell_W\to \frac{L_W}{\rho(t)},
\quad\mu_R\to \frac{M_R}{\rho(t)},
\quad\ell_R\to \frac{L_R}{\rho(t)^2}.
\]
\end{thm}

Theorem~\ref{thm:goc} follows directly from Theorem~\ref{thm:contractivity-2},
and its proof is provided in Section~\ref{sec:goc}.
The central idea is to express the modulated free energy's dissipation as
\[
I (m^N_t | g^N_t),
\]
where $g^N_t$ is the $N$-particle Gibbs measure
corresponding to the free energy functional $\mathcal F(\cdot | m_t)$.
This mirrors the main argument in a recent work
of \textcite{RosenzweigSerfatyMLSI},
where $g^N_t$ is assumed to satisfy a log-Sobolev inequality uniformly in $N$.
Here, we do not impose this assumption; instead, we apply
Theorem~\ref{thm:contractivity-1} to the new free energy
and achieve contraction up to an $O(1)$ term.

\begin{rem}
The condition~\eqref{eq:goc-error} corresponds
to the functional inequality derived by \textcite[Proposition~1.1]{SerfatyME}
for mean-field Coulomb flows, later extended to all orders of commutation
by \textcite{RosenzweigSerfatySharpCommutator}.
In contrast, our aim is not to establish an inequality
accommodating singularities in $W$, but to derive long-time relaxation,
reflected in a decaying $\gamma$.
In our regular setting, the left-hand side of \eqref{eq:goc-error}
can be often controlled by $H(m^N_t | m_t^{\otimes N})$
through large deviation estimates of the Jabin--Z.~Wang type;
see Corollary~\ref{cor:jw}.
It then suffices to apply Theorem~\ref{thm:coercivity}
to further bound $H(m^N_t | m_t^{\otimes N})$
by $\mathcal F^N(m^N_t | m_t)$.
\end{rem}

\subsection{Independent projection}
\label{sec:mr-ip}

Finally, we establish long-time convergence for the \emph{independent
projection} of Langevin dynamics, introduced by
\textcite{LackerIndependentProjections}, as the concluding main result.
Let $N \geqslant 2$.
The independently projected Langevin dynamics is defined,
for each $i \in [N]$, by
\begin{equation}
\label{eq:ip-sde}
\dd X^i_t = - \nabla V(X^i_t) \dd t
- \frac{1}{N-1} \sum_{j \in [N]: j \neq i}
\Expect[\nabla_1 W(X^i_t, X^j_t) | X^i_t] \dd t
+ \sqrt{2} \dd B^i_t,
\end{equation}
where $B^i$ are independent standard $d$-dimensional Brownian motions
and $X^i_t \in \X$.
This dynamics addresses the mean-field variational inference problem
\[
\inf_{\xi^1,\ldots,\xi^N \in \mathcal P_2(\X)}
H(\xi^1\otimes\cdots\otimes\xi^N | m'^N_*),
\]
where $m'^N_*$ is the $N$-particle Gibbs measure without self-interaction:
\[
m'^N_*(\dd\vect x) \propto
\exp \biggl(
  - \sum_{i \in [N]} V(x^i)
  - \frac{1}{2(N-1)} \sum_{i,j \in [N]: i\neq j} W(x^i,x^j)
\biggr)
\dd\vect x.
\]
See \cite{LackerIndependentProjections} for further details and references.

Suppose that the initial value
$\vect X_0 = (X^1_0, \ldots, X^N_0)$ has a tensorized distribution:
\begin{align*}
\Law(\vect X_0) &= \vect\xi_0 = \xi^1_0 \otimes \cdots \otimes \xi^N_0.
\intertext{%
It is straightforward to verify that, under suitable regularity assumptions on
$V$ and $W$, the unique strong solution to \eqref{eq:ip-sde} remains
tensorized, that is,}
\Law(\vect X_t) &= \vect\xi_t = \xi^1_t \otimes \cdots \otimes \xi^N_t.
\end{align*}
Moreover, each $\xi^i_t$ solves the Fokker--Planck equation
\begin{equation}
\label{eq:ip-fp}
\partial_t \xi^i_t
= \Delta \xi^i_t
+ \nabla\cdot(\nabla V\xi^i_t)
+ \nabla\cdot\biggl(
\int_{\X} \nabla_1 W(\cdot,y)\,\bar\xi^{-i}_t(\dd y)\,\xi^i_t
\biggr),
\end{equation}
where
\[
\bar\xi^{-i}_t \coloneqq \frac 1{N-1}
\sum_{j\in[N]:j \neq i} \xi^j_t.
\]

Our analysis of the long-time behavior of the projected dynamics is based on
\eqref{eq:ip-fp}, thus operating at the marginal distribution level. More
precisely, for $\xi^1, \ldots, \xi^N \in \mathcal P_2(\X)$, we introduce the
functional
\[
\mathcal F^N_{\ind}(\xi^1, \ldots, \xi^N | m_*)
= \sum_{i \in [N]} H(\xi^i | m_*)
+ \frac{1}{2(N-1)} \sum_{i, j \in [N]: i \neq j}
\langle W_*, \xi^i \otimes \xi^j \rangle.
\]
Observe that
\begin{multline*}
\mathcal F^N_{\ind}(\xi^1, \ldots, \xi^N | m_*)
- \mathcal F^N(\xi^1 \otimes \cdots \otimes \xi^N | m_*) \\
= \frac{1}{2N(N-1)}
\sum_{i, j \in [N]: i \neq j}
\langle W_*, \xi^i \otimes \xi^j \rangle
- \frac{1}{2N} \sum_{i \in [N]}
\langle W_*, (\xi^i)^{\otimes 2} \rangle.
\end{multline*}
When $W$ is bounded, this difference is of order $O(1)$, whereas both
$\mathcal F^N_{\ind}$ and $\mathcal F^N$ are typically of order $O(N)$.
Hence the difference becomes negligible for large $N$.
For this reason, we also refer to $\mathcal F^N_{\ind}$
as the modulated free energy.
Furthermore, the dynamics
\eqref{eq:ip-fp} is a constrained gradient flow for
$\mathcal F^N_{\ind}(\cdot|m_*)$.
Specifically,
\begin{equation}
\label{eq:ip-dissipation}
\frac{\dd}{\dd t} \mathcal F^N_{\ind}(\xi^1_t, \ldots, \xi^N_t | m_*)
= - \sum_{i \in [N]} I(\xi^i_t | \Pi[\bar\xi^{-i}_t]).
\end{equation}
See \cite[Section~2.1 and Section~3]{LackerIndependentProjections} for details.

As in the unprojected Langevin case, we investigate the coercivity and
contractivity of the modulated free energy functional $\mathcal F^N_{\ind}$.

\begin{thm}
\label{thm:ip-coercivity}
Let Assumptions~\ref{assu:W} and \ref{assu:T} hold.
Let \eqref{eq:mf-coercivity} hold with $\delta \in (0,1]$.
Let $N \geqslant 2$.
Then for all $\xi^1$, \dots, $\xi^N \in \mathcal P_2(\X)$,
\[
\mathcal F^N_{\ind} (\xi^1,\ldots,\xi^N | m_*)
\geqslant \delta_N \sum_{i \in [N]} H(\xi^i | m_*),
\]
where $\delta_N$ is given by
\[
\delta_N
= \delta - \frac{1+M_W+\ell_W}{N-1}.
\]
\end{thm}

\begin{thm}
\label{thm:ip-contractivity}
Let Assumptions~\ref{assu:W}, \ref{assu:T}
and \ref{assu:LS-Pi} hold.
Let \eqref{eq:mf-coercivity} and \eqref{eq:mf-contractivity-1} hold
with the same constant $\delta \in (0,1]$.
Let $N > \delta^{-1}(1+M_W+\ell_W) + 1$.
Then for all $\xi^1$, \dots, $\xi^N \in \mathcal P_2(\X)$,
\[
\sum_{i \in [N]} I(\xi^i | \Pi[\bar\xi^{-i}])
\geqslant 2\lambda_N \mathcal F^N_{\ind} (\xi^1,\ldots,\xi^N | m_*),
\]
where $\lambda_N$ is given by
\[
\frac{\lambda_N}{2\rho_{\Pi}}
= \delta - \frac{2(1+M_W+\ell_W)}{(N-1)\delta-1-M_W-\ell_W}.
\]
\end{thm}

The proofs of Theorems~\ref{thm:ip-coercivity} and~\ref{thm:ip-contractivity}
will be presented in Section~\ref{sec:ip}.
These proofs are analogous to those of Theorems~\ref{thm:coercivity}
and~\ref{thm:contractivity-1}, but are considerably simpler:
particle independence removes the need to
approximate the empirical measure by conditional ones,
and thus eliminates the associated error estimates.

Combining \eqref{eq:ip-dissipation} with Theorems~\ref{thm:ip-coercivity}
and~\ref{thm:ip-contractivity} immediately yields exponential entropy
convergence for the projected dynamics.

\begin{cor}
Under the assumptions and notations of Theorem~\ref{thm:ip-contractivity},
the independently projected Langevin dynamics \eqref{eq:ip-fp} satisfies
\[
\sum_{i \in [N]} H(\xi^i_t| m_*)
\leqslant \delta_N^{-1} e^{-2\lambda_N t}
\mathcal F^N_{\ind} (\xi^1_0, \ldots, \xi^N_0 | m_*),
\]
where $\delta_N$, $\lambda_N$ are defined in
Theorems~\ref{thm:ip-coercivity} and \ref{thm:ip-contractivity}, respectively.
\end{cor}

\subsection{Future directions}
\label{sec:mr-future-directions}

\subsubsection*{Optimal relaxation rate at all scales}

While Theorem~\ref{thm:contractivity-2} establishes near-optimal
large-scale relaxation rates, the method for achieving the optimal
relaxation rate at all scales, or equivalently, the best
log-Sobolev constant, remains unknown. The one-step renormalization approach
of \textcite{BauerschmidtBodineauSimple} appears insufficient for obtaining the
optimal constant, as Section~\ref{sec:exm-xy} below demonstrates for the
mean-field XY model.

Moreover, the validity of the conjecture by \textcite{DGPSPhase} remains open.
As the authors note, the mean-field limit
provides only a \emph{coarse-grained} view of the $N$-particle system.
Condition~\eqref{eq:mf-pl} alone may
therefore fail to yield relaxation rates at small scales,
since such information may be lost in the coarse graining process.

\subsubsection*{Renormalization group perspective}

It was brought to the author's notice that the construction
of the conditional approximation bears similarity to
the \emph{coordinate localization} scheme
of \textcite{ChenEldanLocalizationSchemes}.
This scheme proposes the gradual decomposition
\[
m^N_* = \Expect_{\vect X \sim m^N_*} [ \delta_{\vect X^{[k-1]}}
\otimes \Law(\vect X^{\llbracket k, N\rrbracket} | \vect X^{[k-1]}) ],
\qquad k \in [N],
\]
where, in renormalization group terminology,
$\delta_{\vect X^{[k-1]}}$ serves as the macroscopic renormalized field,
and $\Law(\vect X^{\llbracket k, N\rrbracket} | \vect X^{[k-1]})$ serves as
the microscopic fluctuation measure.
The proof of log-Sobolev inequality for $m^N_*$ then relies on
the \emph{entropic stability} across levels $k \in [N]$.

It remains unclear how to interpret the proofs of
Theorems~\ref{thm:contractivity-1} and~\ref{thm:contractivity-2} in
the language of renormalization group.
Our method necessarily introduces error terms when alternating
between $\mu_{\vect X}$ and $\bar\nu$.
Consequently, any such interpretation would lead
to an inexact renormalization procedure,
more akin to Kadanoff's block-spin approach~\cite{KadanoffScaling}
than to the Polchinski flow \cite{PolchinskiRenormalization}.

\subsubsection*{Non-exchangeable interaction}

The conditional approximation approach can be adapted
to non-exchangeable interactions.
Specifically, instead of \eqref{eq:ps-sde}, we consider
\[
\dd X^i_t = - \nabla V(X^i_t) \dd t
- \sum_{j \in [N]} a_{ij} \nabla_1 W(X^i_t, X^j_t) \dd t
+ \sqrt 2 \dd B^i_t,
\]
where $a_{ij} \geqslant 0$.
This adaptation requires comparing the measures
\[
\sum_{j \in [N]} a_{ij} \delta_{X^j},
\quad \sum_{j \in [N]} a_{ij} \nu_j.
\]
By the martingale argument of Lemma~\ref{lem:approx-1},
for all bounded kernel $U \colon \X \times \X \to \R$,
\[
\Expect \biggl[
\biggl< U, \biggl( \sum_{j \in [N]} a_{ij} \delta_{X^j}
- \sum_{j \in [N]} a_{ij} \nu_j \biggr)^{\!\otimes 2} \biggr> \biggr]
= O\biggl( \sum_{j \in [N]} a_{ij}^2 \biggr).
\]
The approximation is therefore accurate when the right-hand side is small.
This smallness condition also appears in recent studies of non-exchangeable
diffusions~\cite{LYZNonExchangeable}
and stochastic differential games~\cite{CJRNonAsymptotic}.
We leave the detailed analysis of this extension
as an open direction for future work.

\section{Examples}
\label{sec:exm}

This section presents examples illustrating the scope and limitations
of the results in Section~\ref{sec:mr}.
Section~\ref{sec:exm-xy} introduces the mean-field XY model,
where all results from Section~\ref{sec:mr}
hold throughout its subcritical regime.
Section~\ref{sec:exm-cw} introduces the supercritical double-well
Curie--Weiss model with an external magnetic field,
where the coercivity holds but the contractivity does not.

\subsection{Mean-field XY model}
\label{sec:exm-xy}

In the mean-field XY model, $\X = \T = \R / 2\pi \mathbb Z$ and
\[
V = 0,\quad W(\theta,\theta') = - J \cos(\theta - \theta'),
\]
where $J > 0$.
By translation invariance, the uniform measure
\[
m_*(\dd\theta) = \frac{\dd\theta}{2\pi}
\]
is invariant to the Langevin dynamics \eqref{eq:mf-fp}.

We first verify the regularity and functional inequality framework
in Section~\ref{sec:mr-reg-fi}.
The vector-valued function
\[
r(\theta) = \begin{pmatrix} \cos \theta \\ \sin \theta \end{pmatrix}
\]
provides a spectral decomposition for $W$:
\[
W(\theta,\theta') = -Jr(\theta) \cdot r(\theta')
= -J ( \cos\theta\cos\theta' + \sin\theta\sin\theta' ).
\]
In particular $W$ is a negative kernel.
According to Remark~\ref{rem:bbd-spectral}, by setting
\[
W^+ = 0,\quad W^- = -W,\quad R = - JW,
\]
Assumptions~\ref{assu:W} and \ref{assu:R} are verified
with the constants
\[
M_W = J,\quad M_R = J^2,\quad L_W = L_R = 0.
\]
The uniform measure $m_*$ verifies Assumption~\ref{assu:LS} with $\rho = 1$.
Let $\nu \in \mathcal P(\T)$.
The associated local equilibrium reads
\[
\Pi[\nu] (\dd \theta) \propto
\exp \bigl( J \langle r, \nu\rangle \cdot r(\theta) \bigr)\,
m_*(\dd \theta),
\]
where $\Osc J \langle r,\nu\rangle \cdot r(\cdot) \leqslant 2J$.
Applying the Holley--Stroock perturbation result \cite{HolleyStroockLSI}
then verifies Assumption~\ref{assu:LS-Pi} with $\rho_\Pi = e^{-2J}$.

For $J < 2$, using \cite[Theorem~D.2]{DLSPhase},
we can verify the conditions of Lemma~\ref{lem:xy-perturbed-coer-fe}
in Appendix~\ref{app:mode-decomposition}.
When $m = m_*$, the lemma yields, in particular,
conditions \eqref{eq:mf-coercivity} and \eqref{eq:mf-contractivity-1} with
\[
\delta = 1 - \frac J2.
\]
This constant is optimal according to Remark~\ref{rem:optimality-delta}.
For $J > 2$, the limit dynamics \eqref{eq:mf-fp}
possesses multiple invariant measures with equal free energy.
For $J = 2$, the model is critical,
and \eqref{eq:mf-coercivity} and \eqref{eq:mf-contractivity-1}
no longer capture the leading-order behavior.
Both cases are beyond the scope of this paper.
We therefore focus on the case $J < 2$ in the sequel.

At this stage, we have verified
the assumptions of all results from Section~\ref{sec:mr}
except for Theorem~\ref{thm:goc}.
In particular, Theorem~\ref{thm:coercivity} holds with
\begin{alignat*}{2}
\delta_N &= (1-\varepsilon) \delta, &\quad
\quad \Delta_{\mathcal F} &\leqslant 4
\biggl(1+\frac{2}{\delta\varepsilon}\biggr)
\intertext{and Theorem~\ref{thm:contractivity-1} with}
\frac{\lambda_N}{\rho_\Pi} &= (1-\varepsilon) \delta,&\quad
\frac{\Delta_I}{2\rho_\Pi} &\leqslant
96 + \frac{32e^4}{\varepsilon} + \frac{960}{\delta^2\varepsilon}.
\end{alignat*}
The coercivity constant $\delta_N$ is near-optimal
as argued in Section~\ref{sec:mr-coercivity}.
The contractivity constant $\lambda_N$ appears suboptimal,
with the loss arising from uniform log-Sobolev constant $\rho_\Pi$.
Nonetheless, we find $\lambda_N \sim 1 - J/2$ up to a multiplicative constant,
which is the expected critical behavior.
Moreover, the result of \textcite[Theorem~1]{BauerschmidtBodineauSimple}
gives the log-Sobolev constant
\[
\rho_\Pi \biggl( 1 + \frac{2J}{\delta} \biggr)^{\!-1}.
\]
Thus, when $J$ is close to $2$,
Theorem~\ref{thm:contractivity-1} achieves a multiplicative improvement
of around $4$ in the large-scale relaxation rate.
Theorem~\ref{thm:contractivity-2} would further
yield near-optimal large-scale relaxation
for symmetric particle distributions, but the author is not aware of
an explicit expression for the best constant in \eqref{eq:mf-contractivity-2}.

Combining standard global-in-time propagation of chaos
with Theorem~\ref{thm:goc}
yields the following entropic generation of chaos estimate.
However, verifying the assumptions of Theorem~\ref{thm:goc} requires extra work,
so we defer the proof to Appendix~\ref{app:xy}.

\begin{prop}
\label{prop:xy-goc}
There exists $C_J > 0$, depending only on $J \in (0,2)$, such that
for all $\varepsilon \in (0,1)$
and all $t \geqslant 0$,
\[
H(m^N_t | m_t^{\otimes N})
\leqslant C_J
e^{-2(1-\varepsilon)\delta\rho_{\Pi}t}
H(m^N_0 | m_0^{\otimes N}) + \frac{C_J}{(1-\varepsilon)\varepsilon}.
\]
\end{prop}

The novelty of this result lies in the fact that the last term is $O(1)$ as
$N\to\infty$. The log-Sobolev inequality for the Gibbs measure $m^N_*$,
as established by \textcite{BauerschmidtBodineauSimple}, yields contraction only
for $H(m^N_t | m^N_*)$ and thus does not establish chaoticity on short time
scales. In contrast, global-in-time propagation of chaos provides chaoticity
estimates within finite time horizons, but these estimates deteriorate
exponentially as time increases. While combining both approaches yields
uniform-in-time chaoticity, this comes at the cost of a significantly degraded
convergence rate in $N$, typically much worse than $O(1)$; see
\textcite[Corollary~5]{GuillinMonmarcheKinetic} and an earlier work
of the author~\cite[Corollary~2.8]{ulpoc}.

The supercritical regime $J > 2$, in which \eqref{eq:mf-fp} admits
a circle of single-peaked stationary solutions, lies beyond the scope of
the present analysis and has been investigated in depth
by \textcite{BGPSynchronization},
who characterize the long-time random motion on this manifold of
invariant measures.

A natural direction for future work is the multi-mode generalization
of the present model on the circle, the noisy Kuramoto--Daido model,
which has been recently studied at the mean-field level by
\textcite{MunRosenzweigPhaseTransitions};
combining their analysis with the main results of the present paper
should yield analogous large-scale estimates
for the associated $N$-particle system in the subcritical regime.

\subsection{Double-well Curie--Weiss model}
\label{sec:exm-cw}

In the double-well Curie--Weiss model, $\X = \R$ and
\[
V(x) = V_h(x) = \frac{\theta x^4}{4} - \frac{\sigma x^2}{2} - hx,\quad
W(x) = - Jxy,
\]
where $\sigma$, $h \in \R$ and $\theta$, $J > 0$.
The mode decomposition in Remark~\ref{rem:bbd-spectral} becomes obvious:
$r(x) = x$.
Accordingly, by setting
\[
W^+ = 0,\quad W^- = -W,\quad R = -JW,
\]
Assumptions~\ref{assu:W} and \ref{assu:R} are verified with the constants
\[
L_W = J,\quad L_R = J^2,\quad M_W=M_R = 0.
\]
Assumption~\ref{assu:LS-Pi} can be established by various methods;
see \cite[Proof of Proposition~6]{socgibbs} and
\cite[Assumption~1.1]{BBDCriterionFreeEnergy}.

For $h = 0$, the measure
with density proportional to $e^{-V_0}$ is invariant.
In this case, denote $m_*(\dd x) \propto e^{-V_0(x)} \dd x$.
Define
\[
\Jc = \frac{\int_{\R} e^{-V_0(x)} \dd x}{\int_{\R} x^2 e^{-V_0(x)} \dd x};
\]
in other words, $\Jc$ is the inverse of the variance of $m_*$.
The potential $V_0$ belongs to the Griffiths--Hurst--Sherman (GHS) class
so the conditions of Lemma~\ref{lem:xy-perturbed-coer-fe}
in Appendix~\ref{app:mode-decomposition} are verified;
see \cite{EMNGHS}.
For $J < \Jc$, applying the lemma yields
\eqref{eq:mf-coercivity} and \eqref{eq:mf-contractivity-1}
with $\delta = 1 - J/\Jc$,
a constant that Remark~\ref{rem:optimality-delta} shows to be optimal.
Consequently, all results from Section~\ref{sec:mr}
except Theorem~\ref{thm:goc} hold.
In particular, Theorems~\ref{thm:coercivity} and \ref{thm:contractivity-1}
yield, respectively, large-scale concentration and relaxation
with the expected critical behavior.
As in Section~\ref{sec:exm-xy},
the author is not aware of an explicit expression for the best constant
in \eqref{eq:mf-contractivity-2},
so the result of Theorem~\ref{thm:contractivity-2} remains inexplicit.
Verifying the assumptions of Theorem~\ref{thm:goc}
seems challenging to the author, as the unbounded setting
complicates the necessary analytic estimates on the dynamics~\eqref{eq:mf-fp}.

In the sequel, we are instead interested in the supercritical phase
with a non-vanishing external field.
To be precise, define $f \colon \R \to \R$ by
\[
f(\ell) = \log \int_{\R} e^{\ell x-V_0(x)} \dd x.
\]
We consider the parameter regime
\begin{equation}
\label{eq:cw-regime}
J > \Jc,\quad 0 < h < \hc(J),
\end{equation}
where $\hc(J) = Jf'\bigl(\ell(J)\bigr) - \ell(J)$,
and $\ell(J)$ is the unique positive number satisfying
$f''\bigl(\ell(J)\bigr) = 1/J$.
The uniqueness of such $\ell(J)$ is guaranteed by the GHS property
$f'''(\ell) < 0$ for $\ell > 0$.
In this regime, the dynamics \eqref{eq:mf-fp}
admits exactly three invariant measures,
denoted $m_-$, $m_{\neu}$ and $m_+$, and satisfying
\[
m_s(\dd x) \propto e^{\ell_s x-V_0(x)} \dd x,
\qquad\text{$s = -$, $\neu$, $+$,}
\]
with $\ell_- < \ell_{\neu} < \ell_+$.
Furthermore, these numbers solve the self-consistency equation
\begin{equation}
\label{eq:self-consistency}
\ell = h + J f'(\ell).
\end{equation}
We choose $m_* = m_+$.
Lemma~\ref{lem:cw-coercive} in Appendix~\ref{app:cw} verifies the condition
\eqref{eq:mf-coercivity} for some $\delta > 0$.
Consequently the coercivity results from Section~\ref{sec:mr-coercivity} hold.
In particular, by Corollary~\ref{cor:entropy-coercivity},
\[
H(m^N_* | m_+^{\otimes N}) = O(1)
\]
and $m^N_*$ exhibits Gaussian concentration with an $N$-independent constant.
However, the contractivity conditions \eqref{eq:mf-contractivity-1}
and \eqref{eq:mf-contractivity-2} must fail.
Indeed, by taking $\nu = m_-$, $m_{\neu}$,
\[
I(\nu | \Pi[\nu]) = H(\nu | \Pi[\nu]) = 0,
\]
while
\[
\mathcal F(\nu | m_+) \geqslant \delta H(\nu | m_+) > 0.
\]
Moreover, by the optimality argument for Theorem~\ref{thm:contractivity-2},
the failure of \eqref{eq:mf-contractivity-2}
excludes any exponential convergence rate
for the $N$-particle dynamics \eqref{eq:ps-fp}
that holds independently of $N$ and the initial condition.

Nonetheless, restricting the initial condition recovers large-scale
exponential convergence for the particle system.
\textcite[Proposition~21]{MonmarcheReygnerLocal} demonstrated that
if $m^N_0 = m_0^{\otimes N}$ and $m_t$ converges to $m_+$ as $t\to\infty$,
then there exist $C$, $c > 0$ and $\beta \in (0,1)$ such that
for all $N\geqslant 1$ and $t \geqslant 1$,
\[
\frac 1N \mathcal F^N(m^N_t | m_+)
= \frac 1N \mathcal F^N(m^N_t) - \mathcal F(m_+)
\leqslant C \biggl(N^{-\beta} + e^{-ct}\biggr).
\]
Applying Theorem~\ref{thm:coercivity} and enlarging $C$ if necessary,
we obtain
\[
\frac 1N H(m^N_t | m_+^{\otimes N})
\leqslant C(N^{-\beta} + e^{-ct}).
\]
Combining this with global-in-time propagation of chaos estimates
on finite horizons yields the following uniform-in-time bound.

\begin{prop}
\label{prop:cw-uniform-poc}
Assume the parameter regime~\eqref{eq:cw-regime}
and let $m_0 \in \mathcal P_2(\R)$ have $H(m_0 | m_+) < \infty$.
Suppose the solution $(m_t)_{t \geqslant 0}$
of~\eqref{eq:mf-fp} starting from $m_0$ converges to $m_+$ as $t\to\infty$,
and suppose further that $m^N_0 = m_0^{\otimes N}$.
Then there exist constants $C > 0$ and $\beta \in (0, 1)$,
depending on $h$, $J$, $\sigma$, $\theta$ and $m_0$,
such that for all $N \geqslant 1$ and $t \geqslant 0$,
\[
\frac{1}{N} W_2^2(m^N_t, m_t^{\otimes N}) \leqslant C N^{-\beta}.
\]
\end{prop}

\begin{proof}[Sketch of proof of Proposition~\ref{prop:cw-uniform-poc}]
By Assumption~\ref{assu:LS-Pi}, $m_+$ satisfies a Talagrand inequality,
so the displayed bound above transfers to
\[
\frac 1N W_2^2(m^N_t, m_+^{\otimes N}) \leqslant C ( N^{-\beta} + e^{-ct} ).
\]
On any compact time interval $[0,T]$, standard global-in-time
Wasserstein propagation of chaos \cite{SznitmanPOC} ensures
\[
\sup_{t \in [0,T]} \frac 1N W_2^2(m^N_t, m_t^{\otimes N}) \leqslant
\frac{C(T)}{N}.
\]
Moreover, under the present hypotheses, \cite{MonmarcheReygnerLocal}
also provides exponential convergence $m_t \to m_+$ in $W_2$
at some rate $c' > 0$, which controls
$N^{-1}W_2^2(m^N_t, m_t^{\otimes N})$
in terms of $N^{-1}W_2^2(m^N_t, m_+^{\otimes N})$
up to an additive error $e^{-c't}$.
Choosing $T \asymp \log N$ to balance these regimes
yields the claimed bound after possibly reducing the exponent $\beta$.
\end{proof}

To the author's knowledge, this is the first result establishing
uniform-in-time propagation of chaos in the presence of multiple
invariant measures.

As a final remark, if $m^N_0 = m_0^{\otimes N}$ but $m_t \to m_-$ instead,
the $N$-particle system is expected to exhibit \emph{metastablility}:
$m^N_t$ closely tracks $m_t^{\otimes N}$ up to
some exponential-in-$N$ time scale, then relaxes slowly towards $m^N_*$.
Addressing such metastable behavior lies beyond the scope of our main results.
For recent progress in this direction, we refer readers to
\textcite[Theorem~3.1]{DelarueTseUniformPOC},
\textcite{JournelLeBrisUniform},
and \textcite{MonmarcheLongTimePOC}.

\section{Conditional approximation}
\label{sec:approx}

In this section, we introduce the conditional approximation,
which is the central technique of this work.
The idea is to decompose the joint law of the $N$-particle system
into a sequence of conditional distributions,
one for each particle given the preceding ones.
This construction can be motivated by the chain rule for entropy,
which expresses the joint entropy of an $N$-particle system
as the sum of conditional entropies;
see \cite[Theorem~2.6]{BudhirajaDupuisAnalysis}.

Let $\nu^N \in \mathcal P_2(\X^N)$ be arbitrary,
and let $\vect X = (X^1,\ldots,X^N)$ be distributed according to $\nu^N$.
Define the filtration $\mathcal F_k = \sigma(X^1,\ldots,X^k)$
for $k \in [N]$, with $\mathcal F_0$ denoting the trivial $\sigma$-algebra.
Introduce the following sequence of random conditional measures:
\[
\nu_k = \Law(X^k | X^1, \ldots, X^{k-1})
= \Law(X^k | \mathcal F_{k-1}),\quad k \in [N].
\]
We denote also:
\[
\bar \nu = \frac 1N \sum_{k \in [N]} \nu_k.
\]
Throughout this section, $U \colon \X \times \X \to \R$
denotes a symmetric kernel function with quadratic growth.

The first lemma shows that $\bar\nu$ approximates
$\mu_{\vect X}$ to order $1/N$ in a weak sense.

\begin{lem}
\label{lem:approx-1}
The following equality holds:
\[
\Expect [ \langle U,
(\mu_{\vect X} - \bar\nu)^{\otimes 2}
\rangle ]
= \frac 1{N^2} \sum_{k \in [N]}
\Expect [\langle U, (\delta_{X^k} - \nu_k)^{\otimes 2}\rangle].
\]
If in addition $U$ is a positive kernel, then
for all $\varepsilon > 0$,
\begin{align*}
\Expect [ \langle U, \mu_{\vect X}^{\otimes 2}\rangle ]
&\leqslant (1+\varepsilon) \Expect [ \langle U, \bar\nu^{\otimes 2}\rangle]
+ \frac{1+\varepsilon^{-1}}{N^2} \sum_{k \in [N]}
\Expect [\langle U, (\delta_{X^k} - \nu_k)^{\otimes 2}\rangle], \\
\Expect [ \langle U, \bar\nu^{\otimes 2}\rangle ]
&\leqslant (1+\varepsilon)
\Expect [ \langle U, \mu_{\vect X}^{\otimes 2}\rangle]
+ \frac{1+\varepsilon^{-1}}{N^2} \sum_{k \in [N]}
\Expect [\langle U, (\delta_{X^k} - \nu_k)^{\otimes 2}\rangle].
\end{align*}
\end{lem}

\begin{proof}[Proof of Lemma~\ref{lem:approx-1}]
Observe that
\[
\mu_{\vect X} - \bar\nu
= \frac 1N \sum_{k\in [N]} (\delta_{X^k} - \nu_k),
\]
and $\delta_{X^k} - \nu_k$ is $\mathcal F_k$-measurable
and verifies the martingale increment condition:
\[
\Expect [\delta_{X^k} - \nu_k | \mathcal F_{k-1}]
= \Expect \bigl[\delta_{X^k} -
\Expect[\delta_{X^k} | \mathcal F_{k-1}] \big| \mathcal F_{k-1}\bigr]
= 0.
\]
In other words, the sequence
\[
M_k = \frac 1N \sum_{\ell \in [k]} (\delta_{X^\ell} - \nu_\ell),
\quad k \in [N],
\]
with $M_0 = 0$,
forms a $(\mathcal F_k)$-martingale in the space of signed measures.
The martingale property yields
\[
\Expect[
(\mu_{\vect X} - \bar\nu)^{\otimes 2}
]
= \Expect[M_N^{\otimes 2}]
= \sum_{k\in[N]} \Expect[(M_k - M_{k-1})^{\otimes 2}]
= \frac 1{N^2} \sum_{k\in [N]} \Expect[(\delta_{X^k} - \nu_k)^{\otimes 2}].
\]
Integrating against $U$ yields the first assertion.
The positivity of $U$ implies, for all $\varepsilon > 0$,
the Cauchy--Schwarz inequality:
\[
\langle U, \mu_{\vect X}^{\otimes 2}\rangle
= \langle U, (\mu_{\vect X} - \bar\nu + \bar\nu)^{\otimes 2}\rangle
\leqslant (1+\varepsilon) \langle U, \bar\nu^{\otimes 2}\rangle
+ (1+\varepsilon^{-1}) \langle U, (\mu_{\vect X} - \bar\nu)^{\otimes 2}\rangle,
\]
which yields the second assertion.
Interchanging $\mu_{\vect X}$ and $\bar\nu$
in the preceding argument proves the third assertion.
\end{proof}

We now present a second lemma, which will be useful in certain contexts.

\begin{lem}
\label{lem:approx-2}
Suppose that $U$ is a positive kernel
and let $\varepsilon > 0$.
Then,
\[
\Expect[\lvert\langle U, \nu_N \otimes (\mu_{\vect X} - \bar\nu)\rangle\rvert]
\leqslant \varepsilon \Expect [ \langle U, \nu_N^{\otimes 2}\rangle ]
+ \frac{1}{4\varepsilon N^2} \sum_{k \in [N]}
\Expect [\langle U, (\delta_{X^k} - \nu_k)^{\otimes 2}\rangle ].
\]
If further the distribution of $\vect X$ is symmetric, then
\begin{IEEEeqnarray*}{rCl}
\frac 1N\sum_{k \in [N]}\Expect[\langle U, \nu_k^{\otimes 2}\rangle]
&\leqslant& \Expect[\langle U, \mu_{\vect X}^{\otimes 2}\rangle]
+ \varepsilon \Expect[ \langle U, \nu_N^{\otimes 2}\rangle] \\
&&\adjustbin + \frac{1}{4\varepsilon N^2} \sum_{k \in [N]}
\Expect [\langle U, (\delta_{X^k} - \nu_k)^{\otimes 2}\rangle ], \\
\frac 1N\sum_{k \in [N]}\Expect[\langle U, \nu_k^{\otimes 2}\rangle]
&\geqslant& \Expect[\langle U, \mu_{\vect X}^{\otimes 2}\rangle]
- \varepsilon \Expect[ \langle U, \nu_N^{\otimes 2}\rangle]
- \frac 1N \Expect[ \langle U, (\delta_{X^N} - \nu_N)^{\otimes 2}\rangle] \\
&&\adjustbin - \frac{1}{4\varepsilon N^2} \sum_{k \in [N]}
\Expect [\langle U, (\delta_{X^k} - \nu_k)^{\otimes 2}\rangle ].
\end{IEEEeqnarray*}
\end{lem}

\begin{proof}[Proof of Lemma~\ref{lem:approx-2}]
The first statement follows by the Cauchy--Schwarz inequality:
\[
\lvert\langle U, \nu_N \otimes (\mu_{\vect X} - \bar\nu)\rangle\rvert
\leqslant \varepsilon \Expect[ \langle U, \nu_N^{\otimes 2}\rangle ]
+ \frac{1}{4\varepsilon}
\Expect [\langle U, (\mu_{\vect X} - \bar\nu)^{\otimes 2}\rangle]
\]
combined with the first claim of Lemma~\ref{lem:approx-1}.

In the case where the distribution of $\vect X$ is symmetric,
the tower rule holds:
\[
\Expect[\nu_N | \mathcal F_{k-1}]
= \Expect\bigl[\Expect[\delta_{X^N} | \mathcal F_{N-1}]
\big|\mathcal F_{k-1}\bigr]
= \Expect[\delta_{X^N} | \mathcal F_{k-1}]
= \Expect[\delta_{X^k} | \mathcal F_{k-1}]
= \nu_k,
\]
which yields
\begin{multline*}
\Expect[ \langle U, \nu_N \otimes \bar\nu\rangle]
= \frac 1N \sum_{k \in [N]}\Expect [\langle U, \nu_N \otimes \nu_k\rangle]
= \frac 1N \sum_{k \in [N]}\Expect \bigl[\langle U,
\Expect[\nu_N|\mathcal F_{k-1}] \otimes \nu_k\rangle\bigr] \\
= \frac 1N \sum_{k \in [N]}\Expect [\langle U, \nu_k^{\otimes 2}\rangle].
\end{multline*}
It follows that
\[
\Expect[\langle U,\nu_k^{\otimes 2}\rangle]
= \Expect[\langle U, \nu_N \otimes \mu_{\vect X}\rangle]
+ \Expect[\langle U, \nu_N \otimes (\bar \nu - \mu_{\vect X})\rangle].
\]
For the first term, we have
\begin{align*}
\Expect[\langle U, \nu_N \otimes \mu_{\vect X}\rangle]
&= \frac{1}{N}
\Expect[\langle U, \nu_N \otimes (\delta_{X^1}+\cdots+\delta_{X^{N-1}})\rangle]
+ \frac 1N \Expect[\langle U, \nu_N \otimes \delta_{X^N}\rangle] \\
&= \frac{1}{N}
\Expect[\langle U, \delta_{X^N}
\otimes (\delta_{X^1}+\cdots+\delta_{X^{N-1}})\rangle]
+ \frac 1N \Expect[\langle U, \nu_N^{\otimes 2}\rangle] \\
&= \Expect[\langle U, \mu_{\vect X}^{\otimes 2}\rangle]
+ \frac 1N \Expect[\langle U, \nu_N^{\otimes 2}\rangle]
- \frac 1N \Expect[\langle U, \delta_{X^N}^{\otimes 2}\rangle] \\
&=\Expect[\langle U, \mu_{\vect X}^{\otimes 2}\rangle]
- \frac 1N \Expect[\langle U, (\delta_{X^N} - \nu_N)^{\otimes 2}\rangle],
\end{align*}
where
$\Expect[\langle U, (\delta_{X^N} - \nu_N)^{\otimes 2}\rangle] \geqslant 0$.
Applying the first assertion of the lemma to the second term
yields the second and third assertion.
\end{proof}

\begin{rem}
Lemma~\ref{lem:approx-2} may initially appear counter-intuitive for the following
reason. For positive $U$, Jensen's inequality yields
\[
\frac 1N \sum_{k \in [N]} \langle U, \nu_k^{\otimes 2}\rangle
\geqslant \biggl\langle U,
\biggl( \frac 1N \sum_{k \in [N]} \nu_k\biggr)^{\!\otimes 2}\biggr\rangle
= \langle U, \bar\nu^{\otimes 2}\rangle.
\]
According to Lemma~\ref{lem:approx-1}, the right-hand side,
when taking expectations, is close to
$\Expect[\langle U, \mu_{\vect X}^{\otimes 2}\rangle]$ up to some error.
Therefore Lemma~\ref{lem:approx-2} effectively reverses Jensen's inequality
at the expense of the additional error term
$\varepsilon \Expect[\langle U, \nu_N^{\otimes 2}\rangle]$.
\end{rem}

In the case where $U$ has bounded double oscillation, the term
\[
\Expect[\langle U, (\delta_{X^k} - \nu_k)^{\otimes 2}\rangle]
\]
appearing in Lemmas~\ref{lem:approx-1}
and \ref{lem:approx-2} can be bounded directly by an $O(1)$ constant.
For kernels satisfying $\nabla_{1,2}^2U \in L^\infty$,
this term can be controlled by entropies via a Talagrand inequality.
The following lemma summarizes these arguments.

\begin{lem}
\label{lem:error-control}
Let Assumption~\ref{assu:T} hold.
For $\X = \T^d$, denote $U_{\bd} = U$, suppose
\[
\Osc_2U_{\bd} \leqslant 4M_U
\]
and set $L_U = 0$.
For $\X = \R^d$, suppose that $U$ admits the decomposition
$U = U_{\bd} + U_{\qd}$ for some positive kernels
$U_{\bd}$, $U_{\qd} \colon \X\times\X\to\R$ such that
\[
\Osc_2U_{\bd} \leqslant 4M_U,
\quad
\lVert\nabla_{1,2}^2U_{\qd}\rVert_{L^\infty} \leqslant L_U.
\]
Then for all $k \in [N]$,
\begin{align*}
\Expect[\langle U, (\delta_{X^k} - \nu_k)^{\otimes 2}]
&\leqslant
4M_U + \frac{4L_U}{\rho} H( \Law(X^k) | m_*)
+ \frac{4L_Ud}{\rho}, \\
\sum_{k \in [N]}
\Expect[\langle U, (\delta_{X^k} - \nu_k)^{\otimes 2}]
&\leqslant
4M_UN + \frac{4L_U}{\rho}H(\nu^N | m_*^{\otimes N})
+ \frac{4L_UdN}{\rho},
\intertext{where $\rho$ is the Talagrand constant
introduced in Assumption~\ref{assu:T}.
Moreover for all $\nu \in \mathcal P_2(\X)$,}
\Expect[\langle U, (\nu-m_*)^{\otimes 2}\rangle]
&\leqslant
\biggl(2M_U + \frac{2L_U}{\rho}\biggr) \Expect[H(\nu | m_*)].
\end{align*}
\end{lem}

\begin{proof}[Proof of Lemma~\ref{lem:error-control}]
By definition, the bounded component satisfies
\[
\langle U_{\bd}, (\delta_{X^k} - \nu_k)^{\otimes 2}\rangle
\leqslant 4M_U.
\]
The $L^2$ projection formula for conditional expectation states that
\[
\Expect [ (\delta_{X^k} - m_*)^{\otimes 2} | \mathcal F_{k-1}]
= \Expect [ (\delta_{X^k} - \nu_k)^{\otimes 2} | \mathcal F_{k-1}]
+ (\nu_k - m_*)^{\otimes 2}.
\]
Integrating both sides against $U_{\qd}$ and using its positivity,
we obtain
\[
\Expect [ \langle U_{\qd}, (\delta_{X^k} - m_*)^{\otimes 2}\rangle]
\geqslant
\Expect[ \langle U_{\qd}, (\delta_{X^k} - \nu_k)^{\otimes 2}\rangle].
\]
Let $Y$ be distributed according to $m_*$ and optimally coupled to $X^k$
with respect to the Euclidean distance, i.e.,
\[
\Expect [ \lvert X^k - Y \rvert^2 ] = W_2^2\bigl( \Law(X^k), m_* \bigr).
\]
By $\lVert \nabla_{1,2}^2U_{\qd} \rVert_{L^\infty} \leqslant L_U$,
\[
\Expect[ \langle U_{\qd}, (\delta_{X^k} - m_*)^{\otimes 2}\rangle ]
\leqslant L_U \Expect[ W_2^2(\delta_{X^k}, m_*) ]
\leqslant 2L_U \bigl(\Expect[ \lvert X^k - Y \rvert^2 ]
+ \Expect[ W_2^2(\delta_{Y}, m_*) ]\bigr).
\]
Talagrand's inequality for $m_*$ yields
\[
\Expect[ \lvert X^k - Y \rvert^2 ]
= W_2^2\bigl(\Law(X^k), m_*\bigr)
\leqslant \frac{2}{\rho} H\bigl( \Law(X^k) \big| m_* \bigr),
\]
while the Poincaré inequality for $m_*$ yields
\[
\Expect[W_2^2(\delta_{Y}, m_*)] = 2 \Var m_*
\leqslant \frac{2d}{\rho}.
\]
Thus for the quadratic component,
\[
\Expect[ \langle U_{\qd}, (\delta_{X^k} - \nu_k)^{\otimes 2}\rangle]
\leqslant \frac{4L_U}{\rho}H\bigl(\Law(X^k) \big| m_* \bigr)
+ 4L_U\Var m_*.
\]
Combining this with the estimate for the bounded component proves
the first assertion.
The second assertion follows by summing the first over $k \in [N]$
and using the subadditivity of entropy
\[
\sum_{k \in [N]} H\bigl( \Law(X^k) | m_* \bigr)
\leqslant H(\nu^N | m_*^{\otimes N}).
\]

Now consider the last assertion.
Let $X$, $Y$ be distributed according to $\nu$, $m_*$, respectively,
and be coupled according to the total variation, namely,
\[
\Expect[ \1_{X \neq Y} ] = \lVert \nu - m_*\rVert_{\TV}.
\]
Let $(X',Y')$ be an independent copy of $(X,Y)$.
The bounded component satisfies
\begin{IEEEeqnarray*}{rCl}
\langle U_{\bd}, (\nu - m_*)^{\otimes 2}\rangle
&=& \Expect[ \langle U_{\bd}, (\delta_X - \delta_Y)
\otimes (\delta_{X'} - \delta_{Y'} ) \rangle ] \\
&\leqslant& 4M_U \Expect [ \1_{X \neq Y} \1_{X' \neq Y'} ] \\
&=& 4M_U \Expect [ \1_{X \neq Y} ] \Expect [ \1_{X' \neq Y'} ] \\
&=& 4M_U \lVert \nu - m_*\rVert_{\TV}^2 \\
&\leqslant& 2 M_U H(\nu | m_*).
\end{IEEEeqnarray*}
The quadratic component satisfies
\[
\Expect[\langle U_{\qd}, (\nu-m_*)^{\otimes 2}\rangle]
\leqslant L_U \Expect[W_2^2(\nu, m_*)]
\leqslant \frac{2L_U}{\rho} \Expect[H(\nu | m_*)].
\]
This establishes the last assertion.
\end{proof}

\section{Coercivity}
\label{sec:coercivity}

This section establishes the coercivity of the modulated free energy
as a direct application of the approximation technique
developed in Section~\ref{sec:approx}.
The proofs of Theorem~\ref{thm:coercivity}
and Corollary~\ref{cor:entropy-coercivity}
are presented in Sections~\ref{sec:proof-coercivity}
and \ref{sec:proof-entropy-coercivity}, respectively.

\subsection{Proof of Theorem~\ref{thm:coercivity}}
\label{sec:proof-coercivity}
Let $\vect X$ be distributed as $\nu^N$
and define $\nu_k = \Law(X^k | \vect X^{[k-1]})$,
$\bar\nu = \frac 1N \sum_{k \in [N]} \nu_k$ as in Section~\ref{sec:approx}.
By rewriting $\mathcal F(\cdot | m_*)$ in \eqref{eq:def-relative-F-centered},
the condition~\eqref{eq:mf-coercivity} becomes
\[
(1-\delta) H(\nu | m_*) + \frac 12 \langle W_*, \nu^{\otimes 2}\rangle
\geqslant 0.
\]
Replacing $\nu$ with $\bar\nu$ and decomposing $W_*$, we obtain
\[
(1-\delta)H(\bar\nu | m_*)
+ \frac{1}{2} \langle W_*^+, \bar\nu^{\otimes 2}\rangle
- \frac{1}{2} \langle W_*^-, \bar\nu^{\otimes 2}\rangle\geqslant 0.
\]
By the chain rule for relative entropy
\cite[Theorem~2.6]{BudhirajaDupuisAnalysis},
\[
H(\nu^N | m_*^{\otimes N})
= \sum_{k \in [N]} \Expect [ H(\nu_k | m_*) ].
\]
On the other hand, convexity of entropy gives
\[
\frac 1N \sum_{k \in [N]} \Expect [ H(\nu_k | m_*) ]
\geqslant \Expect[ H(\bar\nu | m_*) ].
\]
Consequently the following inequality holds for
arbitrary $\varepsilon \in (0,1)$:
\begin{multline}
\label{eq:proof-coercivity-1}
\frac 1NH(\nu^N| m_*^{\otimes N})
+ \frac{1}{2} \Expect[\langle W_*^+, \bar\nu^{\otimes 2}\rangle]
- \frac{1}{2} \Expect[\langle W_*^-, \bar\nu^{\otimes 2}\rangle] \\
\geqslant \frac{(1-\varepsilon) \delta}{N} H(\nu^N | m_*^{\otimes N})
+ \frac{\varepsilon\delta}{N} \Expect[H(\bar\nu | m_*)].
\end{multline}
Let $\varepsilon_1 > 0$ be arbitrary. Lemma~\ref{lem:approx-1} gives
\begin{align*}
\frac{1}{2} \Expect[\langle W_*^+, \bar\nu^{\otimes 2}
- \mu_{\vect X}^{\otimes 2}\rangle]
&\leqslant \frac {\varepsilon_1}2
\Expect[\langle W_*^+, \bar\nu^{\otimes 2}\rangle]
+ \frac{1+\varepsilon_1^{-1}}{2N^2} \sum_{k \in [N]}
\Expect[ \langle W_*^+, (\delta_{X^k} - \nu_k)^{\otimes 2}\rangle], \\
\frac{1}{2} \Expect[\langle W_*^-,
\mu_{\vect X}^{\otimes 2} - \bar\nu^{\otimes 2}\rangle]
&\leqslant \frac {\varepsilon_1}2
\Expect[\langle W_*^-, \bar\nu^{\otimes 2}\rangle]
+ \frac{1+\varepsilon_1^{-1}}{2N^2} \sum_{k \in [N]}
\Expect[ \langle W_*^-, (\delta_{X^k} - \nu_k)^{\otimes 2}\rangle].
\end{align*}
Inserting these inequalities into the right-hand side
of \eqref{eq:proof-coercivity-1} gives
\begin{equation}
\label{eq:proof-coercivity-2}
\begin{IEEEeqnarraybox}[][c]{rCl}
\frac 1N \mathcal F^N(\nu^N | m_*)
&\geqslant& \frac{(1-\varepsilon)\delta}{N} H(\nu^N | m_*^{\otimes N}) \\
&&\adjustbin + \frac{\varepsilon\delta}{N} H(\nu^N | m_*^{\otimes N}) \\
&&\adjustbin - \frac{\varepsilon_1}{2}
\Expect[ \langle W_*^+ + W_*^-, \bar\nu^{\otimes 2}\rangle] \\
&&\adjustbin - \frac{1+\varepsilon_1^{-1}}{2N^2}
\sum_{k \in [N]} \Expect[ \langle W_*^+ + W_*^-,
(\delta_{X^k} - \nu_k)^{\otimes 2}\rangle].
\end{IEEEeqnarraybox}
\end{equation}
Lemma~\ref{lem:error-control} controls the last two terms as follows:
\begin{align*}
\Expect[ \langle W_*^+ + W_*^-, \bar\nu^{\otimes 2}\rangle]
&\leqslant \frac{2(M_W + \ell_W)}{N} H(\nu^N | m_*^{\otimes N}), \\
\sum_{k \in [N]}\Expect[ \langle W_*^+ + W_*^-,
(\delta_{X^k} - \nu_k)^{\otimes 2}\rangle]
&\leqslant 4\bigl(
M_W N + d\ell_W N + \ell_W H(\nu^N | m_*^{\otimes N}) \bigr).
\end{align*}
Substituting these controls into \eqref{eq:proof-coercivity-2}
and selecting
\[
\varepsilon_1 = \frac{\varepsilon \delta}{M_W+\ell_W}
\]
completes the proof.
\qed

\subsection{Proof of Corollary~\ref{cor:entropy-coercivity}}
\label{sec:proof-entropy-coercivity}
The functionals $\mathcal{F}^N(\cdot | m_*)$ and
$H(\cdot | m^N_*)$ differ only by a constant:
\[
\mathcal{F}^N(\nu^N | m_*) - \mathcal{F}^N(m^N_* | m_*)
= H(\nu^N | m^N_*).
\]
Specifically, $m^N_*$ minimizes $\mathcal F^N(\cdot|m_*)$,
and therefore,
\[
\mathcal F^N(m^N_* | m_*) \leqslant \mathcal F^N(m_*^{\otimes N} | m_*).
\]
By definition,
\[
\mathcal F^N(m_*^{\otimes N} | m_*)
= \frac{N}{2} \Expect_{\vect X \sim m_*^{\otimes N}}
[ \langle W_*, \mu_{\vect X}^{\otimes 2}\rangle]
= \frac 12 \Expect_{X \sim m_*}
[\langle W, (\delta_{X} - m_*)^{\otimes 2}\rangle].
\]
Applying Lemma~\ref{lem:error-control} yields
\[
\Expect_{X \sim m_*}
[\langle W, (\delta_{X} - m_*)^{\otimes 2}\rangle]
\leqslant 4M_W + 4\ell_Wd.
\]
Thus we obtain
\[
H(\nu^N | m^N_*) \geqslant \mathcal F^N(\nu^N | m_*)
- 2M_W - 2\ell_Wd.
\]
Combining this inequality with Theorem~\ref{thm:coercivity} yields
the entropy lower bound.

Using Assumption~\ref{assu:T}, we further obtain
\[
H(\nu^N | m^N_*) \geqslant \frac{\delta_N\rho}{2}
W_2^2(\nu^N, m_*^{\otimes N}) - \Delta_{\mathcal F} - 2M_W - 2\ell_Wd.
\]
Plugging in $\nu^N = m^N_*$ yields
\[
\frac{\delta_N\rho}{2} W_2^2(m^N_*, m_*^{\otimes N})
\leqslant \Delta_{\mathcal F} + 2M_W + 2\ell_Wd.
\]
By Cauchy--Schwarz,
\[
W_2^2(\nu^N, m_*^{\otimes N})
\geqslant \frac 12 W_2^2(\nu^N, m_*^N) - 2 W_2^2(m^N_*,m_*^{\otimes N}).
\]
Hence we deduce
\[
H(\nu^N | m^N_*) \geqslant \frac{\delta_N\rho}{4} W_2^2(\nu^N, m^N_*)
- 2(\Delta_{\mathcal F} + 2M_W + 2\ell_Wd).
\]
Applying \cite[Theorem~2]{socgibbs} yields the T$_1$ inequality.
\qed

\section{Contractivity}
\label{sec:contractivity}

This section establishes the contractivity of the
modulated free energy along Langevin dynamics.
Sections~\ref{sec:proof-contractivity-1}
and \ref{sec:proof-contractivity-2} presents the proofs
of Theorems~\ref{thm:contractivity-1} and \ref{thm:contractivity-2},
respectively.
Like the proof of Theorem~\ref{thm:coercivity}, they crucially rely on the
conditional approximation from Section~\ref{sec:approx},
but the contractivity analysis proves substantially more intricate than the
coercivity analysis.
Section~\ref{sec:proof-defective-lsi}
establishes Corollary~\ref{cor:defective-lsi}
as a direct consequence of the two theorems.

Throughout this section, we perform only formal computations
on probability distributions.
These computations can be rigorously justified by approximation arguments;
see \cite[Theorem~2.2]{uklpoc}.

\subsection{Proof of Theorem~\ref{thm:contractivity-1}}
\label{sec:proof-contractivity-1}

The proof is divided into four steps.
The first step employs the conditional approximation introduced
in Section~\ref{sec:approx}.
The second step applies the mean-field contractivity relation
\eqref{eq:mf-contractivity-1} to the mixed measure
\[
\bar\nu = \frac 1N \sum_{ k \in [N] } \nu_k.
\]
The third step symmetrizes the results obtained in the first two steps,
and the fourth step controls the error terms arising in this procedure.

\proofstep{Step 1: Conditional approximation}
We proceed by deriving a lower bound for the relative Fisher information
\begin{equation}
\label{eq:Fisher-expression}
\sum_{k\in [N]}
\int_{\X^N} {\biggl| \nabla_k \log \frac{\nu^N(\x)}
{m_*^N(\x)} \biggr|^2 \, \nu^N(\dd \x)}.
\end{equation}
For $k \in [N]$, introduce the marginal density
\[
\nu^{[k]}(\x^{[k]})
= \int_{\X^{N-k}} \nu^N(\x) \dd \vect x^{\llbracket k+1, N\rrbracket}
\]
and the conditional density
\[
\nu_k(x^k | \x^{[k-1]})
= \frac{\nu^{[k]}(\x^{[k]})}{\nu^{[k-1]}(\x^{[k-1]})}.
\]
We now focus on the final summand in \eqref{eq:Fisher-expression},
corresponding to $k = N$,
and defer the treatment of the remaining ones
to a subsequent symmetrization argument.
This summand equals
\begin{multline*}
\int_{X^{N-1}} \int_{\X}
{\biggl| \nabla_N \log\frac{\nu_N(x^N | \x^{[N-1]})}{m_*(x^N)}
+ \langle \nabla_1 W_*(x^N,\cdot), \mu_{\x}\rangle\biggr|^2} \\
\nu_N(\dd x^N | \x^{[N-1]})\,\nu^{[N-1]}(\dd \x^{[N-1]}),
\end{multline*}
where we used
\begin{align*}
\nabla_N \log \frac{\nu^N(\x)}{m_*^{\otimes N}(\x)}
&= \nabla \log \frac{\nu_N(x^N|\x^{[N-1]})}{m_*(x^N)}, \\
\nabla_N \log \frac{m^N_*(\x)}{m_*^{\otimes N}(\x)}
&= - \langle \nabla_1 W_*(x^N,\cdot), \mu_{\vect x} \rangle.
\end{align*}
Let $\vect X$ be distributed according to $\nu^N$
and define $\nu_k = \Law(X^k | \vect X^{[k-1]})$,
$\bar\nu = \frac 1N \sum_{k \in [N]} \nu_k$ as in Section~\ref{sec:approx}.
Expressing the outer integration over $\vect x^{[N-1]}$ as an expectation,
we obtain
\begin{multline*}
\int_{\X^N}{ \biggl| \nabla_N \log \frac{\nu^N(\x)}
{m_*^N(\x)} \biggr|^2 \, \nu^N(\dd x) }\\
= \Expect \biggl[
\int_{\X} { \biggl\lvert \nabla \log \frac{\nu_N}{m_*}(y)
+ \biggl< \nabla_1 W_*(y,\cdot),
\frac{N-1}{N} \mu_{\vect X^{[N-1]}} + \frac 1N \delta_y \biggr>
\biggr\rvert^2\, \nu_N(\dd y) } \biggr].
\end{multline*}
We aim to perform to the following change in the measure variable:
\[
\frac{N-1}{N} \mu_{\vect X^{[N-1]}} + \frac 1N \delta_y
\to \bar \nu.
\]
The Cauchy--Schwarz inequality
\[
(a + b)^2 \geqslant (1-\varepsilon_1) a^2 - (1+\varepsilon_1^{-1}) b^2,
\qquad \varepsilon_1 \in (0,1]
\]
implies that
\begin{IEEEeqnarray*}{rCl}
\IEEEeqnarraymulticol{3}{l}
{\int_{\X^N} { \biggl| \nabla_N \log \frac{\nu^N(\x)}
{m_*^N(\x)} \biggr|^2\, \nu^N(\dd\x) } } \\
\quad&\geqslant&
(1-\varepsilon_1)
\Expect \biggl[
\int_{\X} { \biggl|\nabla\log\frac{\nu_N(x^N)}{m_*(x^N)}
+\langle\nabla_1 W(x^N,\cdot), \bar\nu \rangle\biggr|^2
\,\nu_N(\dd x^N) } \biggr] \\
&&\adjustbin - (1+\varepsilon_1^{-1})
\Expect[ \int_{\X} { \biggl\lvert
\biggl< \nabla_1 W_*(y,\cdot),
\frac{N-1}{N} \mu_{\vect X^{[N-1]}} + \frac 1N \delta_y - \bar\nu
\biggr> \biggr\rvert^2 \,\nu_N(\dd y) } \biggr].
\end{IEEEeqnarray*}
By the rewriting of $\Pi$ in \eqref{eq:def-Pi-centered},
the term on the second line can be rewritten more concisely as
\[
\Expect[ I(\nu_N | \Pi[\bar\nu])].
\]
As for the last line, we have
\[
\biggl\lvert\biggl< \nabla_1 W_*(y,\cdot),
\frac{N-1}{N} \mu_{\vect X^{[N-1]}} + \frac 1N \delta_y - \bar\nu
\biggr>\biggr\rvert^2
\leqslant
\biggl< R,
\biggl(\frac{N-1}{N} \mu_{\vect X^{[N-1]}} + \frac 1N \delta_y - \bar\nu
\biggr)^{\!\otimes 2}
\biggr>
\]
for each $y \in \X$.
Taking expectations on both sides yields
\begin{IEEEeqnarray*}{rCl}
\IEEEeqnarraymulticol{3}{l}
{\Expect\biggl[ \int_{\X} { \biggl\lvert
\biggl< \nabla_1 W_*(y,\cdot),
\frac{N-1}{N} \mu_{\vect X^{[N-1]}} + \frac 1N \delta_y - \bar\nu
\biggr> \biggr\rvert^2 \,\nu_N(\dd y) } \biggr]} \\
\quad&\leqslant&
\Expect\biggl[ \int_{\X} {
\biggl< R, \biggl( \frac{N-1}{N} \mu_{\vect X^{[N-1]}}
+ \delta_y - \bar\nu \biggr)^{\!\otimes 2}\biggr> \,\nu_N(\dd y) } \biggr] \\
&=& \Expect[\langle R, (\mu_{\vect X} - \bar\nu)^{\otimes 2}\rangle].
\end{IEEEeqnarray*}
Thus we have shown
\begin{equation}
\label{eq:Fisher-lower-bound-nu}
\begin{IEEEeqnarraybox}[][c]{rCl}
\IEEEeqnarraymulticol{3}{l}
{\int_{\X^N} { \biggl| \nabla_N \log \frac{\nu^N(\x)}
{m_*^N(\x)} \biggr|^2 \,\nu^N(\dd\x) } } \\
\quad&\geqslant&
(1-\varepsilon_1) \Expect[I(\nu_N | \Pi[\bar\nu])]
- (1+\varepsilon_1^{-1})
\Expect[ \langle R,
(\mu_{\vect X} - \bar\nu )^{\!\otimes2}
\rangle] \vphantom{\frac 1N}\\
&=& (1-\varepsilon_1)
\Expect[ I( \nu_N | \Pi[\bar\nu] )]
- \frac{(1+\varepsilon_1^{-1})}{N^2}
\sum_{k\in[N]} \Expect[\langle R, (\delta_{X^k} - \nu_k)^{\otimes2}],
\end{IEEEeqnarraybox}
\end{equation}
where the equality follows from Lemma~\ref{lem:approx-1}.
Again by Cauchy--Schwarz,
\begin{align*}
\Expect[ I( \nu_N | \Pi[\bar\nu] )]
&\geqslant \frac 12
\Expect[ I( \nu_N | m_*) ]
- 2 \Expect[ \lvert\langle \nabla_1W_*(X^N, \cdot),
\bar\nu - m_*\rangle\rvert^2] \\
&\geqslant
\frac 12\Expect[ I( \nu_N | m_*) ]
- 2 \Expect[\langle R,(\bar\nu - m_*)^{\otimes 2}\rangle].
\end{align*}
This, together with \eqref{eq:Fisher-lower-bound-nu}, yields
\begin{equation}
\label{eq:Fisher-lower-bound-nu-2}
\begin{IEEEeqnarraybox}[][c]{rCl}
\IEEEeqnarraymulticol{3}{l}
{\int_{\X^N} { \biggl| \nabla_N \log \frac{\nu^N(\x)}
{m_*^N(\x)} \biggr|^2 \,\nu^N(\dd\x) } } \\
\quad&\geqslant& (1-\varepsilon_1-\varepsilon_2)
\Expect[ I(\nu_N | \Pi[\bar\nu])]
+ \frac{\varepsilon_2}{2} \Expect[ I(\nu_N | m_*)] \\
&&\adjustbin - \frac{(1+\varepsilon_1^{-1})}{N^2}
\sum_{k\in[N]} \Expect[\langle R, (\delta_{X^k} - \nu_k)^{\otimes2}]
- 2\varepsilon_2 \Expect[\langle R,(\bar\nu - m_*)^{\otimes 2}\rangle],
\end{IEEEeqnarraybox}
\end{equation}
where $\varepsilon_2 \in (0,1]$.

\proofstep{Step 2: Mixed dissipation}
In this step we focus on the term
\[
\Expect[ I(\nu_N | \Pi[\bar\nu]) ].
\]
Assumption~\ref{assu:LS-Pi} provides the lower bound:
\begin{equation}
\label{eq:proof-contractivity-1-step-2-1}
I(\nu_N | \Pi[\bar\nu])
\geqslant 2\rho_\Pi
H(\nu_N | \Pi[\bar\nu]),
\end{equation}
where the right-hand side can be expressed as follows:
\begin{equation}
\label{eq:proof-contractivity-1-step-2-2}
\begin{IEEEeqnarraybox}[][c]{rCl}
H(\nu_N | \Pi[\bar\nu])
&=& H(\nu_N|m_*)
+ \langle W_*, \nu_N \otimes \bar\nu\rangle \\
&&\adjustbin + \log \int_{\X} \exp
(-\langle W_*(x^*, \cdot), \bar\nu\rangle) \,m_*(\dd x^*).
\end{IEEEeqnarraybox}
\end{equation}

Our goal is to use the condition \eqref{eq:mf-contractivity-1}
to establish a lower bound on the term
\[
\log \int_{\X} \exp
(-\langle W_*(x^*, \cdot), \bar\nu\rangle)\, m_*(\dd x^*).
\]
Expressing both sides of the inequality in \eqref{eq:mf-contractivity-1} as
\begin{align*}
H(\nu | \Pi[\nu])
&= H(\nu | m_*) + \langle W_*, \nu^{\otimes 2}\rangle
+ \log \int_{\X} \exp (-\langle W_*(x^*, \cdot), \nu\rangle)\, m_*(\dd x^*), \\
\mathcal F(\nu | m_*)
&= H(\nu | m_*) + \frac 12 \langle W_*, \nu^{\otimes 2}\rangle,
\end{align*}
we see that \eqref{eq:mf-contractivity-1}
is equivalent to the following inequality:
\[
\log \int \exp (-\langle W_*(x_*,\cdot), \nu\rangle) \,m_*(\dd x_*)
\geqslant - (1-\delta) H(\nu | m_*)
- \biggl(1 - \frac{\delta}{2}\biggr) \langle W_*,\nu^{\otimes 2}\rangle.
\]
Replacing $\nu$ with $\bar\nu$ and using the flat convexity of entropy
\[
\frac 1N \sum_{k\in[N]} H(\nu_k|m_*)
\geqslant H(\bar\nu | m_*),
\]
we obtain
\begin{equation}
\label{eq:proof-contractivity-1-step-2-3}
\begin{IEEEeqnarraybox}[][c]{rCl}
H(\nu_N | \Pi[\bar\nu])
&\geqslant& H(\nu_N|m_*) - \frac{(1-\delta)}N
\sum_{k\in[N]}H(\nu_k|m_*) \\
&&\adjustbin
+ \langle W_*, \nu_N \otimes \bar\nu\rangle
- \biggl(1-\frac{\delta}{2}\biggr)
\langle W_*, \bar\nu^{\otimes 2}\rangle.
\end{IEEEeqnarraybox}
\end{equation}

Combining \eqref{eq:Fisher-lower-bound-nu-2},
\eqref{eq:proof-contractivity-1-step-2-1},
\eqref{eq:proof-contractivity-1-step-2-2} and
\eqref{eq:proof-contractivity-1-step-2-3} yields
\begin{equation}
\label{eq:Fisher-lower-bound-step-2}
\begin{IEEEeqnarraybox}[][c]{rCl}
\IEEEeqnarraymulticol{3}{l}
{\int_{\X^N} { \biggl| \nabla_N \log \frac{\nu^N(\x)}
{m_*^N(\x)} \biggr|^2 \,\nu^N(\dd\x) } } \\
\quad&\geqslant& 2(1-\varepsilon_1-\varepsilon_2)\rho_\Pi
\biggl( \Expect[ H(\nu_N | m_*) ]
- \frac{1-\delta}{N} \sum_{k \in [N]}
\Expect[ H(\nu_k | m_*) ] \biggr) \\
&&\adjustbin + \frac{\varepsilon_2}{2} \Expect[ I(\nu_N | m_*)] \\
&&\adjustbin + 2(1-\varepsilon_1-\varepsilon_2)\rho_\Pi
\biggl(\Expect[ \langle W_*, \nu_N \otimes \bar\nu \rangle]
- \Bigl(1 - \frac\delta 2\Bigr) \sum_{k \in [N]}
\Expect[ H(\nu_k | m_*) ] \biggr)\\
&&\adjustbin - \frac{(1+\varepsilon_1^{-1})}{N^2}
\sum_{k\in[N]} \Expect[\langle R, (\delta_{X^k} - \nu_k)^{\otimes2}]
- 2\varepsilon_2 \Expect[\langle R,(\bar\nu - m_*)^{\otimes 2}\rangle].
\end{IEEEeqnarraybox}
\end{equation}

\proofstep{Step 3: Symmetrization}
In this step we apply a symmetrization procedure to the particle indices.
This procedure derives a lower bound
for the other summands in \eqref{eq:Fisher-expression}
and also for the following entropy and energy terms
\begin{equation}
\label{eq:proof-contractivity-1-step-3-two-terms}
H(\nu_N|m_*) - \frac{(1-\delta)}N
\sum_{k\in[N]}H(\nu_k|m_*),
\qquad
\langle W_*, \nu_N \otimes \bar\nu\rangle.
\end{equation}

Let $S_N$ denote the permutation group acting on the particle indices
and define, for $\sigma \in S_N$, the reordered particle configuration:
\[
\vect X^{\sigma} = (X^{\sigma(1)},\ldots,X^{\sigma(N)}).
\]
Introduce the reordered version of the conditional approximation
\begin{align*}
\nu^\sigma_k &\coloneqq
\Law(X^{\sigma(k)} | X^{\sigma(1)}, \ldots, X^{\sigma(k-1)}) \\
\bar\nu^\sigma &\coloneqq
\frac 1N \sum_{k \in [N]} \nu^\sigma_k.
\end{align*}
For convenience we use the following notation
for averaging over these permutations:
\[
\permsum \cdots
\coloneqq \frac 1{N!} \sum_{\sigma \in S_N} \cdots.
\]

First consider the entropy term in
\eqref{eq:proof-contractivity-1-step-3-two-terms}.
Define the averaged entropy at level $k$:
\[
h_k = \permsum \Expect[H(\nu^\sigma_k|m_*)].
\]
By the convexity of the entropy functional,
\[
\Expect \bigl[ H\bigl( \Law(X^{\sigma(k+1)} | \vect X^{\sigma([k])})
\big| m_*\bigr)\bigr]
\geqslant
\Expect \bigl[ H\bigl( \Law(X^{\sigma(k+1)} | \vect X^{\sigma([k-1])})
\big| m_*\bigr)\bigr].
\]
Averaging this over all permutations gives
\[
h_{k+1} \geqslant h_k.
\]
Moreover, by averaging the chain rule for relative entropy
\[
\sum_{k\in[N]} \Expect [ H(\nu^\sigma_k | m_*) ]
= H(\nu^N | m_*^{\otimes N}),
\]
we obtain
\[
\sum_{k\in[N]} h_k = H(\nu^N | m_*^{\otimes N}).
\]
Therefore the entropy term, when averaged over all permutations, satisfies
\begin{multline}
\permsum \Expect\biggl[
H(\nu^\sigma_N | m_*)
- \frac{1-\delta}{N} \sum_{k \in [N]}
H(\nu^\sigma_k | m_*)
\biggr] \\
= Nh_N - (1-\delta) \sum_{k \in [N]} h_k
\geqslant
\delta H(\nu^N | m_*^{\otimes N}).
\label{eq:proof-contractivity-1-step-3-entropy}
\end{multline}

Then consider the energy term in
\eqref{eq:proof-contractivity-1-step-3-two-terms}.
We decompose it into positive and negative components:
\[
\langle W_*, \nu_N \otimes \bar\nu\rangle
= \langle W_*^+ - W_*^-, \nu_N \otimes \bar \nu\rangle.
\]
Applying Lemma~\ref{lem:approx-2} each component gives
\begin{IEEEeqnarray*}{rCl}
\Expect[ \langle W_*, \nu_N \otimes \bar\nu\rangle ]
&\geqslant& \Expect[ \langle W_*, \nu_N \otimes \mu_{\vect X}\rangle]
- \varepsilon_3 \Expect[ \langle W_*^+ + W_*^-, \nu_N^{\otimes 2}\rangle] \\
&&\adjustbin - \frac{1}{4\varepsilon_3N^2}
\sum_{k \in [N]} \Expect[ \langle W_*^+ + W_*^-,
(\delta_{X^k} - \nu_k)^{\otimes 2}\rangle],
\end{IEEEeqnarray*}
for all $\varepsilon_3 > 0$.
Observe that
\[
\Expect[ \langle W_*, \nu_N \otimes \mu_{\vect X}\rangle]
= \Expect[ \langle W_*, \delta_{X^N} \otimes \mu_{\vect X}\rangle]
- \frac 1N\Expect[ \langle W_*, (\delta_{X^N} - \nu_N)^{\otimes 2}\rangle],
\]
Averaging the first term on the right recovers the interaction energy:
\[
\permsum
\Expect[ \langle W_*, \delta_{X^{\sigma(N)}} \otimes \mu_{\vect X}\rangle]
= \Expect[ \langle W_*, \mu_{\vect X}^{\otimes 2}\rangle].
\]
Therefore, averaging the energy term, we find
\begin{equation}
\begin{IEEEeqnarraybox}[][c]{rCl}
\permsum
\Expect[\langle W_*, \nu_N \otimes \bar\nu\rangle]
&\geqslant& \Expect[ \langle W_*, \mu_{\vect X}^{\otimes 2}\rangle] \\
&&\adjustbin- \varepsilon_3\permsum
\Expect[ \langle W_*^+ + W_*^-, (\nu_N^\sigma)^{\otimes 2}\rangle ] \\
&&\adjustbin- \frac{1}{4\varepsilon_3N^2}\permsum
\sum_{k \in [N]} \Expect[ \langle W_*^+ + W_*^-,
(\delta_{X^{\sigma(k)}} - \nu^\sigma_k)^{\otimes 2}\rangle] \\
&&\adjustbin- \frac 1N\permsum
\Expect[ \langle W_*, (\delta_{X^{\sigma(N)}} - \nu^\sigma_N)
^{\otimes 2}\rangle].
\end{IEEEeqnarraybox}
\label{eq:proof-contractivity-1-step-3-energy}
\end{equation}

Symmetrizing the final Fisher summand recovers the full information:
\[
\permsum \int_{\X^N}
{\biggl| \nabla_{\sigma(N)}
\log \frac{\nu^N(\vect x)}{m^N_*(\vect x)} \biggr|^2 \,\nu^N(\dd\vect x)}
= \frac 1N I(\nu^N | m^N_*).
\]
Thus, averaging \eqref{eq:Fisher-lower-bound-step-2}
and applying the lower bounds
\eqref{eq:proof-contractivity-1-step-3-entropy}
and \eqref{eq:proof-contractivity-1-step-3-energy}, we obtain
\begin{equation}
\label{eq:Fisher-lower-bound-step-3}
\begin{IEEEeqnarraybox}[][c]{rCl}
\IEEEeqnarraymulticol{3}{l}
{\frac 1N I(\nu^N|m_*^N)} \\
\quad&\geqslant& \frac{2(1-\varepsilon_1-\varepsilon_2)\delta\rho_\Pi}{N}
H(\nu^N|m_*^{\otimes N})
\\
&&\adjustbin + \frac{\varepsilon_2}{2} \Expect[ I(\nu_N | m_*)] \\
&&\adjustbin + 2(1-\varepsilon_1-\varepsilon_2)\rho_\Pi
\biggl( \Expect[ \langle W_*, \mu_{\vect X}^{\otimes 2}\rangle ]
- \Bigl( 1 - \frac\delta 2\Bigr)
\permsum \Expect[ \langle W_*, (\bar\nu^\sigma)^{\otimes 2}\rangle ]
\biggr) \\
&&\adjustbin- 2\varepsilon_3\rho_\Pi\permsum
\Expect[ \langle W_*^+ + W_*^-, (\nu_N^\sigma)^{\otimes 2}\rangle ] \\
&&\adjustbin-
2\varepsilon_2 \permsum\Expect[\langle R,
(\bar\nu^\sigma - m_*)^{\otimes 2}\rangle] \\
&&\adjustbin - \frac{(1+\varepsilon_1^{-1})}{N^2}
\sum_{k\in[N]} \Expect[\langle R, (\delta_{X^k} - \nu_k)^{\otimes2}] \\
&&\adjustbin- \frac{\rho_\Pi}{2\varepsilon_3N^2}\permsum
\sum_{k \in [N]} \Expect[ \langle W_*^+ + W_*^-,
(\delta_{X^{\sigma(k)}} - \nu^\sigma_k)^{\otimes 2}\rangle] \\
&&\adjustbin- \frac {2\rho_\Pi}N\permsum
\Expect[ \langle W_*, (\delta_{X^{\sigma(N)}} - \nu^\sigma_N)
^{\otimes 2}\rangle].
\end{IEEEeqnarraybox}
\end{equation}

\proofstep{Step 4: Error control}
Lemma~\ref{lem:error-control} gives
\begin{multline*}
2\varepsilon_3\rho_\Pi \permsum \Expect[
\langle W_*^+ + W_*^-, (\nu^\sigma_N)^{\otimes 2}\rangle]
\leqslant 4\varepsilon_3\rho_\Pi (M_W + \ell_W)
\Expect[H(\nu_N | m_*) ] \\
\leqslant \frac{2\varepsilon_3(M_W+\ell_W)\rho_\Pi}{\rho}
\Expect[I(\nu_N | m_*) ].
\end{multline*}
Thus by choosing
\begin{equation}
\label{eq:proof-contractivity-varepsilon-3}
\varepsilon_3
= \frac{\rho\varepsilon_2}{4(M_W+\ell_W)\rho_\Pi},
\end{equation}
we ensure
\[
2\varepsilon_3\rho_\Pi\permsum
\Expect[\langle W_*^+ + W_*^-, (\nu^\sigma_N)^{\otimes 2}\rangle]
\leqslant \frac{\varepsilon_2}{2} \Expect[ I(\nu|m_*) ].
\]

In \eqref{eq:Fisher-lower-bound-step-3},
if we could identify $\bar\nu^\sigma$ with $\mu_{\vect X}$
and neglect the error terms arising from the conditional approximation,
the proof would be complete.
Therefore it remains to rigorously control these error terms.

For simplicity we denote
\begin{IEEEeqnarray*}{rCl}
\mathcal F' &=& \frac \delta N H(\nu^N | m_*^{\otimes N})
+ \Expect[ \langle W_*, \mu_{\vect X}^{\otimes 2}\rangle]
- \biggl(1 - \frac\delta2\biggr) \permsum \Expect[\langle W_*,
(\nu^\sigma_N)^{\otimes 2}\rangle] \\
E_1 &=& 2\varepsilon_2 \permsum\Expect[\langle R,
(\bar\nu^\sigma - m_*)^{\otimes 2}\rangle] \\
&&\adjustbin+ \frac{(1+\varepsilon_1^{-1})}{N^2}
\sum_{k\in[N]} \Expect[\langle R, (\delta_{X^k} - \nu_k)^{\otimes2}] \\
&&\adjustbin+ \frac{\rho_\Pi}{2\varepsilon_3N^2}\permsum
\sum_{k \in [N]} \Expect[ \langle W_*^+ + W_*^-,
(\delta_{X^{\sigma(k)}} - \nu^\sigma_k)^{\otimes 2}\rangle] \\
&&\adjustbin+ \frac {2\rho_\Pi}N\permsum
\Expect[ \langle W_*, (\delta_{X^{\sigma(N)}} - \nu^\sigma_N)
^{\otimes 2}\rangle].
\end{IEEEeqnarray*}
Thus under the choice \eqref{eq:proof-contractivity-varepsilon-3},
the inequality \eqref{eq:Fisher-lower-bound-step-3} becomes
\begin{equation}
\label{eq:proof-contractivity-1-step-4-1}
\frac 1N I(\nu^N | m^N_*)
\geqslant 2(1-\varepsilon_1-\varepsilon_2)\rho_\Pi \mathcal F' - E_1.
\end{equation}
Let $\varepsilon_4 > 0$.
We decompose $\mathcal F'$ into three terms:
\begin{IEEEeqnarray*}{rCl}
\mathcal F' &=& \frac{(1-\varepsilon_4)\delta}{N} \mathcal F^N(\nu^N | m_*) \\
&&\adjustbin +\varepsilon_4 \delta \biggl( \frac 1N H(\nu^N | m^N_*)
+ \frac 12 \permsum \Expect[ \langle W_*, (\bar\nu^\sigma)^{\otimes 2}\rangle]
\biggr) \\
&&\adjustbin + \biggl( 1 - \frac{(1-\varepsilon_4)\delta}{2}\biggr)
\Expect\biggl[ \biggl\langle W_*, \mu_{\vect X}^{\otimes 2}
- \permsum (\bar\nu^\sigma)^{\otimes 2}\biggr\rangle \biggr],
\end{IEEEeqnarray*}
and denote further:
\[
f = \frac 1N H(\nu^N | m^N_*)
+ \frac 12 \permsum \Expect[ \langle W_*, (\bar\nu^\sigma)^{\otimes 2}\rangle].
\]
Separating the positive and negative components of $W_*$
and applying Lemma~\ref{lem:approx-1},
we bound the last term as follows:
\begin{IEEEeqnarray*}{rCl}
\IEEEeqnarraymulticol{3}{l}{
\biggl\lvert\biggl( 1 - \frac{(1-\varepsilon_4)\delta}{2}\biggr)
\Expect\biggl[ \biggl\langle W_*, \mu_{\vect X}^{\otimes 2}
- \permsum (\bar\nu^\sigma)^{\otimes 2}\biggr\rangle \biggr]\biggr\rvert} \\
\quad&\leqslant&
\sum_{s = \pm}
{\biggl\lvert\Expect\biggl[ \biggl\langle W_*^s, \mu_{\vect X}^{\otimes 2}
- \permsum (\bar\nu^\sigma)^{\otimes 2}\biggr\rangle \biggr]\biggr\rvert} \\
&\leqslant& \varepsilon_5
\permsum \Expect[ \langle W_*^+ + W_*^-, (\bar\nu^{\sigma})^{\otimes 2}\rangle]
\\
&&\adjustbin + \frac{1+\varepsilon_5^{-1}}{N^2}
\permsum\sum_{k \in [N]}
\Expect[ \langle W_*^+ + W_*^-,
(\delta_{X^{\sigma(k)}} - \nu^\sigma_k)^{\otimes 2}\rangle] \\
&\eqqcolon& E_2,
\end{IEEEeqnarray*}
where $\varepsilon_5 > 0$ is arbitrary.
Inserting this estimate into \eqref{eq:proof-contractivity-1-step-4-1} yields
\begin{equation}
\label{eq:proof-contractivity-1-step-4-2}
\frac 1NI(\nu^N | m^N_*)
\geqslant \frac{2(1-\varepsilon_1-\varepsilon_2-\varepsilon_4)\delta\rho_\Pi}
{N} \mathcal F^N(\nu^N | m_*)
+ 2\varepsilon_4\delta \rho_\Pi f
- E_1 - 2\rho_\Pi E_2.
\end{equation}

The terms $E_1$ and $E_2$ now constitute the final errors.
According to Lemma~\ref{lem:error-control}, they satisfy
\begin{IEEEeqnarray*}{rCl}
E_1 &\leqslant& 4\varepsilon_2(\mu_R + \ell_R)\rho
\permsum\Expect[ H(\bar\nu^{\sigma} | m_*)] \\
&&\adjustbin + \frac{4(1+\varepsilon_1^{-1})\rho}{N}
\biggl(\mu_R + \frac{\ell_R}{N} H(\nu^N | m_*^{\otimes N})
+ \ell_Rd\biggr) \\
&&\adjustbin + \frac{2\rho_\Pi}{\varepsilon_3 N}
\biggl(M_W + \frac{\ell_W}{N} H(\nu^N | m_*^{\otimes N}) + \ell_Wd\biggr) \\
&&\adjustbin+ \frac{8\rho_\Pi}{N}
\biggl(M_W + \frac{\ell_W}{N} H(\nu^N|m^N_*) + \ell_Wd\biggr), \\
E_2 &\leqslant&
2\varepsilon_5(M_W + \ell_W) \permsum\Expect[ H(\bar\nu^\sigma | m_*)] \\
&&\adjustbin
+ \frac{4(1+\varepsilon_5^{-1})}{N}
\biggl(M_W + \frac{\ell_W}{N}
H(\nu^N | m_*^{\otimes N}) + \ell_Wd\biggr).
\end{IEEEeqnarray*}

The flat convexity of the entropy functional implies
\[
\permsum \Expect[ H(\bar\nu^\sigma | m_*) ]
\leqslant \frac 1N\permsum\sum_{k \in [N]}
\Expect[ H(\nu^\sigma_k | m_*) ]
= \frac 1N H(\nu^N | m_*^{\otimes N}).
\]
Replacing $\nu$ with $\bar\nu^\sigma$ in \eqref{eq:mf-coercivity} gives
\[
(1-\delta)H(\bar\nu^\sigma | m_*)
+ \frac 12 \langle W_*, (\bar\nu^\sigma)^{\otimes 2}\rangle \geqslant 0.
\]
Taking the expectation and averaging over permutations, we obtain
\begin{multline*}
\frac{1-\delta}{N}
H(\nu^N | m_*^{\otimes N})
+ \frac 12 \permsum\Expect[\langle W_*, (\bar\nu^\sigma)^{\otimes 2}\rangle]\\
\geqslant
\permsum\biggl(
(1-\delta) \Expect[H(\bar\nu^\sigma|m_*)]
+ \frac 12 \Expect[\langle W_*, (\bar\nu^\sigma)^{\otimes 2}\rangle]
\biggr) \geqslant 0.
\end{multline*}
This is equivalent to the following:
\[
\frac 1NH(\nu^N | m_*^{\otimes N}) \leqslant
\delta^{-1} \biggl( \frac 1N H(\nu^N|m_*^{\otimes N})
+ \permsum\frac 12 \langle W_*, (\bar\nu^\sigma)^{\otimes 2}\rangle \biggr)
= \delta^{-1} f.
\]
Thus we have shown
\[
\permsum \Expect[ H(\bar\nu^\sigma | m_*) ]
\leqslant \frac 1N H(\nu^N | m_*^{\otimes N})
\leqslant \delta^{-1}f.
\]

Applying the above inequality to the upper bounds for $E_1$ and $E_2$, we obtain
\begin{IEEEeqnarray*}{rCl}
E_1 &\leqslant&
\biggl(4\varepsilon_2(\mu_R+\ell_R)\rho
+ \frac{4(1+\varepsilon_1^{-1})\ell_R\rho}{N}
+ \frac{2(4+\varepsilon_3^{-1})\ell_W\rho_\Pi}{N}
\biggr) \delta^{-1} f \\
&&\adjustbin
+ \frac{1}{N}
\biggl( 4(1+\varepsilon_1^{-1})(\mu_R+\ell_Rd)\rho
+ 2(4+\varepsilon_3^{-1})(M_W + \ell_Wd)\rho_\Pi \biggr), \\
E_2 &\leqslant&
\biggl(2\varepsilon_5(M_W+\ell_W)
+ \frac{4(1+\varepsilon_5^{-1})\ell_W}{N}\biggr)
\delta^{-1} f
 \\
&&\adjustbin + \frac{4(1+\varepsilon_5^{-1})}{N}(M_W+\ell_Wd),
\end{IEEEeqnarray*}
Let $\varepsilon \in (0,1)$.
Choose the remaining constants as follows:
\begin{IEEEeqnarray*}{rCl}
\varepsilon_1 &=& \frac{\varepsilon}{3}, \\
\varepsilon_2 &=&
\frac{\varepsilon}
{3\bigl(1+2(\mu_R+\ell_R)\rho/\delta^2\rho_\Pi\bigr)}, \\
\varepsilon_5 &=& \frac{\delta^2\varepsilon}{6(M_W + \ell_W)}, \\
\varepsilon_4 &=&
\delta^{-2}\biggl( \frac{2(\mu_R+\ell_R)\rho\varepsilon_2}{\rho_\Pi}
+ \frac{2(1+\varepsilon_1^{-1})\rho\ell_R}{\rho_\Pi N}
+\frac{(4+\varepsilon_3^{-1})\ell_W}{N}
\biggr)\\
&&\adjustbin+ \delta^{-2}\biggl(2(M_W+\ell_W)\varepsilon_5
+ \frac{4(1+\varepsilon_5^{-1})\ell_W}{N} \biggr).
\end{IEEEeqnarray*}
Under these choices,
\begin{IEEEeqnarray*}{rCl}
\varepsilon_1+\varepsilon_2+\varepsilon_4
&\leqslant& \varepsilon
+ \frac{2(1+\varepsilon_1^{-1})\rho\ell_R}{\rho_\Pi N}
+ \frac{(8+\varepsilon_3^{-1}+4\varepsilon_5^{-1})\ell_W}{\delta^2N}, \\
\varepsilon_4 \delta f - \frac{E_1}{2\rho_\Pi} - E_2
&\geqslant& - \frac{2(1+\varepsilon_1^{-1})\rho}{\rho_\Pi N}(\mu_R + \ell_Rd)
- \frac{8+\varepsilon_3^{-1}+4\varepsilon_5^{-1}}{N}(M_W+\ell_Wd).
\end{IEEEeqnarray*}
Substituting the expressions for $\varepsilon_i$ in
\eqref{eq:proof-contractivity-1-step-4-2}
yields the desired result.
\qed

\subsection{Proof of Theorem~\ref{thm:contractivity-2}}
\label{sec:proof-contractivity-2}

The proof is divided into three steps.
The first step starts from the conclusion of Step~1 in the proof of
Theorem~\ref{thm:contractivity-1}
and applies the tower rule to extract the graded dissipation:
\[
\frac{1}{N}\sum_{k \in [N]} \Expect[ I(\nu_k | \Pi[\nu_k] ) ].
\]
The second step invokes \eqref{eq:mf-contractivity-2} for each $\nu_k$
and establishes the contractivity up to error terms.
The third step bounds these errors, thereby completing the proof.

\proofstep{Step 1: Tower rule}
Since $\nu^N$ is symmetric,
\[
\frac 1N I(\nu^N|m^N_*)
= \int_{\X} { \biggl\lvert \nabla_N \log \frac{\nu^N}{m^N_*} (\x) \biggr\rvert^2
\,\nu^N(\dd\x) }.
\]
The inequality \eqref{eq:Fisher-lower-bound-nu-2}, established
in Step~1 of the proof of Theorem~\ref{thm:contractivity-1}, becomes
\begin{equation}
\label{eq:proof-contractivity-2-Fisher-lower-bound}
\begin{IEEEeqnarraybox}[][c]{rCl}
\IEEEeqnarraymulticol{3}{l}{\frac 1N I(\nu^N|m^N_*)} \\
\quad&\geqslant& (1-\varepsilon_1-\varepsilon_2)
\Expect[ I(\nu_N | \Pi[\bar\nu])]
+ \frac{\varepsilon_2}{2} \Expect[ I(\nu_N | m_*)] \\
&&\adjustbin - \frac{(1+\varepsilon_1^{-1})}{N^2}
\sum_{k\in[N]} \Expect[\langle R, (\delta_{X^k} - \nu_k)^{\otimes2}]
- 2\varepsilon_2 \Expect[\langle R,(\bar\nu - m_*)^{\otimes 2}\rangle].
\end{IEEEeqnarraybox}
\end{equation}
Here we extract contractivity from $I(\nu_N | \Pi[\bar\nu])$
using a different approach from the proof of Theorem~\ref{thm:contractivity-1}.
By definition, this term equals
\begin{multline}
\label{eq:proof-contractivity-2-step-1-1}
I(\nu_N | m_*)
+ 2 \int_{\X} \nabla \log \frac{\nu_N}{m_*}(y)
\cdot \langle \nabla_1 W_*(y,\cdot), \bar\nu \rangle \,\nu_N(\dd y) \\
+ \int_{\X} { \lvert \langle \nabla_1 W_*(y,\cdot), \bar\nu\rangle
\rvert^2 \,\nu_N(\dd y) },
\end{multline}
where we identified the conditional measure $\nu_N$ with its density.
In the symmetric setting, the tower rule holds:
\[
\Expect[\nu_N | \mathcal F_{k-1}] = \nu_k,
\]
By the flat convexity of the Fisher information,
the first term in \eqref{eq:proof-contractivity-2-step-1-1} satisfies
\begin{equation}
\label{eq:proof-contractivity-2-step-1-2}
\Expect[I(\nu_N | m_*)] \geqslant \frac 1N \sum_{k \in [N]}
\Expect[I(\nu_k | m_*)]
\end{equation}
The second term in \eqref{eq:proof-contractivity-2-step-1-1} satisfies
\begin{IEEEeqnarray*}{rCl}
\IEEEeqnarraymulticol{3}{l}
{\int_{\X} \nabla \log \frac{\nu_N}{m_*}(y)
\cdot \langle \nabla_1 W_*(y,\cdot), \bar\nu \rangle\,\nu_N(\dd y)} \\
\quad&=& \frac 1N \sum_{k \in [N]}
\int_{\X} \nabla \log \frac{\nu_N}{m_*}(y)
\cdot \langle \nabla_1 W_*(y,\cdot), \nu_k\rangle \, \nu_N(\dd y)
\\
&=& \frac 1N \sum_{k\in[N]}
\int_{\X} \nabla \nu_N(y) \cdot \langle \nabla_1 W_*(y,\cdot),
\nu_k\rangle \dd y \\
&&\adjustbin - \frac 1N \sum_{k \in [N]}
\int_{\X} \nabla \log m_*(y) \cdot \langle \nabla_1 W_*(y,\cdot),
\nu_k\rangle \nu_N(y) \dd y.
\end{IEEEeqnarray*}
Both integrals at the end are linear in $\nu_N$.
Hence by the tower rule,
\begin{multline}
\label{eq:proof-contractivity-2-step-1-3}
\Expect\biggl[
\int_{\X} \nabla \log \frac{\nu_N}{m_*}(y)
\cdot \langle \nabla_1 W_*(y,\cdot), \bar\nu \rangle \,\nu_N(\dd y)
\biggr] \\
= \frac 1N \sum_{k \in [N]}
\Expect\biggl[
\int_{\X} \nabla \log \frac{\nu_k}{m_*}(y)
\cdot \langle \nabla_1 W_*(y,\cdot), \nu_k \rangle \,\nu_k(\dd y)
\biggr].
\end{multline}
Inserting \eqref{eq:proof-contractivity-2-step-1-2}
and \eqref{eq:proof-contractivity-2-step-1-3}
in \eqref{eq:proof-contractivity-2-step-1-1},
we obtain
\begin{IEEEeqnarray*}{rCl}
\Expect[ I(\nu_N | \Pi[\bar\nu]) ]
&\geqslant& \frac 1N \sum_{k \in [N]} \Expect[ I(\nu_k | m_*)] \\
&&\adjustbin + \frac 2N \sum_{k \in [N]}
\Expect\biggl[
\int_{\X} \nabla \log \frac{\nu_k}{m_*}(y)
\cdot \langle \nabla_1 W_*(y,\cdot), \nu_k \rangle \,\nu_k(\dd y)
\biggr] \\
&&\adjustbin+
\Expect\biggl[\int_{\X} { \lvert \langle \nabla_1 W_*(y,\cdot), \bar\nu\rangle
\rvert^2 \,\nu_N(\dd y) }\biggr].
\end{IEEEeqnarray*}
The first two terms above appear in the sum
$\frac 1N \sum_{k \in [N]} \Expect[ I(\nu_k | \Pi[\nu_k])]$,
since for each $k\in[N]$,
\begin{IEEEeqnarray*}{rCl}
{\vphantom{\int_{\X}}}
I(\nu_k | \Pi[\nu_k] )
&=& I(\nu_k | m_*) \\
&&\adjustbin +2
\int_{\X} \nabla \log \frac{\nu_k}{m_*}(y)
\cdot \langle \nabla_1 W_*(y,\cdot), \nu_k \rangle \,\nu_k(\dd y)
\\
&&\adjustbin+
\int_{\X} { \lvert \langle \nabla_1 W_*(y,\cdot), \nu_k\rangle
\rvert^2 \,\nu_N(\dd y) }.
\end{IEEEeqnarray*}
We can complete this square:
\begin{equation}
\label{eq:proof-contractivity-2-step-1-4}
\begin{IEEEeqnarraybox}[][c]{rCl}
\IEEEeqnarraymulticol{3}{l}
{\vphantom{\biggl[}\Expect[ I(\nu_k | \Pi[\bar\nu]) ]} \\
\quad&=& \frac 1N \sum_{k \in [N]}
\Expect[ I(\nu_k | \Pi[\nu_k] ) ] \\
&&\adjustbin
+ \Expect \biggl[
\int_{\X} \biggl(
\lvert \langle \nabla_1 W_*(y,\cdot), \bar\nu\rangle \rvert^2
- \frac 1N \sum_{k \in [N]}
\lvert \langle \nabla_1 W_*(y,\cdot), \nu_k\rangle \rvert^2
\biggr) \,\nu_N(\dd y) \biggr].
\end{IEEEeqnarraybox}
\end{equation}

It remains to study the last term in \eqref{eq:proof-contractivity-2-step-1-4}.
By Assumption~\ref{assu:R}, for each $y \in \X$,
\begin{IEEEeqnarray*}{rCl}
\IEEEeqnarraymulticol{3}{l}
{\lvert\langle \nabla_1 W_*(y, \cdot), \bar\nu\rangle\rvert^2
- \frac 1N \sum_{k \in [N]}
\lvert\langle\nabla_1 W_*(y,\cdot), \nu_k\rangle\rvert^2} \\
\quad&=&
- \frac 12 \sum_{k,\ell\in [N]}
\lvert\langle \nabla_1 W_*(y,\cdot),
\nu_k - \nu_\ell\rangle\rvert^2 \\
&\geqslant&
- \frac 12 \sum_{k,\ell\in [N]}
\langle R, (\nu_k - \nu_\ell)^{\otimes 2}\rangle \\
&=& \langle R, \bar\nu^{\otimes 2}\rangle
- \frac 1N\sum_{k \in [N]} \langle R, \nu_k^{\otimes 2}\rangle,
\end{IEEEeqnarray*}
whereas by Lemmas~\ref{lem:approx-2} and \ref{lem:approx-1},
for all $\varepsilon_3$, $\varepsilon_4 > 0$,
\begin{multline*}
\frac 1N \sum_{k \in [N]}
\Expect[ \langle R, \nu_k^{\otimes 2}\rangle]
\leqslant (1+\varepsilon_4) \Expect[ \langle R, \bar\nu^{\otimes 2}\rangle]
+ \varepsilon_3 \Expect[ \langle R, \nu_N^{\otimes 2}\rangle] \\
+ \frac{1+\varepsilon_4^{-1}+(4\varepsilon_3)^{-1}}{N^2}
\sum_{k \in [N]}
\Expect[ \langle R, (\delta_{X^k} - \nu_k)^{\otimes 2}\rangle ].
\end{multline*}
Consequently, the last term in \eqref{eq:proof-contractivity-2-step-1-4}
satisfies
\begin{IEEEeqnarray*}{rCl}
\IEEEeqnarraymulticol{3}{l}
{\Expect \biggl[
\int_{\X} \biggl(
\lvert \langle \nabla_1 W_*(y,\cdot), \bar\nu\rangle \rvert^2
- \frac 1N \sum_{k \in [N]}
\lvert \langle \nabla_1 W_*(y,\cdot), \nu_k\rangle \rvert^2
\biggr) \,\nu_N(\dd y) \biggr]} \\
\quad&\geqslant&
- \varepsilon_4 \Expect[ \langle R, \bar\nu^{\otimes 2}\rangle ]
- \varepsilon_3 \Expect[ \langle R, \nu_N^{\otimes 2}\rangle]
\vphantom{\biggl[} \\
&&\adjustbin- \frac{1+\varepsilon_4^{-1}+(4\varepsilon_3)^{-1}}{N^2}
\sum_{k \in [N]}
\Expect[ \langle R, (\delta_{X^k} - \nu_k)^{\otimes 2}\rangle ].
\end{IEEEeqnarray*}
Combining this with
\eqref{eq:proof-contractivity-2-step-1-4}
and \eqref{eq:proof-contractivity-2-Fisher-lower-bound} yields
\begin{equation}
\label{eq:proof-contractivity-2-step-1-5}
\begin{IEEEeqnarraybox}[][c]{rCl}
\frac 1N I(\nu^N | m^N_*)
&=& \frac{1-\varepsilon_1-\varepsilon_2}{N}
\sum_{k \in [N]} \Expect[ I(\nu_k | \Pi[\nu_k] ) ]
+ \frac{\varepsilon_2}{2} \Expect[ I(\nu_N | m_*) ] \\
&&\adjustbin - (2 \varepsilon_2 + \varepsilon_4)
\Expect[ \langle R, \bar\nu^{\otimes 2}\rangle]
- \varepsilon_3 \Expect[ \langle R, \nu_N^{\otimes 2}\rangle]
\vphantom{\sum}\\
&&\adjustbin
- \frac{2+\varepsilon_1^{-1}+\varepsilon_4^{-1}+(4\varepsilon_3)^{-1}}{N^2}
\sum_{k \in [N]}
\Expect[ \langle R, (\delta_{X^k} - \nu_k)^{\otimes 2}\rangle ].
\end{IEEEeqnarraybox}
\end{equation}

\proofstep{Step 2: Graded dissipation}
Applying \eqref{eq:mf-contractivity-2} to each $\nu_k$ yields
\[
\frac 1N\sum_{k \in [N]} \Expect[ I(\nu_k | \Pi[\nu_k]) ]
\geqslant \frac{2\lambda}{N} \sum_{k \in [N]}
\Expect[ \mathcal F(\nu_k | m_*) ].
\]
Let $\varepsilon_5 > 0$.
We separate an $\varepsilon_5$ part from the right-hand side:
\begin{IEEEeqnarray*}{rCl}
\sum_{k \in [N]}
\Expect[ \mathcal F(\nu_k | m_*) ]
&=& (1-\varepsilon_5 + \varepsilon_5)\sum_{k \in [N]}
\Expect[ \mathcal F(\nu_k | m_*) ] \\
&=& (1-\varepsilon_5) H(\nu^N | m_*^{\otimes N})
+ (1-\varepsilon_5) \sum_{k \in [N]}
\Expect[ \langle W_*, \nu_k^{\otimes 2}\rangle] \\
&&\adjustbin
+ \varepsilon_5 \sum_{k \in [N]} \Expect [ \mathcal F(\nu_k | m_*) ].
\end{IEEEeqnarray*}
Applying Lemma~\ref{lem:approx-2} separately
to $U = W_*^+$ and $W_*^-$, we deduce that for all $\varepsilon_6 > 0$,
\begin{IEEEeqnarray*}{rCl}
\frac 1N \sum_{k \in [N]}
\Expect[\langle W_*^-, \nu_k^{\otimes 2}\rangle]
&\geqslant& \Expect[ \langle W_*, \mu_{\vect X}^{\otimes 2}\rangle]
- \frac{\varepsilon_6}{2}
\Expect[ \langle W_*^+ + W_*^-, \nu_N^{\otimes 2}\rangle] \\
&&\adjustbin- \frac{1}{2\varepsilon_6N^2}\sum_{k \in [N]}
\Expect[ \langle W_*^+ + W_*^-, (\delta_{X^k} - \nu_k)^{\otimes 2}\rangle] \\
&&\adjustbin
- \frac 1N \Expect[ \langle W_*^+, (\delta_{X^N} - \nu_N)^{\otimes 2}\rangle].
\end{IEEEeqnarray*}
Combining this with \eqref{eq:proof-contractivity-2-step-1-5} yields
\begin{equation}
\label{eq:proof-contractivity-2-step-2-last}
\begin{IEEEeqnarraybox}[][c]{rCl}
\IEEEeqnarraymulticol{3}{l}
{\frac 1N I(\nu^N | m^N_*)} \\
\quad
&\geqslant& \frac{2(1-\varepsilon_1-\varepsilon_2-\varepsilon_5)\lambda}{N}
\mathcal F^N(m^N | m_*)
+ \frac{\varepsilon_2}{2} \Expect[ I(\nu_N | m_*) ] \\
&&\adjustbin + \frac{2\varepsilon_5\lambda}{N}
\sum_{k \in [N]} \Expect[ \mathcal F(\nu_k | m_*) ] \\
&&\adjustbin - (2 \varepsilon_2 + \varepsilon_4)
\Expect[ \langle R, \bar\nu^{\otimes 2}\rangle]
- \varepsilon_3 \Expect[ \langle R, \nu_N^{\otimes 2}\rangle]
- \varepsilon_6 \lambda\Expect[ \langle W_*^+ + W_*^-,
\nu_N^{\otimes 2} \rangle ] \vphantom{\frac 1N}\\
&&\adjustbin
- \frac{2+\varepsilon_1^{-1}+\varepsilon_4^{-1}+(4\varepsilon_3)^{-1}}{N^2}
\sum_{k \in [N]}
\Expect[ \langle R, (\delta_{X^k} - \nu_k)^{\otimes 2}\rangle ] \\
&&\adjustbin - \frac{\lambda}{\varepsilon_6N^2}
\sum_{k \in [N]}
\Expect[ \langle W_*^+ + W_*^-, (\delta_{X^k} - \nu_k)^{\otimes 2}\rangle] \\
&&\adjustbin - \frac{2\lambda}{N}
\Expect[ \langle W_*^+, (\delta_{X^N} - \nu_N)^{\otimes 2}\rangle ].
\end{IEEEeqnarraybox}
\end{equation}

\proofstep{Step 3: Error control}
Choose $\varepsilon_3$, $\varepsilon_6$ as follows:
\begin{align*}
\varepsilon_3
&= \frac{\varepsilon_2}{4(\mu_R + \ell_R)}, \\
\varepsilon_6
&= \frac{\varepsilon_2}{4(M_W + \ell_W)\eta}
\end{align*}
By Lemma~\ref{lem:error-control}, this ensures
\[
\varepsilon_3 \Expect [ \langle R, \nu_N^{\otimes 2} ]
+ \lambda\varepsilon_6
\Expect [ \langle W_*^+ + W_*^-, \nu_N^{\otimes 2}\rangle ]
\leqslant \frac{\varepsilon_2}{2} \Expect [ I(\nu_N | m_*) ].
\]

Define
\begin{align*}
f &= \frac 1N \sum_{k \in [N]}
\Expect[ \mathcal F(\nu_k | m_*) ], \\
E_1 &= (2\varepsilon_2 + \varepsilon_4)\Expect[\langle R,
\bar\nu^{\otimes 2}\rangle],
\end{align*}
and let $E_2$ denote the sum of the last three terms
in \eqref{eq:proof-contractivity-2-step-2-last}.
The latter becomes
\begin{equation}
\label{eq:proof-contractivity-2-step-3-1}
\frac 1N I(\nu^N | m^N_*)
\geqslant \frac{2(1-\varepsilon_1-\varepsilon_2-\varepsilon_5)\lambda}{N}
\mathcal F^N(\nu^N | m_*)
+ 2\varepsilon_5\lambda f- E_1 - E_2.
\end{equation}
Using Lemma~\ref{lem:error-control} and the convexity:
\[
\Expect[H(\bar\nu | m_*)] \leqslant
\frac 1N \sum_{k \in [N]} \Expect[ H(\nu_k | m_*)]
= \frac 1NH(\nu^N |m_*^{\otimes N}),
\]
we obtain
\begin{IEEEeqnarray*}{rCl}
E_1 &\leqslant& \frac{2(2\varepsilon_2 + \varepsilon_4)
(\mu_R + \ell_R)\rho}{N} H(\nu^N | m_*^{\otimes N}), \\
E_2 &\leqslant&
\frac{4(2+\varepsilon_1^{-1}+\varepsilon_4^{-1}+(4\varepsilon_3)^{-1})\rho}{N}
\biggl(\mu_R + \frac{\ell_R}{N} H(\nu^N | m_*^{\otimes N}) + \ell_Rd\biggr) \\
&&\adjustbin +
\frac{4(2+\varepsilon_6^{-1})\lambda}{N}
\biggl(M_W + \frac{\ell_W}{N} H(\nu^N | m_*^{\otimes N}) + \ell_Wd\biggr).
\end{IEEEeqnarray*}
Moreover, \eqref{eq:mf-coercivity} implies
\[
\frac 1NH(\nu^N | m_*^{\otimes N})
= \frac 1N\sum_{k \in [N]} \Expect[ H(\nu_k | m_*) ]
\leqslant
\frac 1{\delta N}\sum_{k \in [N]} \Expect[ \mathcal F(\nu_k | m_*) ]
= \delta^{-1}f.
\]
Let $\varepsilon \in (0,1)$.
We now choose
\begin{align*}
\varepsilon_1 &= \frac{\varepsilon}{3}, \\
\varepsilon_2
&= \frac{\varepsilon}{3\bigl(1+2(\mu_R+\ell_R)/\delta\eta\bigr)}, \\
\varepsilon_4 &= \frac{\delta\eta\varepsilon}{3(\mu_R + \ell_R)}, \\
\varepsilon_5
&= \frac{(2\varepsilon_2+\varepsilon_4)(\mu_R + \ell_R)\rho}{\delta\eta}
+ \frac{2\bigl(2+\varepsilon_1^{-1}+\varepsilon_4^{-1}+(4\varepsilon_3)^{-1}
\bigr) \ell_R}{\delta\eta N}
+ \frac{2(2+\varepsilon_6^{-1})\ell_W}{\delta N}.
\end{align*}
Under these choices,
\begin{IEEEeqnarray*}{rCl}
\varepsilon_1 + \varepsilon_2 + \varepsilon_5
&=& \varepsilon
+ \frac{2\bigl(2+\varepsilon_1^{-1}+\varepsilon_4^{-1}+(4\varepsilon_3)^{-1}
\bigr) \ell_R}{\delta\eta N}
+ \frac{2(2+\varepsilon_6^{-1})\ell_W}{\delta N}, \\
\varepsilon_5 \eta f - \frac{E_1 + E_2}{2\rho}
&\geqslant&
- \frac{2\bigl(2+\varepsilon_1^{-1}+\varepsilon_4^{-1}+(4\varepsilon_3)^{-1}
\bigr)}{N} (\mu_R + \ell_R d) \\
&&\adjustbin - \frac{2(2+\varepsilon_6^{-1})\eta}{N} (M_W + \ell_Wd).
\end{IEEEeqnarray*}
Substituting the expressions for $\varepsilon_i$ in
\eqref{eq:proof-contractivity-2-step-3-1} yields the desired result.
\qed

\subsection{Proof of Corollary~\ref{cor:defective-lsi}}
\label{sec:proof-defective-lsi}
As in the proof of Corollary~\ref{cor:entropy-coercivity}, we observe that
\[
\mathcal{F}^N(\nu^N | m_*) - \mathcal{F}^N(m^N_* | m_*)
= H(\nu^N | m^N_*).
\]
By Theorem~\ref{thm:coercivity},
\[
\mathcal{F}^N(m^N_* | m_*)
\geqslant \delta_N H(m^N_* | m_*^{\otimes N}) - \Delta_{\mathcal{F}}
\geqslant -\Delta_{\mathcal{F}}.
\]
Hence,
\[
\mathcal{F}^N(\nu^N | m_*) \geqslant H(\nu^N | m^N_*) - \Delta_{\mathcal{F}}.
\]
Combining this bound with the conclusions of
Theorems~\ref{thm:contractivity-1} and~\ref{thm:contractivity-2}
proves the defective log-Sobolev inequality.
The exponential convergence of entropy for the Langevin dynamics
then follows by Grönwall's lemma.
\qed

\section{Generation of chaos}
\label{sec:goc}

This section presents the proof of Theorem~\ref{thm:goc}.
In the proof we perform only formal computations,
as in Section~\ref{sec:contractivity}.

\begin{proof}[Proof of Theorem~\ref{thm:goc}]
By the computation of \textcite[Proposition~2.1]{BJWPKSCompteRendu},
differentiating $\mathcal F^N(m^N_t | m_t)$ in time leads to
\begin{equation}
\label{eq:BJW}
\begin{IEEEeqnarraybox}[][c]{rCl}
\IEEEeqnarraymulticol{3}{l}
{\frac{\dd \mathcal F^N(m^N_t | m_t)}{\dd t}} \\
\quad&=& - \sum_{i \in [N]}
\int_{\X^N} {\biggl\lvert \nabla_i \log \frac{m^N_t(\vect x)}{m_t(x^i)}
+ \langle \nabla_1 W(x^i, \cdot), \mu_{\vect x} - m_t\rangle
\biggr\rvert^2 m^N_t(\dd\vect x)} \\
&&\adjustbin
- N \int_{\X^N} \int_{X^2}
\nabla_1 W(y,z) \cdot v_t(z)\,
(\mu_{\vect x} - m_t)^{\otimes 2} (\dd y\dd z)\,
m^N_t(\dd\vect x).
\end{IEEEeqnarraybox}
\end{equation}
By hypothesis, the last line is bounded by
\[
2\gamma(t) \mathcal F^N(m^N_t | m_t) + M(t).
\]
It remains to exploit the negative term on the second line of \eqref{eq:BJW}.

Define $g^N_t$ as the probability measure on $\X^N$ proportional to
\[
\exp \biggl( - \frac{N}{2} \langle W, (\mu_{\vect x} - m_t)^{\otimes 2}
\rangle\biggr) m_t^{\otimes N}(\dd\vect x).
\]
The second line of \eqref{eq:BJW} now reads
\[
- I(m^N_t | g^N_t).
\]
Observe that $g^N_t$ is the $N$-particle Gibbs measure associated to
the free energy functional:
\[
\nu \mapsto H(\nu | m_t) + \frac 12 \langle W, (\nu - m_t)^{\otimes 2}\rangle
= \mathcal F(\nu | m_t),
\]
for which, by its coercivity assumption, $m_t$ is the unique minimizer.
Denote by $W_t$ the reduced kernel associated to $m_t$.
Working under this new free energy amounts to making the following substitution:
\[
m_* \to m_t,\quad W_* \to W_t.
\]
Notably, we have the replacements:
\[
\mathcal F^N(\cdot | m_*) \to \mathcal F^N(\cdot | m_t),
\quad \Pi \to \Pi_{m_t}.
\]
The conditions of Theorem~\ref{thm:contractivity-1}
after substitution are satisfied.
As a result,
\[
I (m^N_t | g^N_t) \geqslant 2\lambda_N(t) \mathcal F^N(m^N_t | m_t)
- \Delta_{I}(t).
\]
Plugging this into \eqref{eq:BJW} and applying Grönwall's lemma
completes the proof.
\end{proof}

\section{Independent projection}
\label{sec:ip}

This section establishes the coercivity and contractivity of the modulated free
energy for the independent projection of Langevin dynamics. The proofs of
Theorems~\ref{thm:ip-coercivity} and~\ref{thm:ip-contractivity} are presented
in the following two subsections, respectively. The key step in both proofs
is to apply mean-field coercivity~\eqref{eq:mf-coercivity}
or contractivity~\eqref{eq:mf-contractivity-1} for suitable mixed measures.

Throughout this section, we perform only formal computations,
as in Section~\ref{sec:contractivity}.

\subsection{Proof of Theorem~\ref{thm:ip-coercivity}}

Define
\[
\bar\xi = \frac 1N \sum_{i \in [N]} \xi^i.
\]
Replacing $\nu$ by $\bar\xi$ in \eqref{eq:mf-coercivity} gives
\[
(1-\delta) H(\bar\xi | m_*)
+ \frac{1}{2N^2} \sum_{i,j\in [N]} \langle W_*, \xi^i \otimes \xi^j\rangle
\geqslant 0.
\]
Thus,
\begin{equation}
\label{eq:proof-ip-coercivity-1}
\begin{IEEEeqnarraybox}[][c]{rCl}
\IEEEeqnarraymulticol{3}{l}
{\sum_{i \in [N]} H(\xi^i | m_*)
+ \frac{1}{2(N-1)} \sum_{i,j \in [N] : i\neq j}
\langle W_*,\xi^i \otimes \xi^j\rangle} \\
\quad&\geqslant&
\sum_{i \in [N]} H(\xi^i | m_*)
- \frac{(1-\delta)N^2}{N-1} H(\bar\xi | m_*)
- \frac{1}{2(N-1)} \sum_{i \in [N]} \langle W_*, (\xi^i)^{\otimes 2}\rangle.
\end{IEEEeqnarraybox}
\end{equation}
The convexity of entropy implies
\begin{align*}
H(\bar\xi | m_*) &\leqslant \frac 1N \sum_{i \in [N]}
H(\xi^i | m_*),
\intertext{while Lemma~\ref{lem:error-control} gives}
\sum_{i \in [N]} \langle W_*, (\xi^i)^{\otimes 2}\rangle
&\leqslant 2(M_W+\ell_W) \sum_{i \in [N]} H(\xi^i | m_*).
\end{align*}
Inserting these estimates into \eqref{eq:proof-ip-coercivity-1} concludes.
\qed

\subsection{Proof of Theorem~\ref{thm:ip-contractivity}}

By Assumption~\ref{assu:LS-Pi},
\[
\sum_{i \in [N]}
I(\xi^i | \Pi[\bar\xi^{-i}])
\geqslant 2\rho_\Pi \sum_{i \in [N]} H(\xi^i | \Pi[\bar\xi^{-i}]),
\]
where the right-hand side satisfies
\begin{multline*}
\sum_{i \in [N]} H(\xi^i | \Pi[\bar\xi^{-i}])
= \sum_{i \in [N]} H(\xi^i | m_*)
+ \frac{1}{N-1}\sum_{i,j \in [N] : i \neq j}
\langle W_*, \xi^i \otimes \xi^j\rangle \\
+ \sum_{i \in [N]}
\log \int_{\X} \exp( - \langle W(x_*, \cdot),
\bar\xi^{-i}\rangle )\, m_*(\dd x_*).
\end{multline*}
Replacing $\nu$ with $\bar\xi^{-i}$ in \eqref{eq:mf-contractivity-1} yields
\[
\log \int_{\X} \exp( - \langle W(x_*, \cdot),
\bar\xi^{-i}\rangle )\, m_*(\dd x_*)
\geqslant - (1-\delta)
H(\bar\xi^{-i} | m_*)
- \biggl(1-\frac\delta 2\biggr)
\langle W_*, (\bar\xi^{-i})^{\otimes 2}\rangle.
\]
Thus by combining the three inequalities above, we obtain
\begin{equation}
\label{eq:proof-ip-contractivity-1}
\begin{IEEEeqnarraybox}[][c]{rCl}
\IEEEeqnarraymulticol{3}{l}
{\frac{1}{2\rho_\Pi}\sum_{i \in [N]}
I(\xi^i | \Pi[\bar\xi^{-i}])} \\
\quad&\geqslant&
\sum_{i \in [N]} H(\xi^i | m_*)
- (1-\delta) \sum_{i \in [N]} H(\bar\xi^{-i} | m_*) \\
&&\adjustbin + \frac{1}{N-1}\sum_{i,j \in [N] : i \neq j}
\langle W_*, \xi^i \otimes \xi^j\rangle
- \biggl(1-\frac\delta2\biggr)
\sum_{i \in [N]}
\langle W_*, (\bar\xi^{-i})^{\otimes 2}\rangle.
\end{IEEEeqnarraybox}
\end{equation}
By the convexity of entropy, the second line satisfies
\begin{multline*}
\sum_{i \in [N]} H(\xi^i | m_*)
- (1-\delta) \sum_{i \in [N]} H(\bar\xi^{-i} | m_*) \\
\geqslant
\sum_{i \in [N]} H(\xi^i | m_*)
- \frac{1-\delta}{N-1} \sum_{i \in [N]} \sum_{j\in [N] : j \neq i}
H(\xi^j | m_*)
\geqslant \delta \sum_{i \in [N]} H(\xi^i | m_*).
\end{multline*}
The third line satisfies
\[
\sum_{i \in [N]}
\langle W_*, (\bar\xi^{-i})^{\otimes 2}\rangle
= \frac{N-2}{(N-1)^2} \sum_{i,j\in [N]:i \neq j}
\langle W_*, \xi^i\otimes\xi^j\rangle
+ \frac 1{N-1} \sum_{i \in [N]}
\langle W_*, (\xi^i)^{\otimes 2}\rangle.
\]
Consequently,
\begin{IEEEeqnarray*}{rCl}
\IEEEeqnarraymulticol{3}{l}
{\frac 1{N-1}\sum_{i,j \in [N] : i \neq j}
\langle W_*, \xi^i \otimes \xi^j\rangle
- \biggl(1-\frac\delta2\biggr)
\sum_{i \in [N]}
\langle W_*, (\bar\xi^{-i})^{\otimes 2}\rangle} \\
\quad&=& \biggl( \frac{\delta}{2(N-1)}
+ \frac{1 - \delta/2}{(N-1)^2}\biggr)
\sum_{i,j \in [N] : i \neq j}
\langle W_*, \xi^i \otimes \xi^j\rangle \\
&&\adjustbin- \frac{1-\delta/2}{N-1}
\sum_{i \in [N]} \langle W_*, (\xi^i)^{\otimes 2}\rangle.
\end{IEEEeqnarray*}
Our hypotheses ensure that $\delta_N$ in Theorem~\ref{thm:ip-coercivity}
is positive.
Thus applying Theorem~\ref{thm:ip-coercivity} and Lemma~\ref{lem:error-control}
yields
\begin{align*}
\frac{1}{2(N-1)}\sum_{i,j \in [N]: i \neq j}
\langle W_*,\xi^i\otimes\xi^j\rangle
&\geqslant -(1-\delta_N)\sum_{i \in [N]} H(\xi^i | m_*)
\geqslant -\sum_{i \in [N]} H(\xi^i | m_*),
\\
\sum_{i \in [N]} \langle W_*, (\xi^i)^{\otimes 2}\rangle
&\leqslant 2(M_W+\ell_W)\sum_{i \in [N]} H(\xi^i | m_*).
\end{align*}
Inserting the estimates for the two lines
into \eqref{eq:proof-ip-contractivity-1} yields
\[
\frac{1}{2\rho_\Pi}
\sum_{i \in [N]}
I(\xi^i | \Pi[\bar\xi^{-i}])
\geqslant \delta \mathcal F^N_{\ind} (\xi^1,\ldots,\xi^N | m_*)
- \frac{2(1+M_W+\ell_W)}{N-1} \sum_{i \in [N]}
H(\xi^i | m_*).
\]
Theorem~\ref{thm:ip-coercivity} further bounds the last term as follows:
\[
\sum_{i \in [N]} H(\xi^i|m_*)
\leqslant \delta_N^{-1} \mathcal{F}^N_{\ind} (\xi^1, \ldots, \xi^N |m_*).
\]
This completes the proof.
\qed

\appendix

\section{Mean-field limit of functionals}
\label{app:recover-mf}

This appendix demonstrates that
taking the limit $N\to\infty$ in the $N$-particle functionals
yields their respective mean-field counterparts.

\begin{lem}
\label{lem:recover-free-energy}
Let Assumption~\ref{assu:W} hold.
Let $\nu \in \mathcal P_2(\X)$ satisfy $\mathcal F(\nu | m_*) < \infty$.
Then,
\[
\lim_{N \to \infty} \frac{\mathcal F^N(\nu^{\otimes N} | m_*)}{N}
= \mathcal F(\nu | m_*).
\]
\end{lem}

\begin{proof}[Proof of Lemma~\ref{lem:recover-free-energy}]
By definition,
\begin{align*}
\mathcal F^N(m^{\otimes N} | m_*)
&= H(\nu^{\otimes N} | m_*^{\otimes N})
+ \frac{N}{2}
\int_{\X^N} { \langle W_*, \mu_{\vect x}^{\otimes 2}\rangle\,\nu^{\otimes N}
(\dd\vect x)} \\
&= N H(\nu | m_*)
+ \frac{N}{2} \langle W_*, \nu^{\otimes 2}\rangle
+ \frac 12\int_{\X} {\langle W, (\delta_{x}-\nu)^{\otimes 2}\rangle\,
\nu(\dd x)}.
\end{align*}
Separating the positive and negative components in the last term,
we find
\[
\biggl\lvert
\int_{\X} {\langle W, (\delta_{x}-\nu)^{\otimes 2}\rangle\,\nu(\dd x)}
\biggr\rvert
\leqslant 4M_W + 2L_W \Var \nu.
\]
Dividing $\mathcal F^N(\nu^{\otimes N} | m_*)$ by $N$
and letting $N\to\infty$ completes the proof.
\end{proof}

\begin{lem}
\label{lem:recover-Fisher}
Let Assumption~\ref{assu:R} hold.
Let $\nu \in \mathcal P_2(\X)$ satisfy $I(\nu | m_*) < \infty$.
Then,
\[
\lim_{N \to \infty} \frac{I(m^{\otimes N} | m^N_*)}{N}
= I(\nu | \Pi[\nu]).
\]
\end{lem}

\begin{proof}[Proof of Lemma~\ref{lem:recover-Fisher}]
By Step~1 of the proof of Theorem~\ref{thm:contractivity-1},
\begin{multline*}
\frac{I(\nu^{\otimes N} | m^N_*)}{N} \\
= \Expect\biggl[
\int_{\X} {\biggl\lvert
\nabla \log \frac{\nu}{m_*} (y)
+ \biggl\langle \nabla_1 W_*(y,\cdot),
\frac{N-1}{N} \mu_{\vect X^{[N-1]}} + \frac 1N \delta_y
\biggr\rangle \biggr\rvert^2\, \nu(\dd y)\biggr]},
\end{multline*}
where $\vect X$ is distributed according to $\nu^{\otimes N}$.
Performing the change of measure
\[
\mu_{\vect X^{[N-1]}} + \frac 1N \delta_y \to \nu
\]
and applying Cauchy--Schwarz inequalities, we obtain
\[
\biggl\lvert \frac{I(\nu^N | m^N_*)}{N} - I(\nu | \Pi[\nu]) \biggr\rvert
\leqslant \varepsilon I(\nu | \Pi[\nu] )
+ \frac{1 + \varepsilon^{-1}}{N}
\Expect [ \langle R, (\mu_{\vect X} - \nu)^{\otimes 2} \rangle ],
\]
where $\varepsilon \in (0,1)$ is arbitrary.
As in the proof of Lemma~\ref{lem:recover-free-energy},
the associated error term satisfies
\[
\Expect[ \langle R, (\mu_{\vect X} - \nu)^{\otimes 2}\rangle ]
\leqslant \frac{4M_R + 2L_R \Var \nu}{N}.
\]
Letting $N \to \infty$ and using the arbitrariness
of $\varepsilon$ completes the proof.
\end{proof}

\section{Mode decomposition}
\label{app:mode-decomposition}

This appendix considers interactions that admit a mode decomposition:
\[
W(x,y) = -J r(x) \cdot r(y),
\]
where $r \colon \X \to \R^n$ and $J > 0$,
and derives criteria
for conditions~\eqref{eq:mf-coercivity}
and \eqref{eq:mf-contractivity-1}.
This is a special case
of the decomposition discussed in Remark~\ref{rem:bbd-spectral}.

Let $\nu \in \mathcal P_2(\X)$. For $ h \in \R^n$, we write
\[
T_h \nu(\dd x) = \frac{e^{h \cdot r(x)}\,\nu(\dd x)}
{\int_{\X} e^{h \cdot r(x')}\,\nu(\dd x')},
\]
whenever $T_h\nu$ is well defined as a probability measure.
We call such measures \emph{tilts} for $\nu$.
We suppose in particular that $T_hm_*$ is well defined for all $h \in \R^n$.
By definition, local equilibria are tilts:
\[
\Pi[\nu] = T_{\langle r, \nu - m_*\rangle} m_*.
\]
Define $f_* \colon \R^n \to \R$ as follows:
\begin{equation}
\label{eq:def-f}
f_*(h) = \log \int_{\T^1} e^{h \cdot r(\theta)}\,m_*(\dd \theta).
\end{equation}
We suppose that $f_*$ is $\mathcal C^\infty$ on $\R^n$ and
we refer to it as the \emph{magnetic Helmholtz free energy}.
Differentiating $f$, we find
\begin{align*}
\nabla f_*(h) &= \langle r, T_h m_*\rangle, \\
\nabla^2 f_*(h) &= \Cov_{T_h m_*} r
= \langle r^{\otimes 2}, T_h m_*\rangle
- \langle r, T_h m_*\rangle^{\otimes 2}.
\end{align*}
By an abuse of notation, we write
\[
\Cov T_h m_* = \Cov_{T_h m_*} r.
\]
Since $m_*$ has density, so does $T_h m_*$. Therefore for all $h \in \R^n$,
\[
\nabla^2 f(h) = \Cov {T_h m_*} > 0.
\]
Consequently, the image set $\Image \nabla f_*$ is open
and the Legendre transform
\begin{equation}
\label{eq:def-g}
g_*(\mu) = \sup_{h \in \R^n} h\cdot\mu - f_*(h) \in (-\infty, +\infty]
\end{equation}
is well defined.
We call $g_*$ the \emph{magnetic Gibbs free energy}
for the reference measure $m_*$.
Define
\begin{equation}
\label{eq:def-mu-h}
\mu^*_h = \langle r, T_h m\rangle = \nabla f_*(h).
\end{equation}
This is the $\mu$ variable associated to $h$ via the Legendre transform.
On $\Image \nabla f_*$,
the function $g_*$ is finite and $\mathcal C^\infty$,
and satisfies, in particular,
\begin{align*}
\nabla g_*(\mu) |_{\mu = \mu^*_h} &= h, \\
\nabla^2 g_*(\mu) |_{\mu = \mu^*_h} &= \nabla^2 f_*(h)^{-1},
\end{align*}
where $\nabla^2 f_*(h)^{-1}$ denotes the inverse matrix of
$\nabla^2 f_*(h)$.

The values of the mean-field functionals on tilts
can be expressed in terms of the finite-dimensional functions $f_*$ and $g_*$,
together with the mapping $h \mapsto \mu^*_h$.

\begin{lem}
\label{lem:functional-tilts}
For all $h \in \R^n$,
\begin{align*}
\langle W, (T_h m_* - m_*)^{\otimes 2}\rangle
&= -J \lvert \mu^*_h - \mu^*_0\rvert^2 \\
H(T_h m_* | m_*) &= g_*(\mu^*_h), \\
H(T_h m_* | \Pi[T_h m_*])
&= g_*(\mu^*_h) + f_*\bigl(J(\mu^*_h-\mu^*_0)\bigr)
- J\mu^*_0 \cdot (\mu^*_h - \mu^*_0)
- J\lvert \mu^*_h - \mu^*_0\rvert^2.
\end{align*}
\end{lem}

\begin{proof}[Proof of Lemma~\ref{lem:functional-tilts}]
Since $W(x,y) = - J r(x)\cdot r(y)$,
the interaction energy satisfies
\[
\langle W, (T_h m_* - m_*)^{\otimes 2}\rangle
= - J \langle r^{\otimes 2}, (T_h m_* - m_*)^{\otimes 2}\rangle
= - J \lvert \mu^*_h - \mu^*_0 \rvert^2.
\]
For the second equality, by the definition of $T_h m$,
\[
\log\frac{T_hm_*}{m_*} = h \cdot r - f_*(h).
\]
Hence,
\[
H(T_h m_* | m_*) = \biggl\langle \log \frac{T_hm_*}{m_*}, T_h m_* \biggr\rangle
= h \cdot \mu^*_h - f_*(h) = g_*(\mu^*_h).
\]
The last quantity satisfies
\begin{IEEEeqnarray*}{rCl}
H(T_h m_* | \Pi[T_h m_*])
&=& H(T_h m_* | m_*) + \langle W_*, (T_h m_* - m_*)^{\otimes 2}\rangle \\
&&\adjustbin + \log \int_{\X}
\exp \bigl( - \langle W_*(x, \cdot), T_h m_* - m_*\rangle\bigr)
\,m_*(\dd x).
\end{IEEEeqnarray*}
Since $\langle W_*,(T_hm_* - m_*)^{\otimes 2}\rangle
= \langle W,(T_hm_*-m_*)^{\otimes 2}\rangle$,
the terms on the first line equal
$g_*(\mu^*_h)$ and $- J\lvert\mu^*_h - \mu^*_0\rvert^2$, respectively.
The last line satisfies
\begin{IEEEeqnarray*}{rCl}
\IEEEeqnarraymulticol{3}{l}
{\log \int_{\X}
\exp \bigl( - \langle W_*(x, \cdot), T_h m_* - m_*\rangle\bigr)
m_*(\dd\theta)}\\
\quad&=&\log \int_{\X}
\exp \Bigl( J \langle r, T_h m_* - m_*\rangle\cdot \bigl(r(x)
- \langle r,m_*\rangle\bigr)
\Bigr)\, m_*(\dd x)\\
&=& f_*(J \langle r, T_h m_* - m_*\rangle)
- J\langle r, T_hm_* - m_*\rangle\cdot \langle r, m_*\rangle \\
&=& f_*\bigl(J (\mu^*_h - \mu^*_0) \bigr)
- J\mu^*_0 \cdot (\mu^*_h - \mu^*_0)
\end{IEEEeqnarray*}
This completes the proof.
\end{proof}

The next lemma shows that verifying the coercivity
and free energy conditions reduces to checking them for tilts.
This property is referred to as the principle of maximum entropy by
\textcite[Fact~30]{ChenEldanLocalizationSchemes}.

\begin{lem}
\label{lem:maximum-entropy}
Let $\delta \in (0,1)$.
The condition \eqref{eq:mf-coercivity} holds if and only if
\begin{alignat}{3}
&\forall h \in\R^n,&\qquad
\mathcal F(T_h m_* | m_*) &\geqslant \delta H(T_h m_* | m_*),
\tag{Coer$'$} \label{eq:Coer-p} \\
\intertext{Moreover, the condition \eqref{eq:mf-contractivity-1}
holds if and only if}
&\forall h \in \R^n,&\qquad
H(T_hm_* | \Pi[T_hm_*]) &\geqslant \delta \mathcal F(T_h m_*|m_*),
\tag{FE$'$} \label{eq:FE-p}
\end{alignat}
\end{lem}

\begin{proof}[Proof of Lemma~\ref{lem:maximum-entropy}]
By specialization,
the conditions \eqref{eq:mf-coercivity} and \eqref{eq:mf-contractivity-1}
imply \eqref{eq:Coer-p} and \eqref{eq:FE-p}, respectively.
Therefore it suffices to show the converse implications.

Let \eqref{eq:Coer-p} hold and let $\nu \in \mathcal P(\T)$ be arbitrary.
Trivially,
\begin{IEEEeqnarray*}{rCl}
\mathcal F(\nu | m_*)
- \delta H(\nu | m_*) &\geqslant&
\inf_{\nu' : \langle r,\nu'\rangle = \langle r,\nu\rangle}
\mathcal F(\nu' | m_*)
- \delta H(\nu' | m_*).
\intertext{The relation $\langle r,\nu'\rangle = \langle r,\nu\rangle$
can be interpreted as a macroscopic constraint.
The functional to minimize reads}
\mathcal F(\nu' | m_*) - \delta H(\nu' | m_*)
&=& (1-\delta) H(\nu' | m_*)
- \frac J2 \lvert \langle r, \nu' - m_*\rangle\rvert^2,
\end{IEEEeqnarray*}
and its constrained first-order optimality condition writes
\[
(1-\delta) \log \frac{\nu'}{m_*}
- J \langle r, \nu' - m_*\rangle r + \ell r = \text{constant},
\]
where $\ell$ is the Lagrange multiplier.
In other words, $\nu' = T_h m_*$ for
\[
h = \frac{J\langle r,\nu' - m_*\rangle}{1-\delta} - \frac{\ell}{1-\delta}.
\]
The condition $\delta \in (0,1)$ ensures that the optimal $\nu'$ exists
and satisfies the first-order condition.
Since $\nu'$ is a tilt, the condition \eqref{eq:Coer-p} gives
\[
(1-\delta) \mathcal F(\nu' | m_*) - \delta H(\nu' | m_*) \geqslant 0.
\]
This establishes \eqref{eq:mf-coercivity}.
A similar argument shows that \eqref{eq:FE-p}
implies \eqref{eq:mf-contractivity-1}.
\end{proof}

We now consider the free energy landscape
$\mathcal F(\cdot | m)$,
where we replace $m_*$ with another distribution $m \in \mathcal P_2(\X)$.
Define $f \colon \R^n \to \R$
by replacing $m_*$ with $m$ in \eqref{eq:def-f}.
The mappings $g \colon \R^n \to \R$, $\mu_{\cdot} \colon \R^n \to \R^n$
are defined by corresponding substitutions
in \eqref{eq:def-g} and \eqref{eq:def-mu-h}, respectively.
All results above, particularly Lemmas~\ref{lem:functional-tilts}
and \ref{lem:maximum-entropy}, remain true under these substitutions.
Denote also by $\id_{n \times n}$ the $n \times n$ identity matrix.
The following lemma provides criteria
for the coercivity and contractivity of $\mathcal F(\cdot | m)$,
namely \eqref{eq:Coer-m} and \eqref{eq:FE-m},
via perturbation arguments.

\begin{lem}
\label{lem:xy-perturbed-coer-fe}
Suppose that for all $h \in \R^n$,
\[
f_*(h) = \varphi( \lvert h \rvert)
\]
for some $\varphi \colon [0,\infty) \to [0,\infty)$
satisfying $\varphi'$, $\varphi'' > 0$ and $\varphi''' < 0$,
with $\varphi(0)=\varphi'(0) = 0$ and $\varphi''(0) = 1/\Jc > 0$.
Suppose $\alpha^{-1} m_* \leqslant m \leqslant \alpha m_*$
for some $\alpha \in [1, \sqrt{\Jc/J})$.
Then the condition \eqref{eq:Coer-m} holds for
\[
\delta = 1 - \frac {\alpha^2J}{\Jc}.
\]
Moreover, if additionally $2\alpha^2 < J/\Jc + \Jc/J$,
then the condition \eqref{eq:FE-m} holds for
\[
\delta = \frac{\Jc^2 - 2\alpha^2\Jc J + J^2}{\Jc(\Jc - \alpha^2 J)}
\]
\end{lem}

\begin{proof}[Proof of Lemma~\ref{lem:xy-perturbed-coer-fe}]
We first establish the coercivity.
By Lemmas~\ref{lem:functional-tilts} and \ref{lem:maximum-entropy},
it suffices to show that for all $h \in \R^n$,
\[
(1-\delta) g(\mu_h) \geqslant \frac{J}{2} \lvert \mu_h - \mu_0\rvert^2.
\]
Since $g(\mu_0) = 0$, $\nabla g(\mu_0) = 0$, it suffices to prove that
for all $h \in \R^n$,
\[
(1-\delta)\nabla^2 g(\mu) |_{\mu=\mu_h}
\succcurlyeq J\id_{n\times n},
\]
where $\succcurlyeq$ denotes the domination between symmetric matrices.
The Gibbs free energy $g$ satisfies
\[
\nabla^2 g(\mu) |_{\mu = \mu_h}
= \nabla^2 f(h)^{-1}
= (\Cov T_h m)^{-1}.
\]
Using $\alpha^{-1} m \leqslant m \leqslant \alpha m$, we obtain
\[
T_h m(\dd\theta)
= \frac{e^{h\cdot r(\theta)}\,m(\dd\theta)}
{\int_{\X} e^{h\cdot r(\theta')}\,m(\dd\theta')}
\leqslant \alpha^2
\frac{e^{h\cdot r(\theta)}\,m_*(\dd\theta)}
{\int_{\X} e^{h\cdot r(\theta')}\,m_*(\dd\theta')}
= \alpha^2 T_h m_*(\dd\theta).
\]
and similarly,
\[
T_h m \geqslant \alpha^{-2} T_h m_*.
\]
Applying \cite[Lemma~5.1.7]{BGLMarkov} yields
\begin{equation}
\label{eq:covariance-perturbation}
\alpha^{-2} \Cov T_h m_*
\preccurlyeq \Cov T_h m
\preccurlyeq \alpha^2 \Cov T_h m_*,
\end{equation}
Using $f_*(h) = \varphi(\lvert h\rvert)$, we calculate for $h \neq 0$:
\[
\Cov T_h m_* = \nabla^2 f_*(h)
= \varphi''(\lvert h\rvert) e_h \otimes e_h
+ \frac{\varphi'(\lvert h\rvert)}{\lvert h\rvert}
(\id_{n\times n} - e_h \otimes e_h),
\]
where $e_h \coloneqq h / \lvert h\rvert$.
The property $\varphi''' < 0$ implies
\[
\varphi''(\lvert h\rvert) \leqslant \frac{\varphi'(\lvert h\rvert)}
{\lvert h\rvert} \leqslant \varphi''(0) = \frac{1}{\Jc}.
\]
Thus,
\[
\nabla^2 g(\mu)|_{\mu = \mu_h}
\succcurlyeq \alpha^{-2} (\Cov T_h m_* )^{-1}
\succcurlyeq \alpha^{-2} (\Cov m_* )^{-1}
= \frac{\Jc}{\alpha^2} \id_{n\times n}.
\]
Plugging in the expression for $\delta$, we obtain
\[
(1-\delta)\nabla^2 g(\mu) |_{\mu = \mu_h}
\succcurlyeq \frac{J\alpha^2}{\Jc}\frac{\Jc}{\alpha^2}\id_{n\times n}
= J \id_{n\times n}.
\]
This completes the proof of coercivity.

We now establish the contractivity.
By Lemmas~\ref{lem:functional-tilts} and \ref{lem:maximum-entropy},
it suffices to show that for all $h \in \R^n$,
\[
(1-\delta) g(\mu_h) + f\bigl(J(\mu_h - \mu_0)\bigr) - J\mu_0 \cdot (\mu_h-\mu_0)
\geqslant \biggl(1 - \frac\delta2\biggr) J \lvert \mu_h - \mu_0 \rvert^2.
\]
The covariance relation \eqref{eq:covariance-perturbation}
is equivalent to
\[
\alpha^{-2} \nabla^2 f_*(h) \preccurlyeq
\nabla^2 f(h) \preccurlyeq \alpha^2 \nabla^2 f_*(h),
\]
which implies, in particular,
\[
\alpha^{-2} f_*(h)
\leqslant f(h) - \nabla f(0) \cdot h\leqslant \alpha^2
f_*(h).
\]
Substituting $h$ with $J(\mu_h - \mu_0)$ gives
\[
f\bigl(J(\mu_h - \mu_0)\bigr) - J\mu_0 \cdot (\mu_h-\mu_0)
\geqslant \alpha^{-2} f_* \bigl( J(\mu_h - \mu_0) \bigr).
\]
By the definition of $g$, we have
\begin{IEEEeqnarray*}{rCl}
g(\mu) &=& \sup_{h \in \R^n} \mu\cdot h - f(h) \\
&\geqslant& \sup_{h \in \R^d} \mu\cdot h
- \nabla f(0) \cdot h - \alpha^{2} f_*(h) \\
&=& \alpha^2 g_*\bigl(\alpha^{-2}(\mu - \mu_0)\bigr).
\end{IEEEeqnarray*}
Therefore, it suffices to prove for all $\mu \in \Image\nabla f$,
\[
\alpha^2(1-\delta)
g_*\bigl( \alpha^{-2} (\mu - \mu_0) \bigr)
+ \alpha^{-2} f_*\bigl(J(\mu-\mu_0)\bigr)
\geqslant \biggl( 1 - \frac\delta2\biggr) J \lvert \mu - \mu_0\rvert^2.
\]
Denote $\mu' = \mu - \mu_0$
and evaluate both sides of this inequality along the radial path
$t \mapsto t\mu'$ for $t \in [0,1]$.
Then this inequality reduces to verifying the following second-order condition:
\[
\frac{1-\delta}{\alpha^2} e_{\mu'}^{\tran} \nabla^2 g_*(\alpha^{-2} \mu')
e_{\mu'} + \frac{J^2}{\alpha^2} e_{\mu'}^{\tran}\nabla^2 f_*(J\mu') e_{\mu'}
- (2-\delta) J \geqslant 0,
\]
where $e_{\mu'} \coloneqq \mu' \!/ \lvert \mu'\rvert$.
By properties of the Legendre transform,
\[
\nabla^2 g_*(\alpha^{-2} \mu')
= \nabla^2 f_*(\ell)^{-1},
\]
for $\ell \in \R^n$ such that $\nabla f_*(\ell) = \alpha^{-2}\mu'$.
Thus,
\[
J \lvert \mu'\rvert
= J \alpha^2 \lvert\nabla f_*(\ell)\rvert
\leqslant \frac{J\alpha^2}{\Jc} \lvert \ell\rvert
\leqslant \lvert \ell\rvert,
\]
where the first inequality follows from
bounding the operator norm of $\nabla^2f_*$ by $1/\Jc$.
The radial second-order derivatives reduce to those of $\varphi$:
\[
e_{\mu'}^{\tran} \nabla^2 g_*(\alpha^{-2} \mu') e_{\mu'}
= \frac{1}{\varphi''(\lvert\ell\rvert)},\quad
e_{\mu'}^{\tran}\nabla^2 f_*(J\mu') e_{\mu'}
= \varphi''(J\lvert\mu'\rvert),
\]
Moreover, since $\varphi''' < 0$,
\[
\varphi''(J\lvert\mu'\rvert) \geqslant \varphi''(\lvert\ell\rvert).
\]
Therefore, it suffices to show
\[
\frac{1-\delta}{\alpha^2\varphi''(\lvert\ell\rvert)}
+ \frac{J^2}{\alpha^2} \varphi''(\lvert\ell\rvert)
- (2-\delta) J \geqslant 0.
\]
The choice of $\delta$ ensures
\[
\frac{1-\delta}{J^2}
\geqslant \frac 1{\Jc^2}.
\]
So the function $s \mapsto (1-\delta) s^{-1} + J^2 s$ is decreasing
on $(0,1/\Jc]$.
Thus, applying $\varphi''(\lvert\ell\rvert) \leqslant 1/\Jc$ yields
\[
\frac{1-\delta}{\alpha^2\varphi''(\lvert\ell\rvert)}
+ \frac{J^2}{\alpha^2} \varphi''(\lvert\ell\rvert)
- (2-\delta) J
\geqslant \frac{(1-\delta)\Jc}{\alpha^2}
+ \frac{J^2}{\alpha^2\Jc} - (2-\delta)J.
\]
Substituting the value of $\delta$, we see that
the right-hand side is nonnegative.
This completes the proof of contractivity.
\end{proof}

\begin{rem}
\label{rem:optimality-delta}
If $\alpha=1$, namely $m = m_*$,
Lemma~\ref{lem:xy-perturbed-coer-fe} yields $\delta=1-J/\Jc$
for both \eqref{eq:mf-coercivity} and \eqref{eq:mf-contractivity-1}.
By expanding the projected functionals
\[
H( T_h m_* | \Pi [ T_h m_* ]),\quad
\mathcal F( T_h m_* | m_*),\quad
H(T_h m_* | m_*)
\]
around $h = 0$, we see that this constant is optimal.
\end{rem}

\section{Mean-field XY model}
\label{app:xy}

This appendix establishes Proposition~\ref{prop:xy-goc}
and all assumptions and notations in Section~\ref{sec:exm-xy}
remain in force throughout.
We first show two lemmas addressing the short-time regularization
and long-time decay of the dynamics \eqref{eq:mf-fp}, respectively,
and then prove Proposition~\ref{prop:xy-goc}.

\begin{lem}
\label{lem:xy-decay-short-time}
Let $m_\cdot \in \mathcal C\bigl([0,1]; \mathcal P(\T)\bigr)$
be a weak solution to \eqref{eq:mf-fp} with $J < 2$.
Define $D = \sup_{t \in [0,1]} \lVert m_t - m_* \rVert_{\TV}$.
Then for some universal constant $C > 0$,
\[
\lVert \partial_\theta m_1 \rVert_{L^\infty}
\leqslant C (1+ D) D.
\]
\end{lem}

\begin{proof}[Proof of Lemma~\ref{lem:xy-decay-short-time}]
Set $b_t(\theta) = - \langle \partial_1 W(\theta,\cdot), m_t\rangle$.
The measure flow $m_t$ satisfies
\[
\partial_t m_t = \partial_{\theta}^2 m_t - \partial_\theta(b_t m_t).
\]
Define the new variable $f_t = m_t - m_* = m_t - (2\pi)^{-1}$.
This variable satisfies
\[
\partial_t f_t = \partial_\theta^2 f_t
- \partial_\theta (b_t f_t) - \frac{\partial_\theta b_t}{2\pi}.
\]
By the regularity of $W$,
\[
\lVert b_t \rVert_{W^{2,\infty}} \leqslant CD,
\]
where $C>0$ is universal and may change below.
Duhamel's formula gives
\begin{align*}
f_1 &= G_1 \star_{\theta} f_0 - \partial_\theta G \star_{t,\theta}
\biggl(b f
+ \frac{\partial_\theta b}{2\pi}\biggr), \\
\partial_\theta f_1 &= \partial_\theta G_1 \star_\theta f_0
- \partial_\theta G \star_{t,\theta} \partial_\theta
\biggl(bf + \frac{\partial_\theta b}{2\pi}\biggr).
\end{align*}
Here, $G$ is the fundamental solution to the heat equation
on the time interval $[0,1]$ and $G_1$ is its time-$1$ marginal;
and $\star_{\theta}$, $\star_{t,\theta}$ denote
spatial and temporal-spatial convolutions, respectively.
Since $b$ has $W^{2,\infty}$ regularity
and both convolutions are regularizing, a bootstrap argument gives
\[
\lVert \partial_\theta m_1\rVert_{L^\infty}
= \lVert \partial_\theta f_1\rVert_{L^\infty}
\leqslant C (1 + \lVert b\rVert_{L^\infty_t(W^{2,\infty}_\theta)} )
\lVert f \rVert_{L^\infty_t(\TV_\theta)}
\leqslant C(1+D)D. \qedhere
\]
\end{proof}

\begin{lem}
\label{lem:xy-decay-long-time}
Let $m_\cdot \in \mathcal C\bigl([0,+\infty); \mathcal P(\T)\bigr)$
be a weak solution to \eqref{eq:mf-fp} with $J < 2$.
Then there exists a universal constant $C > 0$ such that
for all $t \geqslant (\rho_\Pi\delta)^{-1} \log (C/\delta)$,
\[
\lVert \partial_\theta \log m_t \rVert_{L^\infty}
\leqslant \frac{C}{\delta} e^{-\delta\rho_\Pi t}.
\]
\end{lem}

\begin{proof}[Proof of Lemma~\ref{lem:xy-decay-long-time}]
Since $\lVert m_t - m_*\rVert_{\TV} \leqslant 1$ for all $t \geqslant 0$,
Lemma~\ref{lem:xy-decay-short-time} gives
\[
\lVert \partial_\theta m_1 \rVert_{L^\infty}
\leqslant C
\]
for some universal $C > 0$, which may change below.
This implies that $m_1$ has finite entropy
and satisfies the free energy bound:
\[
\mathcal F(m_1 | m_*) \leqslant C.
\]
For $t \geqslant 1$, the free energy satisfies
\[
\frac{\dd \mathcal F(m_t | m_*)}{\dd t}
= - I(m_t | \Pi[m_t]) \leqslant - 2 \rho_\Pi H(m_t | \Pi[m_t])
\leqslant - 2\delta \rho_\Pi \mathcal F(m_t | m_*),
\]
by applying Assumption~\ref{assu:LS-Pi} and \eqref{eq:mf-contractivity-1}
successively.
Thus for $t \geqslant 1$,
\[
\lVert m_t - m_*\rVert_{\TV}^2 \leqslant
2H(m_t | m_*) \leqslant \frac{C}{\delta} \mathcal F(m_t | m_*)
\leqslant \frac{C}{\delta} e^{-2\delta\rho_\Pi t}.
\]
Applying Lemma~\ref{lem:xy-decay-short-time} to the interval $[t, t+1]$ yields
\[
\lVert \partial_\theta m_{t+1}\rVert_{L^\infty}
\leqslant \frac{C}{\delta} e^{-\delta\rho_\Pi t}.
\]
For $t > (\delta\rho_\Pi)^{-1} \log (C/\delta)$,
\[
\Osc m_{t} \leqslant \pi \lVert \partial_\theta m_t \rVert_{L^\infty}
\leqslant \frac{1}{4\pi}.
\]
This forces $\inf m_{t} \geqslant (4\pi)^{-1}$. Therefore,
\[
\lVert \partial_\theta\log m_t \rVert_{L^\infty}
\leqslant 4\pi \lVert \partial_\theta m_t \rVert_{L^\infty}
\leqslant \frac{C}{\delta} e^{-\delta\rho_\Pi t}. \qedhere
\]
\end{proof}

\begin{proof}[Proof of Proposition~\ref{prop:xy-goc}]
We first establish global-in-time entropic propagation of chaos
using the Jabin--Z.~Wang method.
Differentiating $t \mapsto H(m^N_t | m_t^{\otimes N})$ yields
\[
\frac{\dd H(m^N_t|m_t^{\otimes N})}{\dd t}
\leqslant \frac 14 \sum_{i \in [N]}
\Expect [ \lvert \langle \nabla_1 W(X^i_t, \cdot),
\mu_{\vect X_t} - m_t\rangle \rvert^2],
\]
where $\vect X_t$ is distributed according to $m^N_t$;
see \cite[Remark~3.3]{LLFSharp} for details.
Applying Assumption~\ref{assu:R} and Corollary~\ref{cor:jw}
with $\Osc_2 R \leqslant 16$ successively yields
\[
\sum_{i \in [N]}
\Expect [ \lvert \langle \nabla_1 W(X^i_t, \cdot),
\mu_{\vect X_t} - m_t\rangle \rvert^2]
\leqslant N \Expect[ \langle R, (\mu_{\vect X_t} - m_t)^{\otimes 2}\rangle]
\leqslant C H(m^N_t | m_t^{\otimes N}) + C,
\]
where $C>0$ is universal and may change below.
By Grönwall's lemma, this implies
\begin{equation}
\label{eq:global-poc}
H(m^N_t | m_t^{\otimes N}) \leqslant e^{Ct} H(m^N_0 | m_0^{\otimes N})
+ C(e^{Ct} - 1).
\end{equation}

We now employ Theorem~\ref{thm:goc} to exploit the long-time dissipation
of the modulated free energy.
In the sequel $C_J$ denotes a positive constant that depends only on
$J$ and may change between lines.
Lemma~\ref{lem:xy-decay-long-time} gives that
for $t \geqslant 2(\delta\rho_\Pi)^{-1} \log (C/\delta) = C_J$,
\[
\alpha(t)^{-1} m_* \leqslant m \leqslant \alpha(t) m_*
\]
where
\[
\alpha(t) = 1 + C_J e^{-\delta\rho_\Pi t}.
\]
Changing $C_J$ if necessary, we have, for $t \geqslant C_J$,
\[
\alpha(t)^2 \leqslant 1 + \frac{\delta^2}{2}
< \frac 1J + \frac J4.
\]
In this case,
\[
\frac{4+J^2-4\alpha(t)^2 J}{4 - 2\alpha(t)^2J}
\geqslant \delta - C_J e^{-\delta\rho_\Pi t}
\eqqcolon \delta(t).
\]
By \cite[Theorem~D.2]{DLSPhase},
all conditions of Lemma~\ref{lem:xy-perturbed-coer-fe} are satisfied.
Therefore, \eqref{eq:Coer-m} and \eqref{eq:FE-m}
hold for $m = m_t$ and $\delta = \delta(t)$.
By Holley--Stroock \cite{HolleyStroockLSI},
$m_t$ satisfies a $\rho(t)$-log-Sobolev inequality
with
\[
\rho(t) = \exp (-C_J e^{-\delta\rho_\Pi t}).
\]
Moreover, since for all $\nu \in \mathcal P(\T)$,
\[
\Osc \log \Pi_{m_t} [\nu]
\leqslant 4J + C_J e^{-\delta\rho_\Pi t},
\]
the images of $\Pi_{m_t}$ satisfy a $\rho_\Pi(t)$-log-Sobolev inequality with
\[
\rho_\Pi(t) = \exp ( - C_J e^{-\delta\rho_\Pi t}) \rho_\Pi.
\]
Finally, observe that
\[
- \nabla_1 W(\theta,\theta')  v_t (\theta')
= \sin (\theta - \theta')  v_t(\theta')
= \sin \theta \cos \theta' v_t(\theta')
- \cos \theta \sin \theta' v_t(\theta'),
\]
where $v_t$ satisfies
\[
\lVert v_t\rVert_{L^\infty}
\leqslant \lVert \nabla \log m_t \rVert_{L^\infty}
+ C \lVert m_t - m_* \rVert_{L^1}
\leqslant C_J e^{-\delta\rho_\Pi t}.
\]
Decomposing the kernel $(\theta,\theta') \mapsto - \nabla_1 W(\theta,\theta')$
as in Remark~\ref{rem:jw} and applying Corollary~\ref{cor:jw},
we obtain
\begin{multline*}
- \int_{\X^N} \int_{\X^2}
\nabla_1 W(y,z) \cdot v_t(z)\,(\mu_{\vect x} - m_t)^{\otimes 2}
(\dd y \dd z)\, m^N_t(\dd\vect x) \\
\leqslant C_J e^{-\delta\rho_\Pi t} H(m^N_t | m_t^{\otimes N})
+ C_J e^{-\delta\rho_\Pi t}.
\end{multline*}
Applying Theorem~\ref{thm:coercivity} for the free energy
$\mathcal F(\cdot | m_t)$ yields
\begin{equation}
\label{eq:bound-entropy-by-mfe}
H(m^N_t | m_t^{\otimes N})
\leqslant C_J \mathcal F^N(m^N_t | m_t) + C_J.
\end{equation}
Thus \eqref{eq:goc-error} is verified with
\[
2\gamma(t) = M(t) = C_J e^{-\delta\rho_\Pi t}.
\]
The assumptions of Theorem~\ref{thm:goc} are thus
verified on the interval $[t_0, t]$ for a $t_0$ depending only on $J$.
Specifically, Theorem~\ref{thm:goc} gives
\[
\Gamma(t,s)
\geqslant 2(1-\varepsilon) \delta\rho_\Pi (t-s)
- C_J,\quad
\Delta_I(t)
\leqslant \frac{C_J}{\varepsilon},
\]
and, in consequence,
\[
\mathcal F^N(m^N_t|m_t)
\leqslant C_J e^{-2(1-\varepsilon)\delta\rho_\Pi(t-t_0)}
\mathcal F^N(m^N_{t_0} | m_{t_0})
+ \frac{C_J}{(1-\varepsilon)\varepsilon}.
\]

Combining \eqref{eq:global-poc} with
\[
\mathcal F^N(m^N_{t_0} | m_{t_0})
\leqslant C H(m^N_{t_0} | m_{t_0}^{\otimes N}),
\]
we obtain
\[
\mathcal F^N(m^N_{t_0} | m_{t_0})
\leqslant C e^{Ct_0} H(m^N_0 | m_0^{\otimes N}) + C(e^{Ct_0}-1)
\leqslant C_J H(m^N_0 | m_0^{\otimes N}) + C_J.
\]
Together with \eqref{eq:bound-entropy-by-mfe}, this yields,
for $t \geqslant t_0$,
\[
H(m^N_t | m_t^{\otimes N})
\leqslant C_J e^{-2(1-\varepsilon)\delta\rho_\Pi(t-t_0)}
H(m^N_0 | m_0^{\otimes N}) + \frac{C_J}{(1-\varepsilon)\varepsilon}.
\]
By \eqref{eq:global-poc}, the inequality remains true for $t < t_0$
upon changing $C_J$. Absorbing $t_0$ into $C_J$ completes the proof.
\end{proof}

\section{Double-well Curie--Weiss model}
\label{app:cw}

This appendix establishes the coercivity
of the supercritical double-well Curie--Weiss model
with a non-vanishing external field.

\begin{lem}
\label{lem:cw-coercive}
Under the assumptions and notations of Section~\ref{sec:exm-cw},
in the parameter regime \eqref{eq:cw-regime}, there exists
$\delta > 0$ such that for all $\nu \in \mathcal P_2(\R)$,
\[
\mathcal F(\nu | m_+) \geqslant \delta H(\nu | m_+).
\]
\end{lem}

\begin{proof}[Proof of Lemma~\ref{lem:cw-coercive}]
By Lemma~\ref{lem:maximum-entropy}, it suffices to show that
for all $\ell \in \R$,
\[
(1-\delta) H(T_\ell m_+ | m_+)
+ \frac 12 \langle W, (T_\ell m_+ - m_+)^{\otimes 2} \rangle = 0.
\]
Since $m_+$ is itself a tilt to $m_*$:
\[
m_+ = T_{\ell_+} m_*,
\]
this is equivalent to showing that for all $\ell \in \R$,
\[
(1-\delta) H(T_\ell m_* | T_{\ell_+} m_*)
+ \frac 12 \langle W, (T_{\ell} m_* - T_{\ell_+} m_*)^{\otimes 2}\rangle
\geqslant 0.
\]
By Lemma~\ref{lem:functional-tilts},
this reduces to proving that for all $\mu \in \R$,
\[
G_\delta(\mu) \coloneqq (1-\delta)
\bigl( g(\mu) - g(\mu_{\ell_+}) - g'(\mu_{\ell_+})
(\mu - \mu_{\ell_+}) \bigr)
- \frac{J}{2} (\mu - \mu_{\ell_+})^2 \geqslant 0.
\]
Since $G_\delta$ has superquadratic growth, its global minimum is attained
by some $\mu$ solving the first-order condition
\[
(1-\delta) \bigl( g'(\mu) - g'(\mu_{\ell+}) \bigr) - J (\mu - \mu_{\ell_+})
= 0.
\]
Equivalently, $\mu = \mu_{\ell}$ for some $\ell \in \R$ satisfying
\begin{equation}
\label{eq:delta-self-consistency}
f'(\ell) - f'(\ell_+) = \frac{1-\delta}{J} (\ell - \ell_+).
\end{equation}
The sublinear growth of $f'$ ensures the solutions of
\eqref{eq:delta-self-consistency} remain stay a bounded interval.
When $\delta = 0$, this reduces to the self-consistency equation
\eqref{eq:self-consistency}
which admits exactly three solutions: $\ell_-$, $\ell_{\neu}$ and $\ell_+$.
Since $0 < h < \hc(J)$, by analyzing the function
$\ell \mapsto \ell - Jf'(\ell)$,
we see that these solutions are distinct from $\pm \ell(J)$,
where $\ell(J)$ is the unique positive number satisfying
$f''\bigl(\ell(J)\bigr) = 1/J$.
Thus, for $\ell = \ell_-$, $\ell_{\neu}$, $\ell_+$, we have
\[
J f''(\ell) - 1 \neq 0.
\]
In other words,
they are non-degenerate solutions of~\eqref{eq:self-consistency}.
By the implicit function theorem, for $\delta$ sufficiently small,
\eqref{eq:delta-self-consistency} also admits exactly three solutions,
denoted $\ell_-(\delta)$, $\ell_{\neu}(\delta)$ and $\ell_+$,
where the first two depend continuously on $\delta$
with $\ell_-(0) = \ell_-$ and $\ell_{\neu}(0) = \ell_{\neu}$.
Moreover, for such $\delta$,
\[
\inf G_\delta = \inf\{
G_\delta(\mu_{\ell_-(\delta)}), G_\delta(\mu_{\ell_{\neu}(\delta)}), 0 \}.
\]
Note that $G_0(\mu) = g(\mu) - J\mu^2\!/2 - h\mu+ \text{constant}$.
Studying the monotonicity of $G_0$
and using $h > 0$ yields $G_0(\mu_{\ell_-})$, $G_0(\mu_{\ell_\neu}) > 0$.
By continuity, for small enough $\delta$,
\[
G_\delta(\mu_{\ell_s(\delta)}) > 0,\qquad\text{$s = -$, $\neu$.}
\]
This establishes $\inf G_\delta = 0$.
\end{proof}

\subsubsection*{Acknowledgments}
The author thanks Pierre Monmarché and Zhenjie Ren
for stimulating exchanges inspiring this work;
Thierry Bodineau and Benoît Dagallier
for insightful discussions linking their works%
~\cite{BBDCriterionFreeEnergy,BBDPolchinski}
and that of \textcite{ChenEldanLocalizationSchemes}
to the present study;
and Louis-Pierre Chaintron
for a valuable remark on Theorem~\ref{thm:ip-contractivity}
and helpful suggestions on the manuscript.

The author is also grateful to Daniel Lacker for discussions
on the Dupuis--Ellis approach to large deviations
\cite{DupuisEllisWeakConvergence} after the manuscript was made public.

\sloppy
\printbibliography
\parbox{\textwidth}{\myauthorinfo}
\end{document}